\documentclass[11pt]{amsart}

\usepackage{amsfonts,epsfig}
\usepackage{latexsym}
\usepackage{amssymb}
\usepackage{amsmath}
\usepackage{amsthm}
\usepackage{graphics}
\usepackage[all]{xy}
\usepackage[T2A]{fontenc}
\usepackage{multirow}
\usepackage{hhline}
\usepackage{array, booktabs, ctable}
\usepackage[backref]{hyperref}
\usepackage{enumerate}

\newtheorem{defn}{Definition}[section]

\newtheorem{corollary}[defn]{Corollary}
\newtheorem{lemma}[defn]{Lemma}

\newtheorem{thm}[defn]{Theorem}
\newtheorem{theorem}[defn]{Theorem}
\newtheorem{cor}[defn]{Corollary}
\newtheorem{prop}[defn]{Proposition}

\theoremstyle{definition}

\newtheorem*{ack}{Acknowledgements}
\newtheorem{remark}[defn]{Remark}

\newtheorem{example}[defn]{Example}

\newcommand{\CC}{\mathbb C}

\newcommand{\Q}{\mathbb Q}
\newcommand{\Z}{\mathbb Z}

\newcommand{\F}{\mathbb F}

\newcommand{\PP}{\mathbb P}

\newcommand{\OK}{\mathcal{O}_K}
\newcommand{\Of}{\mathcal{O}_{K,f}}
\newcommand{\cC}{\mathcal{C}}
\newcommand{\OO}{\mathcal{O}}

\newcommand{\calc}{\mathcal{C}}

\newcommand{\Rats}{\mathbb{Q}}

\newcommand{\Gal}{\operatorname{Gal}}
\newcommand{\Aut}{\operatorname{Aut}}
\newcommand{\GQ}{\Gal(\overline{\Rats}/\Rats)}

\newcommand{\GF}{\Gal(\overline{F}/F)}

\newcommand{\GL}{\operatorname{GL}}
\newcommand{\PGL}{\operatorname{PGL}}

\newcommand{\End}{\operatorname{End}}

\newcommand{\Cl}{\mathrm{Cl}}

\begin{document}



\title[Galois representations of elliptic curves with CM]{Galois representations attached to elliptic curves with complex multiplication}

\author{\'Alvaro Lozano-Robledo}
\address{University of Connecticut, Department of Mathematics, Storrs, CT 06269, USA}
\email{alvaro.lozano-robledo@uconn.edu} 
\urladdr{http://alozano.clas.uconn.edu}




\begin{abstract} The goal of this article is to give an explicit classification of the possible $p$-adic Galois representations that are attached to elliptic curves $E$ with CM defined over $\mathbb{Q}(j(E))$. More precisely, let $K$ be an imaginary quadratic field, and let $\mathcal{O}_{K,f}$ be an order in $K$ of conductor $f\geq 1$. Let $E$ be an elliptic curve with CM by $\mathcal{O}_{K,f}$, such that $E$ is defined by a model over $\mathbb{Q}(j(E))$. Let $p\geq 2$ be a prime, let $G_{\mathbb{Q}(j(E))}$ be the absolute Galois group of $\mathbb{Q}(j(E))$, and let $\rho_{E,p^\infty}\colon G_{\mathbb{Q}(j(E))}\to \operatorname{GL}(2,\mathbb{Z}_p)$ be the Galois representation associated to the Galois action on the Tate module $T_p(E)$. The goal is then to describe, explicitly, the groups of $\operatorname{GL}(2,\mathbb{Z}_p)$ that can occur as images of $\rho_{E,p^\infty}$, up to conjugation, for an arbitrary order $\mathcal{O}_{K,f}$. 
\end{abstract}

\maketitle


\section{Introduction}

Let $F$ be a number field, let $E/F$ be an elliptic curve, let $p$ be a prime, let $T_p(E)=\varprojlim E[p^n]$ be the $p$-adic Tate module attached to $E$, and fix a $\Z_p$-basis for $T_p(E)$. The natural action of the absolute Galois group of $F$, denoted by $G_F = \Gal(\overline{F}/F)$, produces a $p$-adic Galois representation $\rho_{E,p^\infty}\colon \GF \to \Aut(T_p(E))\cong \GL(2,\Z_p)$. Serre's open image theorem (see \cite{serre1}) implies that, if we fix an elliptic curve $E/F$ without complex multiplication, then $\rho_{E,p^\infty}$ is surjective for all but finitely many primes $p$. Much work has been done to classify the possible images when  $F=\Q$ and $\rho_{E,p^\infty}$ is {\it not} surjective. For instance, Rouse and Zureick-Brown have classified the possible $2$-adic images for non-CM curves (see \cite{rouse}), and Sutherland and Zywina have produced a conjectural list of all the possible mod $p$ images (\cite{sutherland}, \cite{zywina}) which is complete if we assume a positive answer to a uniformity question of Serre. In \cite{sutherland-zywina} the authors provide a classification of all the possible $p$-adic representations over $\Q$ that is complete except for a finite set $\mathcal{J}$ of exceptional $j$-invariants (note that $\mathcal{J}$ includes all $j$-invariants in $\Q$ with CM). 

The goal of this article is to describe all the possible images of $p$-adic Galois representations attached to elliptic curves with complex multiplication, as subgroups of $\GL(2,\Z_p)$ defined up to conjugation. More concretely, let $K$ be an imaginary quadratic field, let $\OK$ be the ring of integers of $K$ with discriminant $\Delta_K$, let $f\geq 1$ be an integer, and let $\OO_{K,f}$ be the order of $K$ of conductor $f$. Let $j\colon X(1)\to \PP^1(\CC)$ be the modular $j$-invariant function, and let $j(\CC/\OO_{K,f})$ be the $j$-invariant associated to the order $\OO_{K,f}$ when regarded as a complex lattice. The theory of complex multiplication shows that $j(\CC/\OO_{K,f})$ is an algebraic integer (see Theorem \ref{thm-cmbasics} below). Let $j_{K,f}$ be an arbitrary Galois conjugate of $j(\CC/\OO_{K,f})$. Then, an elliptic curve $E/\Q(j_{K,f})$ with $j(E)=j_{K,f}$ has complex multiplication by $\OO_{K,f}$, i.e., $\End(E)\cong \OO_{K,f}$, and every elliptic curve with CM by $\OO_{K,f}$, and defined over $\Q(j_{K,f})$, arises in this way. It is then the goal of this article to completely describe the possible images of $\rho_{E,p^\infty}\colon \Gal(\overline{\Q(j_{K,f})}/\Q(j_{K,f})) \to \GL(2,\Z_p)$ for any prime $p\geq 2$, and any elliptic curve $E/\Q(j_{K,f})$ with complex multiplication and $j$-invariant $j_{K,f}$. We shall describe the image of $\rho_{E,p^\infty}$ as a subgroup of $\GL(2,\Z_p)$, up to conjugation. When $j_{K,f}\in \Q$, the image of $\rho_{E,p^\infty} \bmod p$ has been determined by Zywina (see \cite{zywina}, \S 1.9). In articles that were produced simultaneously to this one, Bourdon, Clark, and Pollack (\cite{bourdon}), and Bourdon and Clark (\cite{bourdon2}),  have described in great detail the properties of the extension $K(j_{K,f},E[N])/K(j_{K,f})$, where $N\geq 2$ is an integer, and they use this information to prove several applications to the degrees of definition of torsion points and isogenies on CM elliptic curves. In \cite{lombardo}, Lombardo has also recently proved analogous results to those of Bourdon and Clark in the context of abelian varieties of CM type. Our results differ from those of the authors mention above in that we drill down to the level of conjugacy classes of images as subgroups of $\GL(2,\Z_p)$. 

A classification of $p$-adic (or even mod $p$) images attached to elliptic curves can yield many other results as an application, for example to the study of sporadic points on modular curves, abelian division fields of elliptic curves with CM (extending the work of \cite{gonzalez-jimenez-lozano-robledo}), or the classification of isogenies over a fixed field (see for instance the introduction to \cite{rouse}, or other examples such as \cite{lozanoannalen} \cite{gonzalez-jimenez-lozano-robledo}, \cite{chou}, or \cite{daniels}). We will present a number of corollaries of the classification given here in a follow-up paper (see \cite{lozano-applied}). For instance, we will show that there are precisely $28$ possible $2$-adic images for elliptic curves with CM defined over $\Q$, up to conjugation, which completes the classification of Rouse and Zureick-Brown (\cite{rouse}) for $2$-adic images attached to elliptic curves over $\Q$.

Our results are as follows. In our first theorem, we use the basic theory of complex multiplication and class field theory (see \cite{schertz} or \cite{stevenhagen2}) to give an interpretation of the Galois representation associated to the $N$-division field of an elliptic curve with CM. 

\begin{theorem}\label{thm-cmrep-intro}
	Let $E/\Q(j_{K,f})$ be an elliptic curve with CM by $\OO_{K,f}$, let $N\geq 3$, and let $\rho_{E,N}$ be the Galois representation $\Gal(\overline{\Q(j_{K,f})}/\Q(j_{K,f})) \to \Aut(E[N])\cong \GL(2,\Z/N\Z)$.  We define groups of $\GL(2,\Z/N\Z)$ as follows:
	\begin{itemize}
		\item If $\Delta_Kf^2\equiv 0\bmod 4$, or $N$ is odd, let $\delta=\Delta_K f^2/4$, and $\phi=0$.
		\item If $\Delta_Kf^2\equiv 1 \bmod 4$, and $N$ is even, let $\delta=\frac{(\Delta_K-1)}{4}f^2$, let $\phi=f$.
	\end{itemize}
If we define the Cartan subgroup $\cC_{\delta,\phi}(N)$ of $\GL(2,\Z/N\Z)$ by
$$\cC_{\delta,\phi}(N)=\left\{\left(\begin{array}{cc}
a+b\phi & b\\
\delta b & a\\
\end{array}\right): a,b\in\Z/N\Z,\  a^2+ab\phi-\delta b^2 \in (\Z/N\Z)^\times \right\}.$$
with  $\mathcal{N}_{\delta,\phi}(N) = \left\langle \cC_{\delta,\phi}(N),\left(\begin{array}{cc} -1 & 0\\ \phi & 1\\\end{array}\right)\right\rangle$, then there is a $\Z/N\Z$-basis of $E[N]$ such that the image of $\rho_{E,N}$
	is contained in $\mathcal{N}_{\delta,\phi}(N)$.	Moreover, 
	\begin{enumerate}
		\item $\cC_{\delta,\phi}(N)\cong (\Of/N\Of)^\times$  is a subgroup of index $2$ in  $\mathcal{N}_{\delta,\phi}(N)$, and
		\item The index of the image of $\rho_{E,N}$ in $\mathcal{N}_{\delta,\phi}(N)$ coincides with the order of the Galois group $\Gal(K(j_{K,f},E[N])/K(j_{K,f},h(E[N]))$, for a Weber function $h$, and it is a divisor of the order of $\Of^\times/\mathcal{O}_{K,f,N}^\times$, where $\mathcal{O}_{K,f,N}^\times=\{u\in\Of^\times: u\equiv 1 \bmod N\Of\}$.
	\end{enumerate} 
\end{theorem}

Theorem \ref{thm-cmrep-intro} can be used to prove a uniform analogue of Serre's ``open image theorem'' in the CM case, in the sense that one can show that the image of $\Gal(\overline{\Q(j_{K,f})}/\Q(j_{K,f})) \to \varprojlim \Aut(E[N])\cong \GL(2,\widehat{\Z})$ is of finite index as a subgroup of $\mathcal{N}_{\delta,\phi} = \varprojlim \mathcal{N}_{\delta,\phi}(N)$. Such a uniform result can also be found in \cite[Corollary 1.3]{bourdon2} and \cite[Theorem 1.5]{lombardo}. In addition, we use work of Serre and Tate (\cite{serretate}) to prove several surjectivity results for the $p$-adic representation.

\begin{thm}\label{thm-largeimage-intro}
	Let $E/\Q(j_{K,f})$ be an elliptic curve with CM by $\OO_{K,f}$. 	\begin{itemize}
		\item If $\Delta_Kf^2\equiv 0\bmod 4$, let $\delta=\Delta_K f^2/4$, and $\phi=0$.
		\item If $\Delta_Kf^2\equiv 1 \bmod 4$, let $\delta=\frac{(\Delta_K-1)}{4}f^2$, let $\phi=f$.
	\end{itemize}
Then, the following hold:
	\begin{enumerate}
		\item Let $\rho_{E}$ be the Galois representation $\Gal(\overline{\Q(j_{K,f})}/\Q(j_{K,f})) \to \varprojlim \Aut(E[N])\cong \GL(2,\widehat{\Z})$, and let $\mathcal{N}_{\delta,\phi} = \varprojlim \mathcal{N}_{\delta,\phi}(N)$. Then, there is a compatible system of bases of $E[N]$ such that the image of $\rho_E$ is contained in $\mathcal{N}_{\delta,\phi}$, and the index of the image of $\rho_{E}$ in $\mathcal{N}_{\delta,\phi}$ is a divisor of the order $\Of^\times$. In particular, the index is a divisor of $4$ or $6$. Moreover, for every $K$ and $f\geq 1$, and a fixed $N\geq 3$, there is an elliptic curve $E/\Q(j_{K,f})$ such that  the image of $\rho_{E,N}$ is precisely $\mathcal{N}_{\delta,\phi}(N)$.

		\item Let $p>2$ and $j_{K,f}\neq 0$, or $p>3$. Let $G_{E,p^\infty}\subseteq \mathcal{N}_{\delta,\phi}(p^\infty) $ be the image of $\rho_{E,p^\infty}$ with respect to a suitable basis of $T_p(E)$, and let $G_{E,p}\subseteq \mathcal{N}_{\delta,\phi}(p)$ be the image of $\rho_{E,p} \equiv \rho_{E,p^\infty}\bmod p$. Then, $G_{E,p^\infty}$ is the full inverse image of $G_{E,p}$ via the reduction mod $p$ map $\mathcal{N}_{\delta,\phi}(p^\infty)\to \mathcal{N}_{\delta,\phi}(p)$. 
		
		\item Let $p>2$ and let $\chi_{E,p^\infty}=\det(\rho_{E,p^\infty})\colon \Gal(\overline{\Q(j_{K,f})}/\Q(j_{K,f}))\to \Z_p^\times$. Then, $\chi_{E,p^\infty}$ is the $p$-adic cyclotomic character, and the index of the image in $\Z_p^\times$ is $1$ or $2$. Further, the index is $2$ if and only if $p\equiv 1 \bmod 4$ and $\Q(\sqrt{p})\subseteq \Q(j_{K,f})$. In particular, $\chi_{E,p^\infty}$ is surjective onto $\Z_p^\times$ for all but finitely many primes $p$.
		
				\item If $p$ is not a divisor of $2\Delta_K f$, and $j(E)\neq 0$ (or $j(E)=0$ but $p\equiv \pm 1 \bmod 9$), and if $c\in \Gal(\overline{\Q(j_{K,f})}/\Q(j_{K,f}))$ is a fixed complex conjugation, and we fix $\varepsilon \in \{ \pm 1 \}$, then there is a $\Z_p$-basis of $T_p(E)$ such that the image of $\rho_{E,p^\infty}\colon \Gal(\overline{\Q(j_{K,f})}/\Q(j_{K,f})) \to  \GL(2,\Z_p)$ is equal to  $\mathcal{N}_{\delta,\phi}(p^\infty)= \varprojlim \mathcal{N}_{\delta,\phi}(p^n)$ and  $\rho_{E,p^\infty}(c)=c_\varepsilon=\left(\begin{array}{cc} \varepsilon & 0\\ 0 & -\varepsilon\\\end{array}\right)$.
	\end{enumerate}
\end{thm}

Part (1) of Theorem \ref{thm-largeimage-intro}, shows that the image of $\rho_{E}$ in $\mathcal{N}_{\delta,\phi}$ is a divisor of $4$ or $6$. However, the index may never be $1$ for some choices of $j_{K,f}$. The following theorem shows one such example (and it will be shown in our follow-up paper \cite{lozano-applied}).

\begin{thm}\label{thm-j1728-index2-intro}
	Let $E/\Q$ be an elliptic curve with $j(E)=1728$, and choose compatible bases of $E[N]$, for each $N\geq 2$, such that the image of $\rho_{E}$ is contained in $\mathcal{N}_{\delta,\phi}$. Then, the index of the image of $\rho_{E}$ in $\mathcal{N}_{\delta,\phi}$ is $2$ or $4$.
\end{thm}

It is worth noting that part (2) of Theorem \ref{thm-largeimage-intro} says that $\rho_{E,p^\infty}$ is essentially defined at level $p$ for all odd primes, if $j_{K,f}\neq 0$. In contrast, if $E/\Q$ is a curve without CM, then \cite[Corollary 1.3]{rouse} says that $\rho_{E,2^\infty}$ is defined at level $32$, in general. It is not known what is the generic level of definition of $\rho_{E,p^\infty}$ for $p>2$. The cases of $p=3$ (when $j_{K,f}=0$) and $p=2$ are more delicate, and they will be studied in detail in Theorems \ref{thm-badprimes-intro} and \ref{thm-2adictwist-intro}, respectively. Theorem \ref{thm-inclp} is a more detailed version of Theorem \ref{thm-largeimage-intro}, part (2).

In Theorem \ref{thm-largeimage-intro}, part (4), the case of $j=0$ needs to be excluded because a couple of other images are possible that are strictly contained in $\mathcal{N}_{\delta,\phi}(p^\infty)$.

\begin{thm}\label{thm-jzero-goodredn} Let $E/\Q$ be an elliptic curve with $j(E)=0$. Let $K=\Q(\sqrt{-3})$, $f=1$, and $\OO_{K,f} = \OO_K = \Z[(1+\sqrt{-3})/2]$, and $\delta=\Delta_K f^2/4=-3/4$. Let $p>3$ be a prime, and let $c \in \GQ$ be a complex conjugation. Then, for each $\varepsilon \in \{\pm 1 \}$, there is a $\Z_p$-basis of $T_p(E)$ such that the image $G_{E,p^\infty}$ of $\rho_{E,p^\infty}\colon \GQ \to \Aut(T_p(E))\cong \GL(2,\Z_p)$ is contained in $\mathcal{N}_{\delta,0}(p^\infty)$ and $\rho_{E,p^\infty}(c)=c_\varepsilon = \left(\begin{array}{cc} \varepsilon & 0\\ 0 & -\varepsilon\\ \end{array}\right)$. Further, $G_{E,p^\infty}$ is the full inverse image via the natural reduction mod $p$ map  $\mathcal{N}_{\delta,0}(p^\infty)\to \mathcal{N}_{\delta,0}(p)$ of the image $G_{E,p}$ of $\rho_{E,p}\equiv \rho_{E,p^\infty}\bmod p$, and one of the following holds:
		\begin{enumerate}
	\item If $p\equiv \pm 1\bmod 9$, then $G_{E,p}=\mathcal{N}_{\delta,0}(p)$ and $G_{E,p^\infty}= \mathcal{N}_{\delta,0}(p^\infty)$. Moreover, if $p\equiv 1\bmod 9$, then the image is the normalizer of a split Cartan subgroup, and if $p\equiv -1\bmod 9$, then the image is a normalizer of a non-split Cartan subgroup.
	\item If $p\equiv 2$ or $5\bmod 9$, then $G_{E,p}$ is contained in the normalizer of a non-split Cartan subgroup and either $G_{E,p}=\mathcal{N}_{\delta,0}(p)$, or $G_{E,p}=\langle \cC_{\delta,0}(p)^3, c_\varepsilon \rangle $.
	\item If $p\equiv 4$ or $7\bmod 9$, then $G_{E,p}$ is contained in the normalizer of a split Cartan, and either $G_{E,p^\infty}=\mathcal{N}_{\delta,0}(p^\infty)$, or $G_{E,p^\infty}$ is isomorphic to the subgroup $$\left\langle \left\{\left(\begin{array}{cc} a & 0\\ 0 & b\\ \end{array}\right) : a/b \in (\Z_p^\times)^3\right\}, \gamma = \left(\begin{array}{cc} 0 & \varepsilon\\ \varepsilon & 0\\ \end{array}\right)\right\rangle.$$
\end{enumerate}
\end{thm}

Theorems \ref{thm-largeimage-intro} and \ref{thm-jzero-goodredn} take care of the image of $\rho_{E,p^\infty}$ when $p$ does not divide $2f\Delta_K$. Our next result describes the image when $p>2$ divides $f\Delta_K$. 

\begin{thm}\label{thm-badprimes-intro}
	Let $E/\Q(j_{K,f})$ be an elliptic curve with CM by $\Of$ of conductor $f\geq 1$. Let $p$ be an odd prime dividing $f\Delta_K$ (thus, $j\neq 1728$ where $f\Delta_K=-4$), let $G_{E,p^\infty}$ be the image of $\rho_{E,p^\infty}\colon\Gal(\overline{\Q(j_{K,f})}/\Q(j_{K,f}))\to \Aut(T_p(E))$, and let $c\in\Gal(\overline{\Q(j_{K,f})}/\Q(j_{K,f}))$ be a fixed complex conjugation. Then, there is a $\Z_p$-basis of $T_p(E)$ such that $G_{E,p^\infty}\subseteq \mathcal{N}_{\delta,0}(p^\infty)$, with $\delta=\Delta_K f^2/4$ and $\phi=0$, and $\rho_{E,p^\infty}(c)=c_\varepsilon=\left(\begin{array}{cc} \varepsilon & 0\\ 0 & -\varepsilon\\ \end{array}\right)$ for some $\varepsilon \in \{\pm 1 \}$. Moreover, $G_{E,p^\infty}$ is precisely one of the following groups of $\GL(2,\Z_p)$:
	\begin{enumerate}
		\item[(a)] If $j_{K,f}\neq 0,1728$, then either $G_{E,p^\infty}=\mathcal{N}_{\delta,0}(p^\infty)$, or $G_{E,p^\infty}$ is generated by $c_{\varepsilon}$ and the group
		$$\left\{ \left(\begin{array}{cc} a & b\\ \delta b & a\\ \end{array}\right): a\in {\Z_p^\times}^2, b\in  \Z_p\right\}.$$
		\item[(b)] If $j_{K,f}=0$, then $p=3$, and there are twelve possibilities for the image of $\rho_{E,3^\infty}$ (up to conjugation). More concretely, either  $G_{E,3^\infty}=\mathcal{N}_{\delta,0}(3^\infty)$, or $[\mathcal{N}_{\delta,0}(3^\infty):G_{E,3^\infty}]=2,3,$ or $6$, and one of the following holds:
		\begin{enumerate} \item[(i)] If $[\mathcal{N}_{\delta,0}(3^\infty):G_{E,3^\infty}]=2$, then $G_{E,3^\infty}$ is generated by $c_{\varepsilon}$ and 
		$$\left\{ \left(\begin{array}{cc}  a & b\\ -3b/4 & a\\ \end{array}\right): a,b\in  \Z_3,\ a\equiv 1 \bmod 3\right\}\subseteq \GL(2,\Z_3).$$
		
		\item[(ii)]  If $[\mathcal{N}_{\delta,0}(3^\infty):G_{E,3^\infty}]=3$, then $G_{E,3^\infty}$ is generated by $c_{\varepsilon}$ and 
		$$\left\{ \left(\begin{array}{cc} a & b\\ -3b/4 & a\\ \end{array}\right): a\in \Z_3^\times,\ b\equiv 0 \bmod 3 \right\},$$
		$$\text{or }\  \left\langle \left(\begin{array}{cc}  2 & 0\\ 0 & 2\\ \end{array}\right) ,\left(\begin{array}{cc}  1 & 1\\ -3/4 & 1\\ \end{array}\right)\right\rangle,\ \text{ or } \ \left\langle \left(\begin{array}{cc}  2 & 0\\ 0 & 2\\ \end{array}\right) ,\left(\begin{array}{cc}  -5/4 & 1/2\\ -3/8 & -5/4\\ \end{array}\right)\right\rangle\subseteq \GL(2,\Z_3).$$ 
		\item[(iii)]  If $[\mathcal{N}_{\delta,0}(3^\infty):G_{E,3^\infty}]=6$, then $G_{E,3^\infty}$ is generated by $c_{\varepsilon}$ and one of
		$$ \text{or }\ \left\{ \left(\begin{array}{cc}  a & b\\ -3b/4 & a\\ \end{array}\right) : a\equiv 1,\ b\equiv 0 \bmod 3\Z_3 \right\},$$
		$$ \text{or }\ \left\langle \left(\begin{array}{cc}  4 & 0\\ 0 & 4\\ \end{array}\right) ,\left(\begin{array}{cc}  1 & 1\\ -3/4 & 1\\ \end{array}\right)\right\rangle,\ \text{ or } \ \left\langle \left(\begin{array}{cc}  4 & 0\\ 0 & 4\\ \end{array}\right) ,\left(\begin{array}{cc}  -5/4 & 1/2\\ -3/8 & -5/4\\ \end{array}\right)\right\rangle\subseteq \GL(2,\Z_3).$$ 
		\end{enumerate}
	\end{enumerate}
\end{thm}

If $\cC$ is any of the groups in parts (a) or  (b) of Theorem \ref{thm-badprimes-intro}, then the group $\mathcal{N}_1 = \langle \cC,c_1\rangle$ cannot be conjugated to $\mathcal{N}_2 = \langle \cC,c_{-1}\rangle$ as subgroups of $\GL(2,\Z_p)$, in such way that $c_1$ is sent to $c_{-1}$, as we shall show as part of Lemma  \ref{lem-ccfinal}. It is also worth pointing out that the last two groups that appear in parts (b.ii) for $j=0$ surject onto $\mathcal{N}_{\delta,0}(3)$, but they are of index $3$ in $\mathcal{N}_{\delta,0}(3^\infty)$. Similarly, the last two groups that appear in (b.iii) are of index $2$ in $\mathcal{N}_{\delta,0}(3)$, but they are of index $6$ in $\mathcal{N}_{\delta,0}(3^\infty)$. These examples are the CM analogue of those non-CM images described by Elkies (see \cite{elkies}) that are surjective mod 3 but not mod 9. In Example \ref{ex-jzero} we provide examples of elliptic curves with each possible $3$-adic image (including each possible image of complex conjugation $c_\varepsilon$).

It remains to describe the $2$-adic Galois representations attached to elliptic curves with CM. We first describe those attached to $j_{K,f}\neq 0,1728$. From now on, we will use the following notation $H_f=K(j_{K,f})$.

\begin{thm}\label{thm-2adictwist-intro}
	Let $E/\Q(j_{K,f})$ be an elliptic curve with CM by an order $\Of$ in an imaginary quadratic field $K$, with $j_{K,f}\neq 0$ or $1728$. Then, for every $m\geq 1$, we have  $\Gal(H_f(E[2^m])/H_f)\subseteq (\Of/2^m\Of)^\times$. Suppose that $\Gal(H_f(E[2^n])/H_f)\subsetneq (\Of/2^n\Of)^\times$ for some $n\geq 1$, and assume $n$ is the smallest such positive integer. Then, $n=1$, $2$, or $3$, and for all $m\geq 3$, we have $\Gal(H_f(E[2^m])/H_f)\cong (\Of/2^m\Of)^\times/\{\pm 1\}$. Further, there are two possibilities:
	\begin{enumerate} 
		\item If $n\leq 2$, then $\Gal(H_f(E[4])/H_f)\cong (\Of/4\Of)^\times/\{\pm 1\}$ and:
		
		\begin{enumerate} 
			\item $\operatorname{disc}(\Of)=\Delta_Kf^2\equiv 0 \bmod 16$. In particular, we have either
			\begin{itemize}
				\item $\Delta_K\equiv 1 \bmod 4$ and $f\equiv 0\bmod 4$, or
				\item $\Delta_K\equiv 0\bmod 4$ and $f\equiv 0 \bmod 2$.
			\end{itemize} 
			\item The $4$th roots of unity $\mu_4$ are contained in $H_f$, i.e., $\Q(i)\subseteq H_f$.
			\item For each $m\geq 2$, there is a $\Z/2^m\Z$-basis of $E[2^m]$ such that the image of the Galois representation $\rho_{E,2^m}\colon \Gal(\overline{H_f}/H_f)\to \GL(2,\Z/2^m\Z)$ is one of the groups
			$$J_1=\left\langle \left(\begin{array}{cc} 5 & 0\\ 0 & 5\\\end{array}\right),\left(\begin{array}{cc} 1 & 1\\ \delta & 1\\\end{array}\right)\right\rangle \text{ or } J_2=\left\langle \left(\begin{array}{cc} 5 & 0\\ 0 & 5\\\end{array}\right),\left(\begin{array}{cc} -1 & -1\\ -\delta & -1\\\end{array}\right)\right\rangle\subseteq \cC_{\delta,0}(2^m).$$
		\end{enumerate}
		\item If $n=3$, then $\Gal(H_f(E[4])/H_f)\cong (\Of/4\Of)^\times$ and:
		\begin{enumerate}
			\item $\Delta_K\equiv 0 \bmod 8$.
			\item For each $m\geq 3$, there is a $\Z/2^m\Z$-basis of $E[2^m]$ such that the image of the Galois representation $\rho_{E,2^m}\colon \Gal(\overline{H_f}/H_f)\to \GL(2,\Z/2^m\Z)$ is  the group
			$$J_1'=\left\langle \left(\begin{array}{cc} 3 & 0\\ 0 & 3\\\end{array}\right),\left(\begin{array}{cc} 1 & 1\\ \delta & 1\\\end{array}\right)\right\rangle \text{ or } J_2'=\left\langle \left(\begin{array}{cc} 3 & 0\\ 0 & 3\\\end{array}\right),\left(\begin{array}{cc} -1 & -1\\ -\delta & -1\\\end{array}\right)\right\rangle\subseteq \cC_{\delta,0}(2^m).$$
		\end{enumerate}
	\end{enumerate}
	Finally, there is some $\varepsilon \in \{\pm 1 \}$ and $\alpha\in {3,5}$ such that the image of $\rho_{E,2^\infty}$ is a conjugate of 
	$$\left\langle \left(\begin{array}{cc} \varepsilon & 0\\ 0 & -\varepsilon\\\end{array}\right),\left(\begin{array}{cc} \alpha & 0\\ 0  & \alpha\\\end{array}\right),\left(\begin{array}{cc} 1 & 1\\ \delta  & 1\\\end{array}\right)\right\rangle \text{ or } \left\langle \left(\begin{array}{cc} \varepsilon & 0\\ 0 & -\varepsilon\\\end{array}\right),\left(\begin{array}{cc} \alpha & 0\\ 0  & \alpha\\\end{array}\right),\left(\begin{array}{cc} -1 & -1\\ -\delta  & -1\\\end{array}\right)\right\rangle \subseteq \GL(2,\Z_2).$$
\end{thm}

It is also worth pointing out that the groups that appear in (2b) surject onto $\mathcal{N}_{\delta,0}(4)$, but they are of index $2$ in $\mathcal{N}_{\delta,0}(2^\infty)$. These examples are the CM analogue of those non-CM images described by Dokchitser and Dokchitser (see \cite{dokchitser}) that are surjective mod 4 but not mod 8. In Example \ref{ex-2adicimages} we provide examples of elliptic curves with each possible $2$-adic image (including each possible image of complex conjugation $c_\varepsilon$), except for the cases of $j_{K,f}=0,1728$. Our last results describe the $2$-adic images attached to elliptic curves with $j$-invariant $0$ or $1728$.

\begin{thm}\label{thm-j1728-intro}
	Let $E/\Q$ be an elliptic curve with $j(E)=1728$, and let $c\in \GQ$ be a complex conjugation, and $\gamma=\rho_{E,2^\infty}(c)$. Let $G_{E,2^\infty}$ be the image of $\rho_{E,2^\infty}$ and let $G_{E,K,2^\infty}=\rho_{E,2^\infty}(G_{\Q(i)})$. Then, there is a $\Z_2$-basis of $T_2(E)$ such that $G_{E,K,2^\infty}$ is one of the following groups:
	\begin{itemize} 
		\item If $[\cC_{-1,0}(2^\infty):G_{E,K,2^\infty}]=1$, then $G_{E,K,2^\infty}$ is all of $\cC_{-1,0}(2^\infty)$, i.e., 
		$$G_1= \left\{\left(\begin{array}{cc} a & b\\ -b & a\\\end{array}\right)\in \GL(2,\Z_2) : a^2+b^2\not\equiv 0 \bmod 2 \right\}.$$
		
		\item If $[\cC_{-1,0}(2^\infty):G_{E,K,2^\infty}]=2$, then $G_{E,K,2^\infty}$ is one of the following groups:
		$$G_{2,a}=\left\langle -\operatorname{Id}, 3\cdot \operatorname{Id},\left(\begin{array}{cc} 1 & 2\\ -2 & 1\\\end{array}\right) \right\rangle, \text{ or } G_{2,b}=\left\langle -\operatorname{Id}, 3\cdot \operatorname{Id},\left(\begin{array}{cc} 2 & 1\\ -1 & 2\\\end{array}\right) \right\rangle.$$ 
		
		\item If $[\cC_{-1,0}(2^\infty):G_{E,K,2^\infty}]=4$, then $G_{E,K,2^\infty}$ is one of the following groups:
		$$G_{4,a}=\left\langle 5\cdot \operatorname{Id},\left(\begin{array}{cc} 1 & 2\\ -2 & 1\\\end{array}\right) \right\rangle, \text{ or } G_{4,b}=\left\langle 5\cdot \operatorname{Id},\left(\begin{array}{cc} -1 & -2\\ 2 & -1\\\end{array}\right) \right\rangle, \text{ or }$$ 
		$$ G_{4,c}= \left\langle -3\cdot \operatorname{Id},\left(\begin{array}{cc} 2 & -1\\ 1 & 2\\\end{array}\right) \right\rangle, \text{ or } G_{4,d}=\left\langle -3\cdot \operatorname{Id},\left(\begin{array}{cc} -2 & 1\\ -1 & -2\\\end{array}\right) \right\rangle.$$
	\end{itemize}
	Moreover, $G_{E,2^\infty}=\langle \gamma,G_{E,K,2^\infty}\rangle = \langle \gamma', G_{E,K,2^\infty}\rangle$ is generated by one of the groups above, and an element 
	$$\gamma' \in \left\{ c_1=\left(\begin{array}{cc} 1 & 0\\ 0 & -1\\\end{array}\right),c_{-1}=\left(\begin{array}{cc} -1 & 0\\ 0 & 1\\\end{array}\right),c_1'=\left(\begin{array}{cc} 0 & 1\\ 1 & 0\\\end{array}\right),c_{-1}'=\left(\begin{array}{cc} 0 & -1\\ -1 & 0\\\end{array}\right) \right\},$$ such that  
	$\gamma \equiv \gamma' \bmod 4.$
\end{thm}

We provide examples of each kind of $2$-adic image with $j=1728$ in Example \ref{ex-j1728-2adic}.

\begin{thm}\label{thm-jzero-intro}
	Let $E/\Q$ be an elliptic curve with $j(E)=0$, and let $c\in \GQ$ be a complex conjugation. Let $G_{E,2^\infty}$ be the image of $\rho_{E,2^\infty}$ and let $G_{E,K,2^\infty}=\rho_{E,2^\infty}(G_{\Q(\sqrt{-3})})$. Then, there is a $\Z_2$-basis of $T_2(E)$ such that the image $G_{E,2\infty}$ of $\rho_{E,2^\infty}$ is one of the following groups of $\GL(2,\Z_2)$, with $\gamma=\rho_{E,2^\infty}(c)$.
	\begin{itemize}
		\item Either, $[\cC_{-1,1}(2^\infty):G_{E,K,2^\infty}]=3$, and 
		\begin{align*} G_{E,2^\infty} &=\left\langle \gamma', -\operatorname{Id}, \left(\begin{array}{cc} 7 & 4\\ -4 & 3\\\end{array}\right), \left(\begin{array}{cc} 3 & 6\\ -6 & -3\\\end{array}\right)\right\rangle\\
		&=\left\langle \gamma', \left\{\left(\begin{array}{cc} a+b & b\\ -b & a\\\end{array}\right)\in \GL(2,\Z_2) : a\not\equiv 0 \bmod 2,\ b\equiv 0 \bmod 2 \right\}\right\rangle,\end{align*}  
		and $\left\{\left(\begin{array}{cc} a+b & b\\ -b & a\\\end{array}\right): b\equiv 0 \bmod 2 \right\}$ is precisely the set of matrices that correspond to the subgroup of cubes of Cartan elements $\cC_{-1,1}(2^\infty)^3$, which is the unique group of index $3$ in $\cC_{-1,1}(2^\infty)$, 
		\item Or, $[\cC_{-1,1}(2^\infty):G_{E,K,2^\infty}]=1$, and 
		\begin{align*} G_{E,2^\infty} &=\mathcal{N}_{-1,1}(2^\infty)=\left\langle \gamma', -\operatorname{Id}, \left(\begin{array}{cc} 7 & 4\\ -4 & 3\\\end{array}\right), \left(\begin{array}{cc} 2 & 1\\ -1 & 1\\\end{array}\right)\right\rangle\\
		&=\left\langle \gamma', \left\{\left(\begin{array}{cc} a+b & b\\ -b & a\\\end{array}\right)\in \GL(2,\Z_2) : a\not\equiv 0 \text{ or } b\not\equiv 0 \bmod 2 \right\}\right\rangle\end{align*}
	\end{itemize} 
	where $\gamma'\in \left\{ \left(\begin{array}{cc} 0 & 1\\ 1 & 0\\\end{array}\right),\left(\begin{array}{cc} 0 & -1\\ -1 & 0\\\end{array}\right)\right\},$ and $\gamma\equiv \gamma'\bmod 4$.
\end{thm}

We provide examples of each kind of $2$-adic image with $j=0$ in Example \ref{ex-j0-2adic}.

\begin{remark}
	Although in this article we concentrate on elliptic curves $E$ with $j(E)=j_{K,f}$, and defined over the minimal field of definition $\Q(j_{K,f})$, our findings extend to any elliptic curve with CM defined over an arbitrary number field. Indeed, if $G_{E,p^\infty}$ is the image of $\rho_{E,p^\infty} \colon G_{\Q(j_{K,f})}\to \GL(2,\Z_p)$, and $H$ is a subgroup of $G_{E,p^\infty}$ of finite index, then $H$ is open in $G_{E,p^\infty}$, and the inverse image of $H$ via $\rho_{E,p^\infty}$ is also open, hence of the form $\Gal(\overline{\Q(j_{K,f})}/L)$. Thus, if we base-extend $E$ to $L$, then the image of $\rho_{E/L,p^\infty}\colon \Gal(\overline{L}/L)\to \GL(2,\Z_p)$ has image $H$. Conversely, if $E'/L$ is an elliptic curve with CM by $K$ and $j(E)=j_{K,f}$, then $\Q(j_{K,f})\subseteq L$, and if $E/\Q(j_{K,f})$ is an elliptic curve with CM by $\OO_{K,f}$ and $j(E)=j_{K,f}$, then $E'/L$ and $E/L$ are twists of each other, and $\rho_{E'/L,p^\infty}$ can be retrieved as a twist of $\rho_{E/L,p^\infty}$. 
\end{remark}

\subsection{Structure of the paper} The article is organized as follows. In Section \ref{sec-prelims} we recall some basic notions of elliptic curves with CM. In Section \ref{sec-classfieldtheory} we give a class field theory interpretation of the extension $K(j_{K,f},h(E[N]))/K(j_{K,f})$, where $h$ is a fixed Weber function, and in Section \ref{sec-fromhf} we describe the Galois group $\Gal(K(j_{K,f},E[N])/K(j_{K,f}))$. In Section \ref{sec-normalizers} we discuss abstract Cartan subgroups of $\GL(2,R)$ for a ring $R$. In Section \ref{sec-fromq} we describe the Galois group  $\Gal(K(j_{K,f},E[N])/\Q(j_{K,f}))$ and, in particular, we discuss the role of complex conjugation in detail in Section \ref{sec-cc}. In Sections \ref{sec-goodredn} and \ref{sec-badprimes} we discuss, respectively, the image of $\rho_{E,p^\infty}$ for odd primes of good and bad reduction. Finally, in Section \ref{sec-2adic}, we discuss the possible $2$-adic images that arise from elliptic curves with CM.

Theorem \ref{thm-cmrep-intro} is shown in Section \ref{sec-fromq}.  Theorem \ref{thm-largeimage-intro} is shown in Section \ref{sec-proofsparts1and3} (parts (1) and (2)), Proposition \ref{prop-largeimage-intro-part-3} (part (3)), and Theorem  \ref{thm-largeimage-intro-part4} (part (4)). Theorem \ref{thm-j1728-index2-intro} will be shown in a follow-up paper on applications of the classification (\cite{lozano-applied}). Theorem \ref{thm-jzero-goodredn} is shown in Section \ref{sec-proof-of-jzero-goodredn}, while Theorem \ref{thm-badprimes-intro} is put together in Section \ref{sec-proof-badprimes}. Finally, the Theorems \ref{thm-2adictwist-intro}, \ref{thm-j1728-intro}, and \ref{thm-jzero-intro} about $2$-adic images are shown, respectively, in Sections \ref{sec-2adic}, \ref{sec-2adic1728}, and \ref{sec-2adic0}.

\begin{ack}
The author would like to thank Abbey Bourdon, Pete Clark, and Harris Daniels for very helpful conversations and suggestions, over the span of several years. Also, the author would like to thank Drew Sutherland for his help in verifying the images that appear in Examples \ref{ex-jzero} and \ref{ex-2adicimages}. Finally, I  thank Richard Leyland for pointing out a few arguments that needed clarification. I would also like to express my gratitude to the referees for their comments and for pointing out a mistake in the original proof of Cor. \ref{cor-classfieldinterpretofdivfield}.  
\end{ack}

\section{Preliminaries}\label{sec-prelims}

In this section we recall standard facts about elliptic curves with CM, and show a class field theory characterization of the $N$-division field of an elliptic curve with CM. We continue using the notation of the introduction: let $K$ be an imaginary quadratic field, let $\OK$ be the ring of integers of $K$ with discriminant $\Delta_K$, let $f\geq 1$ be an integer, and let $\OO_{K,f}$ be the order of $K$ of conductor $f$.

\begin{thm}[\cite{silverman2}, Ch. 2, Theorems 4.3 and 6.1; \cite{cox}, Theorem 11.1]\label{thm-cmbasics} Let $K$ be an imaginary quadratic field with ring of integers $\OK$, and let $E/\CC$ be an elliptic curve with CM by an order $\Of$ of $\OK$ of conductor $f\geq 1$. Then:
\begin{enumerate}
\item The $j$-invariant of $E$, $j(E)$, is an algebraic integer. 
\item The field $L = K(j(E))$  is the ring class field of the order $\Of$.
\item $[\Q(j(E)):\Q]=[K(j(E)):K]=h(\Of),$ where $h(\Of)$ is the class number of $\Of$.
\item Let $\{E_1,\dots, E_h\}$ be a complete set of representatives of isomorphism classes of elliptic curves over $\CC$ with CM by $\Of$. Then, $\{ j(E_1),\dots, j(E_h)\}$ is a complete set of $\Gal(\overline{\Q}/\Q)$ conjugates of $ j(E)$. 
\end{enumerate}
\end{thm}

In this paper, we denote by $j_{K,f}$ one of the $j$-invariants $j(E_i)$ that appear in part (4) of Theorem \ref{thm-cmbasics}, and $E/\Q(j_{K,f})$ will be an arbitrary elliptic curve with $j$-invariant $j(E)=j_{K,f}$. Our main object of study is the $N$-division field $\Q(j_{K,f},E[N])$ generated by $j_{K,f}$ and the coordinates of $N$-torsion points on $E$. 

\begin{lemma}[\cite{bourdon3}, Lemma 3.15]\label{lem-clark} Let $K$ be an imaginary quadratic field, let $F$ a number field, let $E/F$ be an elliptic curve with CM by $\OO_{K,f}$, and let $N\geq 3$ be an integer. Then, $K\subseteq F(E[N])$. In particular, if $E$ is defined over $\Q(j_{K,f})$, then $K\subseteq K(j_{K,f})\subseteq \Q(j_{K,f},E[N])$ for any $N\geq 3$. 
	\end{lemma} 
	
	While we are mainly interested in the Galois group $\Gal(\Q(j_{K,f},E[N])/\Q(j_{K,f}))$, the subgroup given by $\Gal(\Q(j_{K,f},E[N])/K(j_{k,f}))$, for $N\geq 3$, plays an essential role in the theory (see next section) at least in part due to the fact that all the endomorphisms of $E$ are defined precisely over $K(j_{K,f})$, as the following theorem demonstrates.

\begin{thm}\label{thm-endo} Let $E$ be an elliptic curve with complex multiplication by an order $\OO_{K,f}$ in an imaginary quadratic field $K$, and suppose that $E$ is defined over $\Q(j_{K,f})$. Then, the field of definition of $\operatorname{End}(E)$ is precisely $K(j_{K,f})$. 
\end{thm}
\begin{proof}
	In part (b) of \cite[Ch. II, Theorem 2.2]{silverman2} it is shown that every endomorphism of $E$ is defined over $K(j_{K,f})$, so it suffices to show that not every endomorphism is defined over $\Q(j_{K,f})$. Suppose for a contradiction that every endomorphism of $E$ is defined over $\Q(j(E))$, and fix an isomorphism $[\cdot]_E \colon \OO_{K,f}\to \operatorname{End}(E)$ that is normalized as in \cite[Ch. II, Prop. 1.1]{silverman2}. It follows that if we fix an $\alpha\in\OO_{K,f}$, then $[\sigma(\alpha)]_E = \sigma([\alpha]_E)=[\alpha]_E$ for all  $\sigma\in\Gal(\overline{\Q(j_{K,f})}/\Q(j_{K,f}))$. Hence, $\sigma(\alpha)=\alpha$, and $\alpha\in\Q(j_{K,f})$. This would imply that $K\subseteq \Q(j_{K,f})$, but part (3) of Theorem \ref{thm-cmbasics} shows that $[K(j_{K,f}):\Q(j_{K,f})]=2$. Thus, we have reached a contradiction, and the field of definition of the endomorphisms of $E$ must be precisely $K(j_{K,f})$.
\end{proof}
	
Next, we follow Silverman \cite{silverman2}, Ch. II, p. 134, to define a Weber function.
	
\begin{defn}\label{defn-weber}
Let $E/F$ be an elliptic curve defined over a number field $F$. A finite map $h\colon E\to E/\Aut(E)\cong \PP^1$ also defined over $F$ is called a Weber function for $E/F$.
\end{defn}

It is worth pointing out that $E/\Aut(E)$ is $F$-isomorphic to $\PP^1$. Thus, if we fix an isomorphism $\phi\colon E/\Aut(E)\to \PP^1$ defined over $F$, then the Weber functions for $E/F$ arise from the compositions $h=\psi\circ\phi$ where $\psi$ is an automorphism in $\Aut(\PP^1/F)\cong \PGL_2(F)$ and, conversely, every Weber function $h$ arises in this way.

We can provide an analytic Weber function as follows. Let $\Lambda$ be a lattice in the complex plane such that $E(\CC)\cong \CC/\Lambda$ and fix an isomorphism $f:\CC/\Lambda \to E(\mathbb{C})$ by $z\mapsto (\wp(z,\Lambda),\wp'(z,\Lambda))$. Then, for $P\in E$ and $z\in\CC$ with $f(z)=P$, the function $h$ defined by $h(P):=h(z,\Lambda)$ where $$h(z;\Lambda)=\begin{cases}
\dfrac{g_2(\Lambda)g_3(\Lambda)}{\Delta(\Lambda)} \wp(z,\Lambda) & \text{ if } j(E)\neq 0,1728,\\
\dfrac{g_2(\Lambda)^2}{\Delta(\Lambda)} \wp(z,\Lambda)^2 & \text{ if } j(E)=1728,\\
\dfrac{g_3(\Lambda)}{\Delta(\Lambda)} \wp(z,\Lambda)^3 & \text{ if } j(E) = 0,\\
\end{cases}$$
is a Weber function for $E$ (see \cite{silverman2}, Example 5.5.2, or \cite{schertz}, Section 2.4.4), which takes complex values for $z\in \CC \setminus\Lambda$ and has a pole at each $z\in\Lambda$. We remark that
$$h(\lambda z;\lambda \Lambda) = h(z;\Lambda)$$
for any non-zero $\lambda\in \CC$, since $\wp(\lambda z; \lambda \Lambda)=\lambda^{-2}\wp(z;\Lambda)$, and $g_2$, $g_3$, and $\Delta$ are modular of weight $4$, $6$, and $12$, respectively. Further, $h$ is model independent (i.e., does not depend on the choice of $\Lambda$ or Weierstrass equation for $E$) and, therefore, if $E$ is defined over a number field $F$, we can find a Weber function over $F$. In particular, if $E: y^2=x^3+Ax+B$, with $A,B\in F$, we can choose a Weber function
$$h(P)=\begin{cases}
x(P) & \text{ if } j(E)\neq 0,1728 \ (AB\neq 0),\\
(x(P))^2 & \text{ if } j(E)=1728\ ( B= 0),\\
(x(P))^3 & \text{ if } j(E)=0 \ (A= 0).\\
\end{cases}$$

We finish this section with a result about the size of Cartan subgroups of $\GL(2,\Z/N\Z)$.

\begin{lemma}\label{lem-cartansbgp}
Let $f\geq 1$, let $K$ be an imaginary quadratic field of discriminant $d=\Delta_K$, and let $\Of$ be the order of $K$ of conductor $f$. Let $N\geq 1$, and let $\cC=\cC_{d,f,N}=\Of/N\Of$. Then, the unit group $\cC^\times$ is isomorphic to the group of matrices
$$\left\{\left(\begin{array}{cc}
a+bdf & b \\
\frac{d(1-d)}{4}bf^2 & a\\
\end{array}\right): a,b\in\Z/N\Z,\  a^2+abdf-\frac{d(1-d)}{4}f^2b^2\in (\Z/N\Z)^\times \right\}\subseteq \GL(2,\Z/N\Z).$$
In particular, let $N=p^n$, for some prime $p$ and $n\geq 1$.
\begin{enumerate}
\item If $p\mid df$, then $|\cC_{d,f,p^n}^\times|=p^{2n-1}(p-1)$.
\item If $\gcd(p,f)=1$ and $p$ is split in $K$, then $|\cC_{d,f,p^n}^\times|=p^{2(n-1)}(p-1)^2$.
\item If $\gcd(p,f)=1$ and $p$ is inert in $K$, then $|\cC_{d,f,p^n}^\times|=p^{2(n-1)}(p^2-1)$.
\end{enumerate}
Finally, in all cases $|\cC_{d,f,p^{n+1}}^\times|/|\cC_{d,f,p^n}^\times|=p^2$.
\end{lemma}
\begin{proof}
Let $p$ be a prime and let $n\geq 1$. Let $G$ be the subgroup of $\GL(2,\Z/p^n\Z)$ described in the statement of the lemma. Then, an isomorphism $\cC^\times \cong G$ is given by $\alpha\mapsto M_\alpha$, where $M_{\alpha}$ is the multiplication-by-$\alpha$ matrix with respect to the $\Z$-basis $\{f\tau,1\}$ of $\Of/p^n\Of$, where we have chosen $\tau=\frac{d+\sqrt{d}}{2}$ to cover both cases $d\equiv 0$ or $\equiv 1 \bmod 4$ at once.

Note that if $p$ is odd and $p|f\Delta_K$, or $p=2$ and $d(d-1)f^2/4$ is even, and $\alpha=a+bf\tau$, then $\det(M_\alpha)\equiv a^2\bmod p$, and therefore $\alpha$ is a unit if and only if $a\not\equiv 0 \bmod p$. Thus, $|\cC^\times| = p^n\cdot \varphi(p^n)=p^{2n-1}(p-1)$. If $p=2$, and $d(d-1)f^2/4$ is odd, then $\det(M_\alpha)$ is a unit if and only if $a$ or $b\equiv 1\bmod 2$ and $ab\equiv 0\bmod 2$. Thus, $|\cC^\times| = 2\cdot \varphi(2^n)\cdot 2^{n-1}=2\cdot 2^{2(n-1)} = 2^{2n-1}(2-1)$, so the same formula works for $p=2$.

Otherwise, if $\gcd(p,f)=1$, then $\Of/p^n\Of\cong \OK/p^n\OK$. Thus, $|\cC^\times|=|\OK/p^n\OK|=N_\Q^K(p^n\OK)$, which equals $p^{2n-1}(p-1)$, $p^{2(n-1)}(p^2-1)$, or $p^{2(n-1)}(p-1)^2$, according to whether $p$ is ramified, inert, or split in $K$, respectively.
\end{proof}

\begin{remark}\label{rem-complexconj}
Of course, one can find simpler matrix representations of $\cC_{d,f,N}^\times$, according to whether $\Delta_K\equiv 0$ or $1\bmod 4$. If $\Delta_Kf^2\equiv 0 \bmod 4$ or $N$ is odd, then with respect to the $\Z/N\Z$-basis $\{f\sqrt{\Delta_K}/2,1\}$, the group of units $\cC^\times\cong (\OO_{K,f}/N\OO_{K,f})^\times$ is isomorphic to
$$\cC_{\delta,0}(N)=\left\{\left(\begin{array}{cc}
a & b\\
\delta b & a\\
\end{array}\right): a,b\in\Z/N\Z,\  a^2-\delta b^2\in (\Z/N\Z)^\times \right\}\subseteq \GL(2,\Z/N\Z),$$
where $\delta=\Delta_K f^2/4$, while if $\Delta_K f^2\equiv 1 \bmod 4$ and $N$ is even, then with respect to the basis $\{f(1+\sqrt{\Delta_K})/2,1\}$ we have a representation
$$\cC_{\delta,\phi}(N)=\left\{\left(\begin{array}{cc}
a+bf & b\\
\frac{(\Delta_K-1)}{4}bf^2 & a\\
\end{array}\right): a,b\in\Z/N\Z,\  a^2+abf-\frac{\Delta_K-1}{4}f^2 b^2\in (\Z/N\Z)^\times \right\},$$
where $\delta=(\Delta_K-1)f^2/4$, and $\phi=f$. We will discuss Cartan subgroups in more detail in Section \ref{sec-normalizers}.
\end{remark}

\section{Class field theory interpretation of the $N$-division field}\label{sec-classfieldtheory}

Let $N\geq 3$, and let $E/\Q(j_{K,f})$ be an elliptic curve with CM as above. Lemma \ref{lem-clark} shows that $K\subseteq \Q(j_{K,f},E[N])$, so that $K(j_{K,f},E[N])=\Q(j_{K,f},E[N])$. The goal of this section is to give a class field theory interpretation of the extension $K(j_{K,f},h(E[N]))/K(j_{K,f})$, where $h$ is a fixed Weber function, as in Definition \ref{defn-weber}, which will allow us to describe the structure of the Galois group $\Gal(K(j_{K,f},h(E[N]))/K(j_{K,f}))$. In \cite{bourdon2}, Bourdon and Clark deduce an explicit description of the field $K(j_{K,f},h(E[N]))$ as the compositum of a ray class field and a ring class field (see \cite[Theorem 1.1]{bourdon2}). In \cite{stevenhagen2}, Stevenhagen  gives a description of the extension and the Galois group using an adelic approach and Shimura reciprocity. We shall use a classical class field theory approach (and the work of Schertz, \cite{schertz}) to describe the Galois group in terms of quotients of groups of proper $\OO_{K,f}$-ideals, and to show the isomorphisms we need (Corollaries \ref{cor-classfieldinterpretofdivfield} and \ref{cor-weber}). In the following definitions, we follow the notation of \cite{cox}.

\begin{defn} Let $K$ be an imaginary quadratic field, let $f\geq 1$ be an integer, and let $\Of$ be an order of $K$ with conductor $f$.
	
	\begin{enumerate}
		\item A fractional $\Of$-ideal is a subset of $K$ which is a non-zero finitely generated $\Of$-module. A fractional $\Of$-ideal $\mathfrak{F}$ is said to be proper if $\Of = \{\beta\in K  : \beta\mathfrak{F}\subseteq \mathfrak{F}\}$. {\rm (Note: proper fractional $\Of$-ideals are invertible; see for instance Prop. 7.4 of \cite{cox}.)}
		\item The set of all  proper fractional $\Of$-ideals is denoted by $I(\Of)$. The subgroup of $I(\Of)$ of principal fractional $\Of$-ideals is denoted by $P(\Of)$. The ideal class group of the order $\Of$ is the quotient $\Cl(\Of)=I(\Of)/P(\Of)$.
		\item Let $\mathfrak{A}$ and $\mathfrak{B}$ be integral $\Of$-ideals. We say that $\mathfrak{A}$ is prime to $\mathfrak{B}$ if $\mathfrak{A}+\mathfrak{B}=\Of$. 
		\item Let $\mathfrak{B}$ be a fixed integral ideal of $\Of$. The subgroup of $I(\Of)$ generated by proper ideals of $\Of$ that are prime to  $\mathfrak{B}$ is denoted by $I(\Of,\mathfrak{B})$, and $P(\Of,\mathfrak{B})=P(\Of)\cap I(\Of,\mathfrak{B})$.
		\item Finally, $P_1(\Of,\mathfrak{B})$ is the subgroup of principal fractional $\Of$-ideals of the form $\alpha\Of$, for some $\alpha\in K$ such that $\alpha\equiv 1\ \mathrm{mod}^\times\,\mathfrak{B}$, i.e., there are principal $\Of$-ideals $\alpha_1\Of$, $\alpha_2\Of$ prime to $\mathfrak{B}$, such that $\alpha =\alpha_1/\alpha_2$ and $\alpha_1\equiv \alpha_2\bmod \mathfrak{B}$.  
	\end{enumerate}  
\end{defn}

Next, we cite two results that appear in \cite{schertz}, which are the key ingredients in Corollary \ref{cor-classfieldinterpretofdivfield}.

\begin{thm}[\cite{schertz}, Theorems 6.1.1, and  6.2.3]\label{thm-schertz} Let $E$ be an elliptic curve with CM by an order $\Of$ of an imaginary quadratic field $K$, with $j(E)=j_{K,f}$. Then:
\begin{enumerate}
\item 
 The field $H_f=K(j_{K,f})$ is the ring class field modulo $f$ of $K$. In particular, there is an isomorphism $$\psi\colon \Cl(\Of)\cong I(\Of)/P(\Of) \to \Gal(K(j(E))/K),$$
 such that $\psi$ sends the class of a prime ideal $\wp$ to the Frobenius automorphism associated to $\wp$, with the following additional property. If $\mathfrak{A}\in I(\Of)$, and $[\mathfrak{A}]$ is the class of $\mathfrak{A}$ in $\Cl(\Of)$, then
 $$j(E)^{\psi([\mathfrak{A}])} = j([\mathfrak{A}]\ast E),$$
 where, if $E=E_\Lambda$ has complex lattice $\Lambda$, then $[\mathfrak{A}]\ast E$ is an elliptic curve with lattice $\mathfrak{A}^{-1}\cdot \Lambda$. 
 \item  Let $\mathfrak{A}$ be a proper fractional ideal of $\Of$ such that  $\mathfrak{N}=\mathfrak{A}\cap \Of \neq \Of$. Then, for any integral ideal $\mathfrak{C}\in I(\Of,\mathfrak{N})$, we have that $K(j(E),h(1;\mathfrak{A}\mathfrak{C}^{-1}))$ is the class field $H_{f,\mathfrak{N}}$ of $K$ such that
 $$\frac{I(\Of,\mathfrak{N})}{P_{1}(\Of,\mathfrak{N})}\cong \Gal\left(H_{f,\mathfrak{N}}/K\right),$$
 where the isomorphism sends the class of a prime ideal $\wp$ to the Frobenius automorphism associated to $\wp$.
 \end{enumerate}
\end{thm}

\begin{remark}
	In \cite{schertz}, the terminology ``proper ideal of $\Of$'' means proper {\it fractional} ideal of $\Of$ (see Definition 3.1.1 in \cite{schertz}). 
\end{remark}

Note that if $\mathfrak{A}$ is a proper $\Of$-ideal, and $E=E_\Lambda$ has CM by $\Of$, then an elliptic curve with lattice $\mathfrak{A}^{-1}\cdot \Lambda$ also has CM by $\Of$. Alternatively, since $j([\mathfrak{A}]\ast E)=j(E)^{\psi([\mathfrak{A}])}$ is a conjugate of $j(E)$, it follows from Theorem \ref{thm-cmbasics} that $[\mathfrak{A}]\ast E$ also has CM by $\Of$. As a consequence of Theorem \ref{thm-schertz}, we obtain a description of the field generated by the $N$-torsion points of an elliptic curve.

\begin{cor}\label{cor-classfieldinterpretofdivfield}
Let $E$ be an elliptic curve with CM by $\Of$ and $j(E)=j_{K,f}$, and let $N\geq 2$ be an integer. Then, $H_{f,N}=K(j_{K,f},h(E[N]))$ is the class field $H_{f,N}$ of $K$ such that
 $$\frac{I(\Of,N)}{P_{1}(\Of,N)}\cong \Gal\left(H_{f,N}/K\right),$$ where $I(\Of,N)=I(\Of,N\Of)$ and $P_{1}(\Of,N)=P_{1}(\Of,N\Of)$, and the isomorphism sends the class of a prime ideal $\wp$ to the Frobenius associated to $\wp$. Moreover, the field $H_{f,N}$ does not depend on the choice of curve $E$ with CM by $\Of$.
\end{cor} 

\begin{proof}
Let $E$ be an elliptic curve with complex multiplication by $\Of$, given by a model over $\Q(j_{K,f})$, and such that $E\cong \CC/\Of$, i.e., pick $j(E)=j(\CC/\Of)$ where $\Of$ is regarded as a lattice in $\CC$ (we will show at the end of the proof that the choice of $E$ with CM by $\Of$ does not matter). Let $\alpha\in \Of$ be an arbitrary element such that $\alpha \not\equiv 0 \bmod N\Of$, and let $\mathfrak{A}$ be the fractional ideal $(N/\alpha)=\frac{N}{\alpha}\Of$ of $\Of$. Let $\mathfrak{N}=\mathfrak{A}\cap \Of \neq \Of$. Since $\mathfrak{A}$ is principal (therefore invertible), it is a proper fractional ideal of $\Of$. Then, $\alpha/N \in \frac{1}{N}\Of/\Of \cong E[N]$, and by Theorem \ref{thm-schertz} we have
\begin{align*} H_{f,\mathfrak{N}} & =K(j(E),h(1;\mathfrak{A}))=K\left(j(E),h\left(1;\frac{N}{\alpha}\Of\right)\right)\\
&=K\left(j(E),h\left(\frac{\alpha}{N};\Of\right)\right)\subseteq K(j(E),h(E[N]))
\end{align*}
by the homogeneity of the Weber function $h$, where $H_{f,\mathfrak{N}}$ is the class field defined in Theorem \ref{thm-schertz}. Note that $N\in \mathfrak{N}$, and so $N\Of\subseteq \mathfrak{N}$. Thus, there is an inclusion $H_{f,\mathfrak{N}}\subseteq H_{f,N}$ given by the natural surjection
$$\frac{I(\Of,N)}{P_1(\Of,N)}\to \frac{I(\Of,\mathfrak{N})}{P_1(\Of,\mathfrak{N})},$$
and, therefore, it follows that
$$H_{f,\mathfrak{N}}  =K(j(E),h(1;\mathfrak{A}))\subseteq H_{f,N}.$$ 
On the other hand, when $\alpha=1$ (or more generally, for any $\alpha\in\Of$ relatively prime to $N\Of$), we have $\mathfrak{N}=\mathfrak{A}\cap \Of = N\Of$, and so $K(j(E),h(1;\mathfrak{A})) =  H_{f,N}.$
 Since $K(j(E),h(E[N]))$ is the compositum of all such fields $K\left(j(E),h\left(\frac{\alpha}{N};\Of\right)\right)$ for non-zero classes $\alpha\in\Of/N\Of$, it follows that
$$ H_{f,N}\subseteq K(j(E),h(E[N])) = \prod_{\substack{\alpha\in\Of/N\Of\\ \alpha\not\equiv 0 \bmod N}} K\left(j(E),h\left(\frac{\alpha}{N};\Of\right)\right)\subseteq H_{f,N}.$$
Therefore we have an equality of fields $K(j(E),h(E[N]))= H_{f,N}$, as claimed.

Finally, let $E'/\Q(j(E'))$ be another elliptic curve with CM by $\Of$. By Theorem \ref{thm-cmbasics}, the curves $E$ and $E'$ are Galois conjugates of each other, that is, there is $\sigma\in\Gal(\overline{K}/K)$ such that $j(E')=\sigma(j(E))$ and a model of $E'$ is given by letting $\sigma$ act on the coefficients of a model for $E$. Thus, if $P\in E[N]$, then $\sigma(P)\in E'[N]$, and $h(\sigma(P))=\sigma(h(P))$. Finally, since  $K(j(E),h(E[N]))= H_{f,N}$ is Galois, it follows that $j(E')=\sigma(j(E))$ and $h(\sigma(P))$ are in $H_{f,N}$ for $\sigma\in \Gal(\overline{K}/K)$. Hence $K(j(E),h(E[N])) = K(j(E'),h(E'[N]))$ does not depend on the choice of elliptic curve with CM by $\Of$, as desired.
\end{proof} 

The following proposition describes the Galois group of $H_{f,N}/H_f$ in terms of those of $H_{f,N}/K$ and $H_f/K$. 
	
\begin{prop}\label{prop-propcft} Let $E/\Q(j_{K,f})$ be an elliptic curve with CM by an order $\Of$ of conductor $f\geq 1$ in an imaginary quadratic field $K$, and let $N\geq 2$.  Let $H_f=K(j_{K,f})$. Then, there is a commutative diagram of short exact sequences:
	$$\xymatrix{
		1 \ar@{->}[r] & \dfrac{P(\Of,N)}{P_1(\Of,N)}  \ar@{->}[r] \ar@{->}[d]^\cong & \dfrac{I(\Of,N)}{P_1(\Of,N)}\ar@{->}[r] \ar@{->}[d]^\cong & \dfrac{I(\Of,N)}{P(\Of,N)}\ar@{->}[r] \ar@{->}[d]^\cong & 1\\
		1\ar@{->}[r] & \Gal(H_f(h(E[N]))/H_f)\ar@{->}[r] & \Gal(H_f(h(E[N]))/K)\ar@{->}[r] & \Gal(H_f/K)\ar@{->}[r] & 1
	}$$
	where $P(\Of,N)=P(\Of,N\Of)$.
\end{prop}
\begin{proof} 
By Corollary 7.17, and Prop. 7.19, of \cite{cox} we have isomorphisms  
$$I(\Of,N)/P(\Of,N) \cong I(\Of)/P(\Of)\cong \Cl(\Of).$$ 
By Theorem \ref{thm-schertz}, the principal ideals of $I(\Of)$ act trivially on $j(E)$, and therefore the restriction map of Galois groups $\Gal(H_f(h(E[N]))/K)\to \Gal(H_f/K)$ corresponds to the natural surjection
$$\dfrac{I(\Of,N)}{P_1(\Of,N)}\longrightarrow  \dfrac{I(\Of,N)}{P(\Of,N)}\cong I(\Of)/P(\Of)\cong \Cl(\Of) \cong \Gal(H_f/K).$$
Since the kernel of this map is clearly $P(\Of,N)/P_1(\Of,N)$, it follows that this quotient is isomorphic to $\Gal(H_f(h(E[N]))/H_f)$ as needed.
\end{proof}

The next lemma describes the Galois group of $H_{f,N}/H_f$ in terms of a quotient of units mod $N$.

\begin{lemma}\label{lem-lemcft} There is an exact sequence
$$1 \longrightarrow
 \dfrac{\Of^\times}{\mathcal{O}_{K,f,N}^\times}\longrightarrow\left(\dfrac{\Of}{N\Of}\right)^\times \longrightarrow  \dfrac{P(\Of,N)}{P_1(\Of,N)}\longrightarrow 1$$
where $\mathcal{O}_{K,f,N}^\times=\{u\in\Of^\times: u\equiv 1 \bmod N\Of\}$.
\end{lemma}
\begin{proof}
Let $\mathcal{C}$ be a class in the quotient  $P(\Of,N)/P_1(\Of,N)$. One can find an ideal $(\alpha)$ of $\Of$ in the class of $\mathcal{C}$, for some $\alpha\in \Of$. Since $(\alpha)\in P(\Of,N)$ is an ideal relatively prime to $N\Of$, it follows that $\alpha \bmod N\Of$ is a unit of $\Of/N\Of$. Moreover, if $\beta\in\Of$ is such that $\beta \equiv 1\bmod N\Of$, then $(\alpha\beta)$ is another principal ideal in the class of $\mathcal{C}\in P(\Of,N)/P_1(\Of,N)$. This shows that $\psi:\left(\Of/N\Of\right)^\times \to  P(\Of,N)/P_1(\Of,N)$ is well-defined and surjective. The kernel of $\psi$ is given by the units $\alpha \bmod N\Of$ such that $(\alpha)$ is equivalent to $\Of$ modulo $P_1(\Of,N)$, i.e., if and only if there is $(\beta)\in P_1(\Of,N)$ such that $(\alpha\beta)=\Of$. That is, $\alpha\beta = u \in \Of^\times$, or in other words $\alpha\equiv u \bmod N\Of$ for some $u\in \Of^\times$. Thus, the kernel of $\psi$ is the image of $\Of^\times$ in $(\Of/N\Of)^\times$, as claimed.
\end{proof}

Finally, we can put Proposition \ref{prop-propcft} and Corollary \ref{lem-lemcft} together to describe the Galois group of the extension $K(j_{K,f},h(E[N]))/K(j_{K,f})$ in terms of units mod $N$.

\begin{cor}\label{cor-weber}
Let $E/\Q(j_{K,f})$ be an elliptic curve with CM by an order $\Of$ of conductor $f\geq 1$ in an imaginary quadratic field $K$, and let $N\geq 2$.  Let $H_f=K(j_{K,f})$. Then,
$$\Gal(H_f(h(E[N]))/H_f)\cong \left(\dfrac{\Of}{N\Of}\right)^\times \Big/\  \dfrac{\Of^\times}{\mathcal{O}_{K,f,N}^\times}.$$
\end{cor}

In the following lemma we show that the options for the group $\OO_{K,f,N}^\times$ are rather limited.

\begin{lemma}\label{lem-unitsmod}
	Let $\Of$ be an order of an imaginary quadratic field $K$, let $N\geq 2$, and let $\mathcal{O}_{K,f,N}^\times=\{u\in\Of^\times: u\equiv 1 \bmod N\Of\}$. Then, $\mathcal{O}_{K,f,N}^\times=\{1\}$ is trivial, unless $N=2$, in which case  $\OO_{K,f,N}^\times = \{\pm 1\}$.
\end{lemma}
\begin{proof}
	If $K=\Q(\sqrt{\Delta_K})$ is quadratic imaginary, with $\OO_K=\Z[\tau]$, then $\Of$ is of the form $\Z[f\tau]$. In particular, 
	$$\Of^\times =\begin{cases}
	\{\pm 1, \pm i  \} & \text{ if } \Delta_K=-4, f=1,\ j_{K,f}=1728, \text{ where } i^2+1=0,\\
	\{\pm 1, \pm \rho, \pm \rho^2  \} & \text{ if } \Delta_K=-3,\ f=1,\ j_{K,f}=0, \text{ where } \rho^2+\rho+1=0,\\
	\{\pm 1 \} & \text{ otherwise}.
	\end{cases}$$
	Of these, the only unit $u\in \Of^\times$ such that $u\equiv 1 \bmod N\Of$ is $u=-1$, and only when $N=2$.
\end{proof}

Now that we have a description of the Galois group of the extension $H_f(h(E[N]))/H_f$, we move to describe $\Gal(H_f(E[N])/H_f)$ in the next section, and then we will describe the Galois group of $H_f(E[N])/\Q(j_{K,f})$ in Section \ref{sec-fromq}.

\section{The Galois group $\Gal(H_f(E[N])/H_f)$}\label{sec-fromhf}

In the previous section we have described $\Gal(H_f(h(E[N]))/H_f)$, for a Weber function $h$, and in this section we will describe $\Gal(H_f(E[N])/H_f)$. We will rephrase our results in terms of the Galois representation associated to the natural action of the absolute Galois group of $\Q(j_{K,f})$ on the $N$-torsion of $E$,
$$\rho_{E,N}\colon \Gal\left(\overline{\Q(j_{K,f})}/\Q(j_{K,f})\right)\to \Aut(E[N]).$$  
Our first step is the following result about the Galois group of $H_f(E[N])/H_f$ coming from the basic theory of complex multiplication.

\begin{thm}\label{thm-incl} Let $E/\Q(j_{K,f})$ be an elliptic curve with CM by an order $\Of$ of an imaginary quadratic field $K$, let $H_f = K(j_{K,f})$, and let $N\geq 2$. Then, the restriction of the Galois representation $\rho_{E,N}$ to $\Gal(\overline{H_f}/H_f)$ induces an injection
$$\Gal(H_f(E[N])/H_f)\hookrightarrow \Aut_{\Of/N\Of}(E[N])\cong  \left(\dfrac{\Of}{N\Of}\right)^\times.$$
Moreover 
$$\Gal(H_f(E[N])/H_f(h(E[N]))\hookrightarrow \dfrac{\Of^\times}{\mathcal{O}_{K,f,N}^\times}\hookrightarrow \Of^\times\cong \Aut(E),$$
and the index of the image of $\rho_{E,N}$ in $(\Of/N\Of)^\times$ is the index of   $\Gal(H_f(E[N])/H_f(h(E[N])))$ as a subgroup of $\Of^\times/\mathcal{O}_{K,f,N}^\times$ which, in turn, divides the order of $\Aut(E)$.
\end{thm}
\begin{proof} The $N$-torsion subgroup $E[N]$ is a free $\Of/N\Of$-module of rank $1$ (see \cite[Lemma 1]{parish}). Since the endomorphisms of $E$ are defined over $H_f$ by Theorem \ref{thm-endo}, the elements of $\Gal(H_f(E[N])/H_f)$ and $\End(E)$ commute in their action on $E[N]$. It follows that we have an  embedding $$\Gal(H_f(E[N])/H_f)\hookrightarrow \Aut_{\Of/N\Of}(E[N])\cong  (\Of/N\Of)^\times$$ as in \cite{silverman2}, Ch. 2, Theorem 2.3. For $\sigma\in \Gal(H_f(E[N])/H_f)$, let $\psi_\sigma \in \Aut_{\Of/N\Of}(E[N])$ be the corresponding automorphism. Then, $\sigma \in \Gal(H_f(E[N])/H_f(h(E[N])))$ if and only if $\sigma(h(P))=h(P)$ for all $P\in E[N]$ or, equivalently, $h(\psi_\sigma(P))=h(P)$ for all $P\in E[N]$. Since $h:E\to E/\Aut(E)$ is a Weber function, this is equivalent to $\psi_\sigma(P)=\phi(P)$ for some $\phi\in \Aut(E)$ and all $P\in E[N]$. In other words, $\psi_\sigma = \phi \in \Aut(E)$ when restricted to $E[N]$. Thus,  we have shown that $\sigma \in \Gal(H_f(E[N])/H_f(h(E[N])))$ if and only if $\psi_\sigma$ is induced from an element of $\Aut(E)\cong \Of^\times$, and therefore $\sigma$ maps to an element of the image of $\Of^\times$ in $(\Of/N\Of)^\times$, which is $\Of^\times/\mathcal{O}_{K,f,N}^\times$.

Finally, since $\Gal(H_f(h(E[N]))/H_f)\cong (\Of/N\Of)^\times/(\Of^\times/\mathcal{O}_{K,f,N}^\times)$ by Corollary \ref{cor-weber}, and we have an injection of $\Gal(H_f(E[N])/H_f)$ into $(\Of/N\Of)^\times$, it follows that the index of the image of $\rho_{E,N}$ in $(\Of/N\Of)^\times$ is precisely the index of   $\Gal(H_f(E[N])/H_f(h(E[N])))$ as a subgroup of $(\Of^\times/\mathcal{O}_{K,f,N}^\times)$. This concludes the proof of the theorem.
\end{proof}

\begin{corollary}\label{cor-missesrootofunity}
	With the notation of Theorem \ref{thm-incl}, suppose that the image of $\Gal(H_f(E[N])/H_f)$ is a subgroup $H$ of index $d$ in $(\Of/N\Of)^\times$, for some $N>2$. Then, 
	\begin{enumerate}
		\item Suppose $j_{K,f}\neq 1728$. Then, $d\equiv 0 \bmod 2$ if and only if $-1\bmod N\Of\not\in H$. 
		\item Suppose $j_{K,f}=1728$. Then, $d\equiv 0 \bmod 2$ if and only if  $\zeta_4=i\bmod N\Of\not\in H$, and $d\equiv 0\bmod 4$ if and only if $-1\bmod N\Of \not\in H$.
		\item $d\equiv 0 \bmod 3$  if and only if $j_{K,f}=0$, and $\zeta_3\bmod N\Of\not\in H$.
	\end{enumerate}
	\end{corollary} 
\begin{proof}
	By Theorem \ref{thm-incl}, the index of the image of $\rho_{E,N}$ in $(\Of/N\Of)^\times$ is the index of the Galois group  $\Gal(H_f(E[N])/H_f(h(E[N]))$ as a subgroup of $\Of^\times/\mathcal{O}_{K,f,N}^\times\cong \Of^\times$ (where the last isomorphism comes from Lemma \ref{lem-unitsmod} since $N>2$). Moreover, in the course of the proof of Theorem \ref{thm-incl}, we have shown that $\sigma \in \Gal(H_f(E[N])/H_f(h(E[N])))$ if and only if $\psi_\sigma$ is induced from an element of $\Aut(E)\cong \Of^\times$, and therefore $\sigma$ maps to an element of the image of $\Of^\times$ in $(\Of/N\Of)^\times$, which is $\Of^\times/\mathcal{O}_{K,f,N}^\times$.
	
	The unique subgroup $H\subseteq \Of^\times$ of index $2$ is $\{1\}$, $\{\pm 1\}$, or $\{1,\zeta_3,\zeta_{3}^2 \}$ according to whether $j_{K,f}\neq 0,1728$, $j_{K,f}=1728$, or $j_{K,f}=0$, respectively. Claim (1) follows, as well as the first part of (2). If $j_{K,f}=1728$, then $H$ has index $4$ in $\Of^\times$ if and only if $H$ is trivial, if and only if $-1\bmod N\Of \not\in H$. This shows the second part of (2).
	
	Finally, $3|d$ can only happen for $j_{K,f}=0$, and the  subgroups of $\Of^\times$ of index divisible $3$ are $\{\pm 1\}$ and $\{1\}$ in that case. Thus, $d$ is divisible by $3$ if and only if $j_{K,f}=0$ and $\zeta_3\bmod N\Of\not\in H$, and this proves (3).
\end{proof}

The following result shows conditions that imply that  $\rho_{E,p^\infty}\colon\Gal(\overline{H_f}/H_f)\to \GL(2,\Z_p)$ is determined modulo $p$ for $p\geq 3$ if $j_{K,f}\neq 0$.

\begin{thm}\label{thm-inclp} Let $E/\Q(j_{K,f})$ be an elliptic curve with CM by an order $\Of$ of an imaginary quadratic field $K$, let $H_f = K(j_{K,f})$, let $p\geq 2$ be a prime, and let $h$ be a Weber function. Let $t=1$ if $p>2$ and $t=2$ if $p=2$. Then:
\begin{enumerate}
	\item The index $[H_f(E[p^{n+1}]):H_f(E[p^n])]$ is a divisor of $p^2$, for all $n\geq 1$.
	\item For each $n\geq t$, we have $[H_f(h(E[p^{n+1}])):H_f(h(E[p^n]))] =p^2$.
	\item If $[H_f(E[p^n]):H_f(h(E[p^{n}]))]=1$ for some $n\geq t$, then $[H_f(E[p^m]):H_f(h(E[p^{m}]))]=1$ for all $m\geq n$.
	\item If $[H_f(E[p^n]):H_f(h(E[p^{n}]))]=\# \Of^\times$, then $[H_f(E[p^m]):H_f(h(E[p^{m}]))]=\# \Of^\times$ for all $t\leq m\leq n$.
	\item Suppose one of the following holds:
	\begin{itemize}
		\item $[H_f(E[p^n]):H_f(h(E[p^{n}]))]$ is relatively prime to $p$, for some $n\geq t$.
		\item $j_{K,f}\neq 0$, and $p>2$.
		\item $j_{K,f}=0$ and $p>3$.
	\end{itemize}
	Then, the image of the group $\Gal(H_f(E[p^{n+1}])/H_f)$ in $(\Of/p^{n+1}\Of)^\times$ is the full inverse image of the image of $\Gal(H_f(E[p^{n}])/H_f)$ in $(\Of/p^{n}\Of)^\times$ under the natural reduction map modulo $p^n$. In particular, $[H_f(E[p^{n+1}]):H_f(E[p^n])] =p^2$.
\end{enumerate}	
\end{thm}
\begin{proof}
Let $p\geq 2$ be a fixed prime, and let $G_n$ be the image of $\Gal(H_f(E[p^{n}])/H_f)$ in $(\Of/p^{n}\Of)^\times$.  In this proof, we will argue in terms of the following diagram:
	$$\xymatrix{
	& H_f(E[p^{n+1}]) & \\
	H_f(E[p^n]) \ar@{-}[ur]^a&  & H_f(h(E[p^{n+1}])) \ar@{-}[ul]_b \\
	& H_f(h(E[p^n])) \ar@{-}[ur]_d \ar@{-}[ul]^c &  \\
H_f \ar@{-}[ur] & & }$$
where $a,b,c,d$ are the indices of the corresponding extensions. 
\begin{enumerate}
	\item The group $G_{n+1}$ surjects onto $G_n$ (by restriction) and the kernel of the reduction map $(\Of/p^{n+1}\Of)^\times \to (\Of/p^{n}\Of)^\times$ is of size $p^2$,  by Lemma \ref{lem-cartansbgp}. Hence, $|G_{n+1}|$ is a divisor of $|G_n|\cdot p^2$, and $a=|G_{n+1}/G_n|$ is a divisor of $p^2$, as claimed.
	\item  Let $n\geq t$. By Corollary \ref{cor-weber}, we have an isomorphism $$\Gal(H_f(h(E[p^n]))/H_f)\cong (\Of/p^n\Of)^\times/(\Of^\times/\mathcal{O}_{K,f,p^n}^\times),$$ and since $\mathcal{O}_{K,f,p^n}^\times = \mathcal{O}_{K,f,p^{n+1}}^\times=\{1\}$ by Lemma \ref{lem-unitsmod}, it follows that $d=p^2$ where we have also used the fact that $|(\Of/p^{n+1}\Of)^\times|/|(\Of/p^n\Of)^\times| = p^2$, which was shown in Lemma \ref{lem-cartansbgp}. This shows (2).
	\item Suppose that $c=1$ for some $n\geq t$. By (2), we have $d=p^2$. Since $ac=bd$, we must have $a=p^2$ as well, and therefore $b=1$. Now we can use induction to show that $[H_f(E[p^m]):H_f(h(E[p^{m}]))]=1$ for all $m\geq n$.
	\item Suppose $n\geq t$, and $b=\# \Of^\times$. Then, $ac=bd=\# \Of^\times\cdot  p^2$, and $a$ and $c$ are, respectively, divisors of $p^2$ and $\# \Of^\times$. Hence, we must have $a=p^2$ and $c=\#\Of^\times$.
	\item Suppose one of the hypotheses hold. First, we note that under any of these hypotheses we have $d=p^2$. Then, $[H_f(E[p^n]):H_f(h(E[p^{n}]))]$ and $p^2$ are relatively prime, and therefore, $c$ and $d$ are relatively prime. Since $ac=bd$, we must have $a=d$ and $b=c$. In particular, $a=p^2$. Therefore we must have an equality $|G_{n+1}|=|G_n|\cdot p^2$, which implies that $G_{n+1}$ is indeed the full inverse image of $G_n$ under reduction modulo $p^n$.
\end{enumerate} 
This concludes the proof of the theorem.
\end{proof}

\begin{example} 
The conclusion of Theorem \ref{thm-inclp}, part (5), does not hold in general  when $p=2$ and $n=1$. For instance, let $E/\Q : y^2=x^3+x$, with $j(E)=1728$ and CM by $\Z[i]\subseteq K=\Q(i)=H_f$. Here, $\Q(E[2])=\Q(i)=H_f$, so $H_f(E[2])/H_f$ is trivial. Moveover, $\Q(E[4])=\Q(\zeta_8)$, and so $H_f(E[4])=H_f(\sqrt{2})$ is quadratic over $H_f=\Q(i)$. Thus, $[H_f(E[4]):H_f(E[2])]=2$. The reason for the discrepancy with the results of Theorem \ref{thm-inclp} is the fact that $\OO_{K,f,4}$ is trivial, but $\OO_{K,f,2}=\{\pm  1\}$ is not. 
\end{example}

Below, we have adapted an argument in \cite[Theorem 2.11]{bourdon2} to show that there are elliptic curves with CM by $\Of$ such that $\Gal(H_f(E[N])/H_f)$ is as large as possible, i.e., isomorphic to $(\Of/N\Of)^\times$. The proof is similar, except that here we make sure that $E$ is defined over $\Q(j_{K,f})$.

\begin{thm}\label{thm-twistforfullimage}
Let $\Of$ be an order of conductor $f\geq 1$, contained in an imaginary quadratic field $K$, and let $N\geq 3$ be an integer. Then, there exists an elliptic curve $E/\Q(j_{K,f})$ with CM by $\Of$, such that $\Gal(H_f(E[N])/H_f) \cong (\Of/N\Of)^\times$, where $H_f=K(j_{K,f})$. 
\end{thm}
\begin{proof}
Let $E/\Q(j_{K,f})$ be an arbitrary elliptic curve with CM by the order $\Of$. By Theorem \ref{thm-incl}, we have an injection
$$\Gal(H_f(E[N])/H_f)\hookrightarrow \Aut_{\Of/N\Of}(E[N])\cong  \left(\dfrac{\Of}{N\Of}\right)^\times,$$
and by Corollary \ref{cor-weber}, we have 
$$\Gal(H_f(h(E[N]))/H_f)\cong \left(\dfrac{\Of}{N\Of}\right)^\times \Big/ \dfrac{\Of^\times}{\mathcal{O}_{K,f,N}^\times}.$$
Thus, it suffices to find a curve $E'/\Q(j_{K,f})$ with CM by $\Of$ such that $$\Gal(H_f(E'[N])/H_f(h(E'[N])))\cong \dfrac{\Of^\times}{\mathcal{O}_{K,f,N}^\times} \cong \begin{cases}
\Of^\times/\{\pm 1\} & \text{ if } N=2,\\
\Of^\times & \text{ otherwise},
\end{cases}$$
where the last isomorphism comes from Lemma \ref{lem-unitsmod}. In all cases, $\Of^\times/\OO_{K,f,N}^\times \cong \mu_n$, for $n=n(f,N)=1,2,3,4,$ or $6$.

Suppose that $\Gal(H_f(E[N])/H_f(h(E[N])))\cong \mu_m$ with $m< n$ and $m$ divides $n$. Let $\chi$ be a character $\Gal(\overline{H_f}/H_f)\to \mu_n$ such that the fixed field of $\overline{H_f}$ by the kernel  of $\chi$, denoted by $F_\chi = \overline{H_f}^{\ker(\chi)}$, is an extension of $H_f$ of degree $n$ that is disjoint from $H_f(E[N])$, i.e., $F_\chi \cap H_f(E[N])=H_f$. Note that there are infinitely many such characters: indeed, consider for instance the characters given by $\chi_k(\sigma) = \sigma(\sqrt[n]{k})$ for each square-free integer $k> 1$, with fixed field $F_{\chi_k}=H_f(\sqrt[n]{k})$ (note that $\mu_n \subseteq K\subseteq H_f$ so $\chi_k$ is a character and not just a $1$-cocycle in this case). Since $H_f(E[N])$ is a finite extension, there are infinitely many choices of $k$ such that $F_{\chi_k}$ and $H_f(E[N])$ are disjoint (for example, take $k=p$, where $p$ is an unramified prime in $H_f(E[N])$).

Fix one such character $\chi=\chi_k$, consider $E$ as defined over $H_f$ for a moment, and consider $E^\chi$, the twist of $E$ by $\chi$. A priori, $E^\chi$ is also defined over $H_f$, but note that $E^\chi$ has a model defined over the field of definition of $E$, that is, $\Q(j_{K,f})$, by \cite[Ch. ~X, Prop. ~5.4]{silverman} because $\chi=\chi_k$ with $k\in\Z$. Moreover, $j_{K,f}=j(E)=j(E^\chi)$ because $E$ and $E^\chi$ are isomorphic over $F_{\chi}$.  Furthermore, the Galois representation $\rho_{E^\chi,N}\colon \Gal(\overline{H_f}/H_f) \to (\Of/N\Of)^\times$ is given by a twist of $\rho_{E,N}$ by $\chi$, i.e., 
$$\rho_{E^\chi,N} = \rho_{E,N}\otimes \chi.$$
Now, let $G_{E,N} = \Gal(\overline{H_f}/H_f(E[N]))$ and consider its image via $\rho_{E^\chi,N}$:
$$\rho_{E^\chi,N}(G_{E,N}) = (\rho_{E,N}\otimes \chi)(G_{E,N}).$$
Since $G_{E,N}\subseteq \ker(\rho_{E,N})$, it follows that $\rho_{E,N}(G_{E,N})$ is trivial. Moreover, since $F_\chi \cap H_f(E[N])=H_f$, it follows that $\chi(G_{E,N})=\chi(\Gal(\overline{H_f}/H_f)\cong \mu_n$. Thus, $\rho_{E^\chi,N}(G_{E,N})\cong \mu_n$. This implies that the degree of the extension $H_f(E^\chi[N])/(H_f(E^\chi[N])\cap H_f(E[N]))$ is $n$. By Corollary \ref{cor-classfieldinterpretofdivfield}, we have $H_f(h(E^\chi[N]))=H_f(h(E[N])) \subseteq H_f(E^\chi[N])\cap H_f(E[N]).$ Hence:
$$n=[H_f(E^\chi[N]): H_f(E^\chi[N])\cap H_f(E[N]) ] \leq [H_f(E^\chi[N]): H_f(h(E^\chi[N]))] \leq n,$$
where we have used the fact that, for any elliptic curve $E'$ with CM by $\Of$ we have $$\Gal(H_f(E'[N])/H_f(h(E'[N])))\hookrightarrow \Of^\times/\OO_{K,f,N}^\times \cong \mu_n$$ by Theorem \ref{thm-incl}. We conclude that $[H_f(E^\chi[N]): H_f(h(E^\chi[N]))] = n$, which shows the isomorphism $\Gal(H_f(E^\chi[N])/H_f) \cong (\Of/N\Of)^\times$, as desired.
\end{proof} 

We conclude this section with a collection of  results about the image of $\rho_{E,NM}$ compared to the image of $\rho_{E,N}$ and $\rho_{E,M}$, when $\gcd(N,M)=1$. 

\begin{theorem}
	Let $E/\Q(j_{K,f})$ be an elliptic curve with CM by an order $\Of$ of an imaginary quadratic field $K$, let $H_f = K(j_{K,f})$, and let $M>N\geq 2$ be relatively prime integers. Then, there is an injection 
	$$\Gal(H_f(E[MN])/H_f)\hookrightarrow\Gal(H_f(E[M])/H_f)\times \Gal(H_f(E[N])/H_f)$$
	and the cokernel embeds into  $\Of^\times/\mathcal{O}_{K,f,MN}^\times$. Further, if we put $G_N=\Gal(H_f(E[N])/H_f)$, $Gh_N= \Gal(H_f(E[N])/H_f(h(E[N])))$, and $U_{f,N}=(\Of/N\Of)^\times$, then
	\begin{enumerate}
		\item  The index of $G_{MN}$ in $G_M\times G_N$ is at most $|\Of^\times/\mathcal{O}_{K,f,MN}^\times|$.
		\item If $G_{MN}\cong U_{f,MN}$, then  $G_M\cong U_{f,M}$ and $G_N\cong U_{f,N}$. In particular, $G_{MN}\cong G_{M}\times G_N$.
		\item There is an injection $Gh_{MN}\hookrightarrow Gh_M\times Gh_N$. In particular, if $Gh_M=Gh_N=\{1\}$, then $Gh_{MN}$ would be trivial as well, and this can only happen if $N=2$. In this case, $G_{MN}\cong G_{M}\times G_N$ also.
		\item Suppose $N>2$,  and $Gh_M$ or $Gh_N$ is trivial. Then, $Gh_{MN}$ is trivial  and $G_{MN}\cong G_M\times G_N$.
	\end{enumerate}
\end{theorem}
\begin{proof}
	Let $\gcd(M,N)=1$. The injection $G_{MN}\hookrightarrow G_{M}\times G_N$ is clear, since $H_f(E[MN])$ is the compositum of $H_f(E[M])$ and $H_f(E[N])$. Also, we have $Gh_{MN}\hookrightarrow Gh_M\times Gh_N$ because if $\sigma\in Gh_{MN}$ and we restrict it to $H_f(E[M])$, since $\sigma$ fixes $H_f(h(E[MN]))\supseteq H_f(h(E[M]))$, then $\sigma|_{H_f(E[M])}\in Gh_M$. Then,
	\begin{enumerate}
		\item The smallest $G_{MN}$ can be is $|U_{f,MN}/(\Of^\times/\mathcal{O}_{K,f,MN}^\times)|$ while the right hand side is as large as $U_{f,M}\times U_{f,N}\cong U_{f,MN}$. Hence, the index of $G_{MN}$ in $G_M\times G_N$ is at most $|\Of^\times/\mathcal{O}_{K,f,MN}^\times|$.
		\item If $G_{MN}\cong U_{f,MN}$, then the size of $G_{M}\times G_N$ must be $U_{f,MN}\cong U_{f,M}\times U_{f,N}$ and so $G_M\cong U_{f,M}$ and $G_N\cong U_{f,N}$. In particular, $G_{MN}\cong G_{M}\times G_N$.
		\item The injection $Gh_{MN}\hookrightarrow Gh_M\times Gh_N$ implies that if $Gh_M=Gh_N=\{1\}$, then $Gh_{MN}$ would be trivial as well. Further, if $Gh_M=Gh_N=\{1\}=Gh_{MN}$, then $$|U_{f,MN}|/|\Of^\times/\mathcal{O}_{K,f,MN}^\times|=|U_{f,MN}|/(|\Of^\times/\mathcal{O}_{K,f,M}^\times|\cdot |\Of^\times/\mathcal{O}_{K,f,N}^\times|)$$
		and so
		$$|\Of^\times|\cdot |\mathcal{O}_{K,f,MN}^\times| = |\mathcal{O}_{K,f,M}^\times|\cdot |\mathcal{O}_{K,f,N}^\times|.$$
		If $M,N\neq 2$, then $|\mathcal{O}_{K,f,MN}^\times| = |\mathcal{O}_{K,f,M}^\times|= |\mathcal{O}_{K,f,N}^\times|=1$, and this implies $|\Of^\times|=1$, which is impossible. If $N=2$, then $M\neq 2$ (because we assumed $\gcd(M,N)=1)$, and this implies $|\Of^\times|=2$, which is possible. In this case we would have $|G_{MN}|=|U_{f,MN}|/2$ and $|G_M\times G_N| = |U_{f,MN}|/(2/1\cdot 2/2)$ and so we must have $G_{MN}\cong G_M\times G_N$.
		\item Suppose $N>2$. By Lemma \ref{lem-unitsmod}, we have $\OO_{f,M}^\times =\OO_{K,f,N}^\times =\OO_{f,MN}^\times =\{1 \}$. Without loss of generality, let us assume that $Gh_N$ is trivial. Then, $|G_N|=|U_{f,N}/\Of^\times|$ and so $|G_{MN}|$ is of size at most $|U_{f,MN}/\Of^\times|$. Since that is the least $|G_{MN}|$ can be, we must have an equality, and therefore $Gh_{MN}$ must be trivial, and $G_{MN}\cong G_M\times G_N$.
	\end{enumerate}
\end{proof} 

\section{Interlude: Normalizers of Cartan Subgroups}\label{sec-normalizers}

In this section we discuss properties of abstract Cartan subgroups of matrices, and their normalizers. Let $n\geq 1$ be an integer, and let $p$ be a prime number. For a commutative ring $R$, and elements $\delta,\phi\in R$, we define 
$$\cC_{\delta,\phi}(R)=\left\{\left(\begin{array}{cc}
a+b\phi & b\\
\delta b & a\\
\end{array}\right): a,b\in R,\  a^2+ab\phi-\delta b^2\in R^\times \right\}.$$
Then, one can show that $\cC_{\delta,\phi}(R)$ is a subgroup of $\GL(2,R)$, for any choices of $\delta$ and $\phi$ in $R$. Moreover, the equation
$$ \left(\begin{array}{cc}
1 & 0\\
\phi/2 & 1\\
\end{array}\right)\cdot \left(\begin{array}{cc}
a+b\phi & b\\
\delta b & a\\
\end{array}\right)\cdot \left(\begin{array}{cc}
1 & 0\\
-\phi/2 & 1\\
\end{array}\right)=\left(\begin{array}{cc}
a+b\phi/2 & b\\
(\delta +\phi^2/4)b & a+b\phi/2\\
\end{array}\right)$$
 shows that if $2$ is invertible in $R$, or $\phi\in (2)$, then there is an element $\phi/2 \in R$ (not necessarily unique), such that $\cC_{\delta,\phi}(R)$ is a conjugate of $\cC_{(\delta+\phi^2/4),0}(R)$ in $\GL(2,R)$. We shall abreviate $\cC_{\delta,0}(R)$ by $\cC_\delta(R)$.  Our first step is to compute the normalizer of $\cC_{\delta,\phi}(R)$, when $R$ is a local domain in characteristic zero.
 
 \begin{lemma}\label{lem-variety} Let $R$ be a domain, let $\delta,\phi\in R$, and let $F$ be  the field of fractions of $R$. Let $N_R(\cC_{\delta,\phi}(R))$  be the normalizer of $\cC_{\delta,\phi}(R)$ in $\GL(2,R)$, and let  $N_F(\cC_{\delta,\phi}(F))$ be the normalizer of $\cC_{\delta,\phi}(F)$ in  $\GL(2,F)$. Then, $N_R(\cC_{\delta,\phi}(R))= N_F(\cC_{\delta,\phi}(F)) \cap \GL(2,R)$.
 	\end{lemma} 
 \begin{proof}
 	Let $g\in  N_F(\cC_{\delta,\phi}(F)) \cap \GL(2,R)$ and let $c\in \cC_{\delta,\phi}(R)$. Since $\cC_{\delta,\phi}(R)\subseteq \cC_{\delta,\phi}(F)$, it follows that $gcg^{-1}\in \cC_{\delta,\phi}(F)$, and since $g$ and $c$ are in $\GL(2,R)$, it follows that $gcg^{-1}\in \cC_{\delta,\phi}(R)$. Thus, $g\in N_R(\cC_{\delta,\phi}(R))$. This shows that $N_F(\cC_{\delta,\phi}(F)) \cap \GL(2,R)\subseteq N_R(\cC_{\delta,\phi}(R))$.
 	
 	For the other containment, let us consider $\cC_{\delta,\phi}/R$ as a variety in $\mathbb{A}^4$ defined over $R$ by the equations $X-W=\phi Y$ and $Z=\delta Y$, and a function $\det\colon \cC_{\delta,\phi}(F)\to F$ defined by $\det(X,Y,Z,W)=XW-YZ$. Under this notation, an element $g\in N_R(\cC_{\delta,\phi}(R))$ defines an algebraic function $\tau_g\colon \cC_{\delta,\phi} \to \cC_{\delta,\phi}$, defined over $R$, which commutes with $\det$, i.e., $\det(\tau_g)=\tau_g(\det)$. In particular, if $c=(x_0,y_0,z_0,w_0)\in \cC_{\delta,\phi}(F)$, then $\tau_g(c)\in \cC_{\delta,\phi}(F)$, and since $\det(c)\neq 0$, it follows that $\det(\tau_g(c))\neq 0$ as well. Hence, $\tau_g(c)\in \cC_{\delta,\phi}(F)$. This shows that $N_R(\cC_{\delta,\phi}(R))\subseteq N_F(\cC_{\delta,\phi}(F))$, and since $N_R(\cC_{\delta,\phi}(R))\subseteq \GL(2,R)$, it follows that $N_R(\cC_{\delta,\phi}(R))\subseteq N_F(\cC_{\delta,\phi}(F))\cap \GL(2,R)$, as desired.
 \end{proof}

\begin{lemma}\label{lem-normalizerchar0} If $R$ is a local domain of characteristic $0$, then the normalizer of $\cC_{\delta,\phi}(R)$ in $\GL(2,R)$ is given by $\mathcal{N}_{\delta,\phi}(R)=\langle \cC_{\delta,\phi}(R),c_\phi\rangle$ where $c_\phi=\left(\begin{array}{cc}
-1 & 0\\
\phi & 1\\
\end{array}\right)$. In other words, the normalizer is the group
$$\left\langle \cC_{\delta,\phi}(R),\left(\begin{array}{cc} -1 & 0\\ \phi & 1\\\end{array}\right)\right\rangle = \left\{\left(\begin{array}{cc}
a+b\phi & b\\
\delta b & a\\
\end{array}\right), \left(\begin{array}{cc}
-a & b\\
a\phi- \delta b& a\\
\end{array}\right) : a,b\in R,\ a^2+ab\phi - \delta b^2\in R^\times\right\}.$$
Moreover, if $G$ is a subgroup of $\cC_{\delta,\phi}(R)$ that contains a non-diagonal element, then the normalizer of $G$ in $\GL(2,R)$ is contained in $\mathcal{N}_{\delta,\phi}(R)$. 
\end{lemma} 
\begin{proof}
 Let us first assume that $\phi=0$, so that  $\cC_{\delta,0}(R)=\cC_{\delta}(R)$, let $N_R(\cC_\delta(R))$  be the normalizer of $\cC_{\delta}(R)$ in $\GL(2,R)$, and let $G\subseteq \cC_{\delta}(R)$ be a subgroup, such that $G$ contains an element $A=\left(\begin{array}{cc}
a & b\\
\delta b & a\\
\end{array}\right)\in \cC_{\delta}(R)$ with $b\neq 0$. Let $M=\left(\begin{array}{cc}
c & d\\
e & g\\
\end{array}\right)$ be an element of $\GL(2,R)$ such that $M^{-1}AM$ is in $G\subseteq \cC_{\delta}(R)$. Since
$$M^{-1}AM = \frac{1}{cg-de}\left(\begin{array}{cc}
acg - ade - \delta bcd  + beg  & -\delta bd^2  + bg^2\\
\delta bc^2  - be^2 &  acg - ade + \delta bcd  - beg\\
\end{array}\right) $$
belongs to $\cC_{\delta}(R)$, it follows that 
\begin{align*}
(-\delta bd^2  + bg^2)\delta & = \delta bc^2  - be^2\\
2beg  &= 2\delta bcd.
\end{align*}
Since $b\neq 0$, and since $R$ is an integral domain of characteristic $0\neq 2$, then we get 
\begin{align*}
(-\delta d^2  + g^2)\delta & = \delta c^2  - e^2\\
 eg  &= \delta cd.
\end{align*}
Since $R$ is a local domain, and $cg-de$ is invertible, at least one of $e$ and $g$ must be a unit. Let us assume that $g$ is invertible in $R$ (if $e$ is invertible, one proceeds similarly). Then, $e=\delta cd/g$ and it follows that 
$$g^{-2}(g^2-\delta d^2)(g^2-c^2)=0.$$
If $g^2-\delta d^2=0$, then $\delta=u^2$ is a square in $R$, and it follows that $g=\varepsilon ud$ for some $\varepsilon=\pm 1$, and $e=\delta cd/g = \varepsilon u c$. But in this case, $cg-de = 0$, which is a contradiction. Thus, $g^2-c^2=0$ and so $g=\varepsilon c$, and $e=\varepsilon \delta g$. Thus, the normalizer of $G$ is contained in
$$ N=\left\langle \cC_{\delta}(R),\left(\begin{array}{cc} -1 & 0\\ 0 & 1\\\end{array}\right)\right\rangle = \left\{\left(\begin{array}{cc}
a & b\\
\delta b & a\\
\end{array}\right), \left(\begin{array}{cc}
-a & b\\
- \delta b& a\\
\end{array}\right) : a,b\in R,\ a^2 - \delta b^2\in R^\times\right\}.$$
Moreover, it is easy to show that the normalizer of $\cC_\delta(R)$ contains $N$, and so $N_R(\cC_\delta(R))=N$, as claimed. 

Now suppose that $\phi\neq 0$. Let $F$ be the fraction field of $R$, and let $N_F(\cC_{\delta,\phi}(F))$ be the normalizer of $\cC_{\delta,\phi}(F)$ in  $\GL(2,F)$. By Lemma \ref{lem-variety}, we have $N_R(\cC_{\delta,\phi}(R))= N_F(\cC_{\delta,\phi}(F)) \cap \GL(2,R)$, so we will compute $N_F(\cC_{\delta,\phi}(F))$ instead. If we let $\Lambda = \left(\begin{array}{cc}
1 & 0\\
-\phi/2 & 1\\
\end{array}\right)$, then
$$\Lambda^{-1} \cdot  \cC_{\delta,\phi}(F) \cdot \Lambda = C_{\delta + \phi^2/4}(F)$$
by our observation preceding the statement of the lemma. By our work in the first part of this proof, if we put $\delta'=\delta+\phi^2/4$, we obtain 
$$N_{F}(C_{\delta'}(F)) =  \left\{\left(\begin{array}{cc}
a & b\\
\delta' b & a\\
\end{array}\right), \left(\begin{array}{cc}
-a & b\\
- \delta' b& a\\
\end{array}\right) : a,b\in F,\ a^2 - \delta' b^2\in F^\times\right\}.$$
Hence, $N_F(\cC_{\delta,\phi}(F)) = \Lambda\cdot  N_F(C_{\delta+\phi^2/4}(F))\cdot  \Lambda^{-1}$ is given by
$$\left\{\left(\begin{array}{cc}
a+b\phi & b\\
\delta b & a\\
\end{array}\right), \left(\begin{array}{cc}
-a & b\\
a\phi- \delta b& a\\
\end{array}\right) : a,b\in F,\ a^2+ab\phi - \delta b^2\in F^\times\right\}.$$
Thus, $N_R(\cC_{\delta,\phi}(R))= N_F(\cC_{\delta,\phi}(F)) \cap \GL(2,R)$ is as given in the statement of the lemma. Finally, if $G$ is a subgroup of $\cC_{\delta,\phi}(R)$, then we consider $G$ as a subgroup of $\cC_{\delta,\phi}(F)$ and therefore $G'=\Lambda^{-1}\cdot G\cdot \Lambda$ is a subgroup of $C_{\delta'}(F)$. By the first part of the lemma, the normalizer $N'$ of $G'$ in $\GL(2,F)$ is contained in $N_F(C_{\delta'}(F))$, and therefore the normalizer of $G$ in $\GL(2,F)$ is contained in $N_F(\cC_{\delta,\phi}(F))$. Thus, the normalizer of $G$ in $\GL(2,R)$ must be contained in $N_F(\cC_{\delta,\phi}(F))\cap \GL(2,R)=N_R(\cC_{\delta,\phi}(R))$, as desired.
\end{proof}

The result of Lemma \ref{lem-normalizerchar0} is not true when the characteristic of the ring is positive, or it is not an integral domain. We will describe the normalizers in several cases of interest, namely Cartan subgroups in matrix groups over $\Z/p^n\Z$ for a prime $p$. In order to abbreviate our notation, we will write $\cC_{\delta,\phi}(N)$ for $\cC_{\delta,\phi}(\Z/N\Z)$. The following proposition describes the normalizer of $\cC_{\delta,\phi}(p)$ for all primes $p$, and all values of $\phi$ and $\delta \bmod p$. Since each individual result is easy to verify, we will not include the (lengthy) proof of the proposition here.

\begin{prop}\label{prop-normalizermodp}
Let $p$ be a prime, let $\delta,\phi\in \Z/p\Z$, and consider the groups $\cC_{\delta}(p)=\cC_{\delta}(\Z/p\Z)$ and $\cC_{\delta,\phi}(2)=\cC_{\delta,\phi}(\Z/2\Z)$, let $c=\left(\begin{array}{cc}
-1 & 0\\
0 & 1\\
\end{array}\right)$, and let $c_\phi=\left(\begin{array}{cc}
-1 & 0\\
\phi & 1\\
\end{array}\right)$. Then, 
\begin{enumerate}
\item If $p=2$, and $\delta\equiv \phi\equiv 1 \bmod 2$, then $\cC_{\delta,\phi}(2) = \left\langle \left(\begin{array}{cc}
1 & 1\\
1 & 0\\
\end{array}\right) \right\rangle \cong \Z/3\Z$ and its normalizer is all of $N=\GL(2,\Z/2\Z)=\cC_{\delta,\phi}(2)\rtimes \langle c_\phi\rangle$.
\item If $p=2$, and $\phi\equiv 1$, $\delta\equiv 0\bmod 2$, then $\cC_{\delta,\phi}(2)$ is trivial and its normalizer is all of $\GL(2,\Z/2\Z)$.
\item If $p=2$, and $\delta\equiv 1\bmod 2$, then $\cC_\delta(2) = \left\langle \left(\begin{array}{cc}
0 & 1\\
1 & 0\\
\end{array}\right) \right\rangle \cong \Z/2\Z$ and it is its own normalizer.
\item If $p=2$, and $\delta\equiv 0 \bmod 2$, then $\cC_\delta(2) = \left\langle \left(\begin{array}{cc}
1 & 1\\
0 & 1\\
\end{array}\right) \right\rangle \cong \Z/2\Z$ and it is its own normalizer.
\item If $p> 2$, then
	\begin{enumerate}
    	\item If $\delta \not\equiv 0 \bmod p$, and $\delta$ is a square mod $p$, then $\cC_\delta(p)\cong ((\Z/p\Z)^\times)^2$, and its normalizer $N$ is generated by $\cC_\delta(p)$ and $c$. More concretely, $N\cong \cC_\delta(p)\rtimes \langle c\rangle$. 
            	\item If $\delta \not\equiv 0 \bmod p$, and $\delta$ is not a square mod $p$, then $\cC_\delta(p)\cong \F_{p^2}^\times$, and its normalizer $N$ is generated by $\cC_\delta(p)$ and $c$. More concretely, $N\cong \cC_\delta(p)\rtimes \langle c\rangle$.
                            	\item If $\delta \equiv 0 \bmod p$, then $\cC_\delta(p)\cong (\Z/p^2\Z)^\times$, and its normalizer $N$ is generated by $\cC_\delta(p)$ and $c_g=\left(\begin{array}{cc}
g & 0\\
0 & 1\\
\end{array}\right)$, where $g$ is a primitive root modulo $p$. More concretely, $N\cong \cC_\delta(p)\rtimes_\phi \langle c_g\rangle \cong (\Z/p^2\Z)^\times \rtimes_\phi (\Z/p\Z)^\times$, where the associated map $\phi: \langle c_g\rangle \cong (\Z/p\Z)^\times \to \Aut((\Z/p^2\Z)^\times)$ is defined by $\phi(g)((a,b))=(a,bg)$, for any $(a,b)\in (\Z/p\Z)^\times \times \Z/p\Z\cong (\Z/p^2\Z)^\times$.
    \end{enumerate} 
\end{enumerate}
\end{prop} 

Suppose $\delta\equiv 0 \bmod p$. Then, Prop. \ref{prop-normalizermodp}, part (5c) shows that the index of $\cC_{\delta}(p)$ in its normalizer is $p-1$, which is greater than $2$ if $p>3$. The following lemma shows that there is only one subgroup $N$ of the normalizer such that $N/\cC_{\delta}(p)$ is of order $2$.

\begin{cor} Let $p>2$ be a prime. Let $\delta\equiv 0 \bmod p$, and let $N$ be the normalizer of $C=\cC_\delta(p)$. Then, there is a unique subgroup $N'$ of $N$ containing $C$ such that $N'/C$ is of size $2$, namely:
$$N'=\left\{ \left(\begin{array}{cc}
		\pm a & b\\
		0 & a\\
		\end{array}\right) : a\in (\Z/p\Z)^\times,\ b\in\Z/p\Z \right\}.$$
\end{cor} 
\begin{proof}
	The result follows from Prop. \ref{prop-normalizermodp}, part (5c), since $N/C$ is cyclic of order $p-1$.
\end{proof}

The next lemma is used to construct normalizers in towers of matrix groups.

\begin{lemma}\label{lem-normalizer} Let $\varphi:S \to R$ be a surjective homomorphism of commutative rings such that it induces a surjective homomorphism $\varphi\colon\GL(2,S)\to \GL(2,R)$, let $H_S$ be a subgroup of $\GL(2,S)$ with normalizer $N(H_S)\subseteq \GL(2,S)$, let $H_R=\varphi(H_S)$, and let $N(H_R)$ be the normalizer of $H_R$ in $\GL(2,R)$. Then, 
	\begin{enumerate}
		\item $\varphi(N(H_S))\subseteq N(H_R)$.
		\item In particular, if $N$ is a (set-theoretic) lift of $N(H_R)$ to $\GL(2,S)$ via $\varphi$, and $K_\varphi=\ker(\varphi)\subseteq \GL(2,S)$, then $N(H_S)\subseteq \langle N, K_\varphi \rangle$. 
	\end{enumerate} 
\end{lemma} 
\begin{proof}
	Suppose $\hat{g}\in N(H_S)$, let $g=\varphi(\hat{g})$, and let $a\in H_R$. Since $\varphi(H_S)=H_R$, there is some $\hat{a}\in H_S$ such that $\varphi(\hat{a})=a$. Moreover, since $N(H_S)$ is a normalizer and $\hat{g}\in N(H_S)$, it follows that $\hat{g}\cdot \hat{a} \cdot \hat{g}^{-1} \in H_S$. Therefore,
	$$\varphi(\hat{g}\cdot \hat{a} \cdot \hat{g}^{-1}) = g\cdot a \cdot g^{-1}$$
	belongs to $H_R$. This shows that $g\in N(H_R)$, and therefore $\varphi(N(H_S))\subseteq N(H_R)$ as desired in (1).
	
	From part (1), we obtain $N(H_S)\subseteq \varphi^{-1}(N(H_R))$ and (2) follows.
\end{proof} 

The following result is used in combination with Lemma \ref{lem-normalizer} to compute the normalizers of $\cC_{\delta,\phi}(p^n)$. The proof is trivial (using the fact that $\GL(2,\Z/p^{n+1}\Z)\to \GL(2,\Z/p^n\Z)$ is surjective) and it is omitted. 

\begin{lemma}\label{lem-kernels}
	Let $p$ be a prime, let $n\geq 1$, let $\delta,f\in \Z/p^n\Z$, and consider the groups $\cC_{\delta}(p^n)$ and $C_{\delta,\phi}(\Z/2^n\Z)$. Let $K_{n+1}$ be the kernel of the map $\GL(2,\Z/p^{n+1}\Z)\to \GL(2,\Z/p^n\Z)$ induced by reduction mod $p^n$. Then:
	\begin{enumerate}
		\item $K_{n+1}\cong\left\langle \left(\begin{array}{cc}
		1 & p^n\\
		0 & 1\\
		\end{array}\right),\left(\begin{array}{cc}
		1 & 0\\
		p^n & 1\\
		\end{array}\right),\left(\begin{array}{cc}
		1+p^n & 0\\
		0 & 1\\
		\end{array}\right),\left(\begin{array}{cc}
		1 & 0\\
		0 & 1+p^n\\
		\end{array}\right)\right\rangle\cong (\Z/p\Z)^4$.
		\item $K_{n+1}\cap C_{\delta,\phi}(2^{n+1})=\left\langle \left(\begin{array}{cc}
		1+\phi\cdot 2^n & 2^n\\
		\delta\cdot 2^n & 1\\
		\end{array}\right),\left(\begin{array}{cc}
		1+2^n & 0\\
		0 & 1+2^n\\
		\end{array}\right)\right\rangle\cong (\Z/2\Z)^2.$
		\item  $K_{n+1}\cap \cC_{\delta}(\Z/p^{n+1}\Z)=\left\langle \left(\begin{array}{cc}
		1 & p^n\\
		\delta\cdot p^n & 1\\
		\end{array}\right),\left(\begin{array}{cc}
		1+p^n & 0\\
		0 & 1+p^n\\
		\end{array}\right)\right\rangle\cong (\Z/p\Z)^2.$ 
	\end{enumerate} 
\end{lemma}

Finally, we describe the normalizer of $\cC_{\delta,\phi}(p^n)$ for each prime $p$ and values of $\delta$ and $\phi$.

\begin{prop}
	Let $p$ be a prime, let $n\geq 1$, let $\delta,\phi\in \Z/p^n\Z$, and consider the groups $\cC_{\delta}(p^n)$ and $\cC_{\delta,\phi}(2^n)$, and define matrices in $\GL(2,\Z/p^n\Z)$ by $c_0=\left(\begin{array}{cc}
	-1 & 0\\
	0 & 1\\
	\end{array}\right)$, and $c_\phi=\left(\begin{array}{cc}
	-1 & 0\\
	\phi & 1\\
	\end{array}\right)$.
	Then, 
	\begin{enumerate} 
		\item If $p=2$, and $\delta\equiv \phi\equiv 1 \bmod 2$, then the normalizer of $\cC_{\delta,\phi}(2^n)$ is $N_n=\cC_{\delta,\phi}(2^n)\rtimes \langle c_\phi\rangle$, and $[N_n:\cC_{\delta,\phi}]=2$ for all $n\geq 1$.
		\item If $p=2$, $n\geq 2$, and $\phi\equiv 1$, and $\delta\equiv 0\bmod 2$, then the normalizer of $\cC_{\delta,\phi}(2^n)$ is
		$$N_n = \left\langle \cC_{\delta,\phi}(2^n),c_\phi,\left(\begin{array}{cc}
		1 & 0\\
		2^{n-1} & 1\\
		\end{array}\right),\left(\begin{array}{cc}
		1+2^{n-1} & 0\\
		0 & 1\\
		\end{array}\right)\right\rangle.$$
		Thus, $[N_n:\cC_{\delta,\phi}]=8$ for all $n\geq 2$.
		
		\item If $p=2$, and $\nu_2(\delta)=j$ for some $0\leq j \leq n$, then the normalizer of $C=\cC_\delta(2^n)$ is the group
		$$\left\{ \left(\begin{array}{cc}
		e & k\\
		g & h\\
		\end{array}\right) : \varepsilon=\pm 1,\ h\equiv \varepsilon  e \bmod 2^{n-j-1} 
		\text{ if } j<n,\  g\equiv \varepsilon \delta k \bmod 2^{n-1}\right\}\cap \GL(2,\Z/2^n\Z).$$
		Moreover, $[\mathcal{N}_{\delta}(2^n):\cC_{\delta}(2^n)]=1$ for $n=1$, and $2^{\min\{n,3+j\}}$ for all $n\geq 2$.
		\item If $p>2$, and $\nu_p(\delta)=j$ for some $0\leq j \leq n$, then the normalizer of $C=\cC_\delta(p^n)$ is the group
		$$\left\{ \left(\begin{array}{cc}
		e & k\\
		g & h\\
		\end{array}\right) : \varepsilon=\pm 1,\ h\equiv \varepsilon e \bmod p^{n-j} 
		\text{ if } j\leq n,\  g\equiv \varepsilon \delta k \bmod p^{n}\right\}\cap \GL(2,\Z/p^n\Z),$$
		Moreover, $[\mathcal{N}_{\delta}(p^n):\cC_{\delta}(p^n)]=2\cdot p^{\min\{n-1,j\}}$ for all $n\geq 1$.
	\end{enumerate} 
\end{prop}

\begin{proof}
	For parts (1) and (2), we shall use Lemmas \ref{lem-normalizer} and \ref{lem-kernels}. In particular, Lemma \ref{lem-normalizer} implies that if we have computed $N_n$, then $N_{n+1}$ is contained in the subgroup generated by a lift of $N_n$ to $\GL(2,\Z/2^{n+1}\Z)$ and the kernel $K_{n+1}$ of Lemma \ref{lem-kernels}. Now, part (1) follows from the fact that $K_{n+1}\cap N_{n+1} \subseteq \langle \cC_{\delta,\phi}(2^{n+1}),c_\phi\rangle$ while part (2) follows from the fact that $K_{n+1}\cap N_{n+1} = K_{n+1}$.
	
Before we prove (3) and (4), we observe that, in general, 
	$$\left(\begin{array}{cc}
	e & k\\
	g & h\\
	\end{array}\right)\cdot \left(\begin{array}{cc}
	a & b\\
	\delta b & a\\
	\end{array}\right)\cdot \left(\begin{array}{cc}
	e & k\\
	g & h\\
	\end{array}\right)^{-1}$$
	belongs to $\cC_\delta(R)$ if and only if 
	\begin{enumerate}
		\item $eg= 2\delta hk$, and
		\item $g^2-\delta^2k^2 = \delta (h^2-e^2)$.
	\end{enumerate}
	Assume that $e$ is a unit. Then, $g=\delta kh/e$ and 
	$$\frac{\delta}{e^2}\cdot (\delta k^2 - e^2)\cdot (h^2-e^2)=0.$$
	Further, if $R$ is an integral domain (and $\delta\neq 0$), then either $e^2=h^2$ or $e^2=\delta k^2$. If $e=\pm h$, then $\left(\begin{array}{cc}
	e & k\\
	g & h\\
	\end{array}\right)$ belongs to $\left\langle \cC_\delta(R),\left(\begin{array}{cc}
	-1 & 0\\
	0 & 1\\
	\end{array}\right)\right\rangle$. Otherwise, if $e^2=\delta k^2$, then $k$ must be a unit, $\delta=\gamma^2$ is a square, and $e=\pm \gamma k$. This implies that $(h^2-\delta k^2)(h^2-e^2)=0$. If $h=\pm \gamma k$, then $\left(\begin{array}{cc}
	e & k\\
	g & h\\
	\end{array}\right)=\left(\begin{array}{cc}
	\gamma k & k\\
	\pm \delta k & \pm \gamma k\\
	\end{array}\right)$ also belongs to $\left\langle \cC_\delta(R),\left(\begin{array}{cc}
	-1 & 0\\
	0 & 1\\
	\end{array}\right)\right\rangle$.
	
	We shall first prove (3). Let $N_n$ be the normalizer of $C_n=\cC_{\delta}(\Z/2^n\Z)$ in $\GL(2,\Z/2^n\Z)$. Let $e,f,g,h\in \Z/2^n\Z$, and suppose that $\delta\equiv 0 \bmod 2^j$, for some $1\leq j\leq n$. Then, by our work above,
	$$\left(\begin{array}{cc}
	e & k\\
	g & h\\
	\end{array}\right)\cdot \left(\begin{array}{cc}
	a & b\\
	\delta b & a\\
	\end{array}\right)\cdot \left(\begin{array}{cc}
	e & k\\
	g & h\\
	\end{array}\right)^{-1}$$
	belongs to $C_n$ if and only if the following conditions are met
	\begin{enumerate}
		\item $eg\equiv \delta hk \bmod 2^{n-1}$, and
		\item $g^2-\delta^2k^2 \equiv \delta (h^2-e^2) \bmod 2^n$.
	\end{enumerate}
	Either $e$ or $k$ is a unit modulo $2^n$ (and if $j>0$, then $e$ must be a unit). Then, conditions (1) and (2) are, in turn, equivalent to the existence of $\varepsilon =  1$ or $-1$ such that 
	\begin{enumerate}
		\item[(a)] $h\equiv \varepsilon  e \bmod 2^{n-j-1}$ if $j<n$.
		\item[(b)] $g\equiv \varepsilon \delta k \bmod 2^{n-1}$,
	\end{enumerate}
	so the normalizer $N_n$ is as given in the statement, and a counting argument shows the formula for the index $[N_n:C_n]$.
	
	For (4), when $p>2$, the proof is the same as for (3), except that $2$ is now a unit modulo $p^n$, and conditions (1) and (2) are both modulo $p^n$, and conditions (a) and (b) are modulo $p^{n-j}$ and $p^n$, respectively. 
\end{proof}

The following two lemmas will be used in later sections to determine the possible shapes of the image of complex conjugation. We do not provide a proof for the following lemma, since it is just an elementary calculation. 

\begin{lemma}\label{lem-cc1} Let $R$ be an integral domain, let $a,b,\delta\in R$, and let $\left(\begin{array}{cc}
		 -a & b\\
		-\delta b & a\\
		\end{array}\right)$ be a matrix with zero trace and determinant $-1$ (so that $a^2-\delta b^2=1$). Then,	
		\small 
\begin{align*} & \left(\begin{array}{cc}
			 a-1 & -b\\
			-\delta b & a-1\\
			\end{array}\right)\left(\begin{array}{cc}
			 -a & b\\
			-\delta b & a\\
			\end{array}\right)\left(\begin{array}{cc}
					 a-1 & -b\\
					-\delta b & a-1\\
					\end{array}\right)^{-1}  = \left(\begin{array}{cc} 1  &  0\\
													0 &  -1 \end{array}\right),\\
			& \left(\begin{array}{cc}
			a+1 & -b\\
			-\delta b & a+1\\
			\end{array}\right)\left(\begin{array}{cc}
			-a & b\\
			-\delta b & a\\
			\end{array}\right)\left(\begin{array}{cc}
			a+1 & -b\\
			-\delta b & a+1\\
			\end{array}\right)^{-1} = \left(\begin{array}{cc} -1  &  0\\
			0 &  1 \end{array}\right),\\
			& \left(\begin{array}{cc}
			-\delta b & a+1\\
			\delta (a+1) & - \delta b \\
			\end{array}\right)\left(\begin{array}{cc}
			-a & b\\
			-\delta b & a\\
			\end{array}\right)\left(\begin{array}{cc}
			-\delta b & a+1\\
			\delta (a+1) & - \delta b \\
			\end{array}\right)^{-1}= \left(\begin{array}{cc} 1  &  0\\
			0 &  -1 \end{array}\right)
					\end{align*}
					\normalsize 
					where $(a-1)^2-\delta b^2 = 2(1-a)$, and  $(a+1)^2-\delta b^2 = 2(1+a)$, and $(-\delta b)^2-\delta(a+1)^2= -2\delta (a+1).$
\end{lemma}

\begin{lemma}\label{lem-cc2}
Let $R$ be a domain of characteristic different from $2$, and let $c_\varepsilon = \left(\begin{array}{cc} \varepsilon  &  0\\
0 &  -\varepsilon \end{array}\right)\in \GL(2,R)$ for $\varepsilon = \pm 1$. The matrices $\gamma \in \GL(2,R)$ such that $\gamma \cdot c_{-1} \cdot \gamma^{-1}=c_{1}$ are precisely those in the set 
$$\left\{  \left(\begin{array}{cc} 0  &  b\\
c &  0 \end{array}\right) : b,c \in R^\times \right\}.$$
\end{lemma}
\begin{proof}
If we put $\gamma = \left(\begin{array}{cc} a  &  b\\
c &  d \end{array}\right)\in \GL(2,R)$, then
$$\gamma^{-1}\cdot  c_1\cdot  \gamma = \frac{1}{ad-bc}\cdot \left(\begin{array}{cc} ad+bc  &  2bd\\
-2ac &  -(ad+bc) \end{array}\right).$$
Thus, $\gamma^{-1}\cdot  c_1\cdot  \gamma = c_{-1}$ if and only if $2bd=2ac=0$, and $(ad+bc)/(ad-bc)=-1$. Since $R$ is a domain, it follows that the only possibility is $a=d=0$.
\end{proof}

\section{The Galois group $\Gal(H_f(E[N])/\Q(j_{K,f}))$}\label{sec-fromq}

In  Section \ref{sec-fromhf} we have discussed the generic structure of $\Gal(H_f(E[N])/H_f)$, where we recall that $H_f=K(j_{K,f})$.In order to show Theorem \ref{thm-cmrep-intro}, it remains to understand the structure of $\Gal(H_f(E[N])/\Q(j_{K,f}))$. First, we note that $[K(j_{K,f}):\Q(j_{K,f})]=2$.

\begin{lemma}\label{lem-kjqj}
	Let $E$ be an elliptic curve with CM by an imaginary quadratic field $K$, and $j(E)=j_{K,f}$. Then, $[K(j_{K,f}):\Q(j_{K,f})]=2$.
\end{lemma}
\begin{proof}
	This follows from Theorem \ref{thm-cmbasics}, part (3). Indeed, the multiplicative property of degrees in towers 
	$$[K(j(E)):\Q(j(E))]\cdot [\Q(j(E)):\Q]=[K(j(E)):K]\cdot [K:\Q]$$
	together with $[\Q(j(E)):\Q]=[K(j(E)):K]$ implies the result.
\end{proof}

\begin{lemma}\label{lem-normalizer2}
Let $E$ be an elliptic curve with CM by an order $\Of$, and $j(E)=j_{K,f}$. Fix a $\Z/N\Z$-basis of $E[N]$ and let $\rho_{E,N}\colon \Gal(\overline{\Q(j_{K,f})}/\Q(j_{K,f})) \to \Aut(E[N])\cong \GL(2,\Z/N\Z)$ be the representation afforded by the action of Galois on $E[N]$. Let $G_{E,K,N}=\rho_{E,N}(\Gal(\overline{H_f}/H_f))$, where $H_f=K(j_{K,f})$, and let $G_{E,N}$ be the image of $\rho_{E,N}$. Then, 
\begin{enumerate}
\item $G_{E,N}$ is contained in the normalizer of $G_{E,K,N}$ in $\GL(2,\Z/N\Z)$.
\item If we pick a $\Z/N\Z$-basis of $E[N]$ so that $G_{E,K,N}$ is contained in $\cC_{\delta,\phi}(N)$, with $(\delta,\phi)=(\Delta_K,0)$ or $((\Delta_K-1)f^2/4,f)$, then $G_{E,N}$ is contained in $\mathcal{N}_{\delta,\phi}(N)=\langle \cC_{\delta,\phi}(N),c_\phi\rangle$, with $c_\phi$ defined as in Lemma \ref{lem-normalizerchar0}.
\end{enumerate} 
\end{lemma}
\begin{proof}
Let $G_{\Q(j_{K,f})}$ and $G_{H_f}$ be the absolute Galois group of $\Q(j_{K,f})$ and $H_f$, respectively. Then, $G_{\Q(j_{K,f})}/G_{H_f}\cong \Gal(H_f/\Q(j_{K,f}))$ is of order $2$, by Lemma \ref{lem-kjqj}. In particular, $G_{H_f}$ is a normal subgroup of $G_{\Q(j_{K,f})}$. It follows that $G_{E,K,N}=\rho_{E,N}(G_{H_f})$ is normal in $G_{E,N} = \rho_{E,N}(G_{\Q(j_{K,f})})$ and, therefore, $G_{E,N}$ is contained in the normalizer of $G_{E,K,N}$ in $\GL(2,\Z/N\Z)$. This proves (1).

By the Chinese remainder theorem, it suffices to show (2) when $N=p^n$ is a power of a prime $p$. Fix a prime $p$, and let $\rho_{E,p^\infty}\colon G_{\Q(j_{K,f})} \to \Aut(T_p(E))\cong \GL(2,\Z_p)$ be the representation associated to the action of Galois on the Tate module $T_p(E)$. Then, $G_{E,K,p^\infty}=\rho_{E,p^\infty}(G_{H_f})\subseteq  \cC_{\delta,\phi}(p^\infty)=\cC_{\delta,\phi}(\Z_p)$ with respect to an appropriate $\Z_p$-basis of $T_p(E)$ (as in Remark \ref{rem-complexconj}). Since $\Z_p$ is a local domain of characteristic $0$, Lemma \ref{lem-normalizerchar0} shows that $G_{E,p^\infty}=\rho_{E,p^\infty}(G_{\Q(j(E))})\subseteq \mathcal{N}_{\delta,\phi}(p^\infty)=\mathcal{N}_{\delta,\phi}(\Z_p)$. Hence, $G_{E,p^n} \equiv G_{E,p^\infty} \bmod p^n$ is contained in $\mathcal{N}_{\delta,\phi}(p^n)$, as desired. 
\end{proof}

We are ready to show Theorem  \ref{thm-cmrep-intro}.

\begin{theorem}\label{thm-cmrep}
	Let $E/\Q(j_{K,f})$ be an elliptic curve with CM by $\OO_{K,f}$, let $N\geq 3$, and let $\rho_{E,N}$ be the Galois representation $\Gal(\overline{\Q(j_{K,f})}/\Q(j_{K,f})) \to \Aut(E[N])\cong \GL(2,\Z/N\Z)$.  We define groups of $\GL(2,\Z/N\Z)$ as follows:
	\begin{itemize}
		\item If $\Delta_Kf^2\equiv 0\bmod 4$, or $N$ is odd, let $\delta=\Delta_K f^2/4$, and $\phi=0$.
		\item If $\Delta_Kf^2\equiv 1 \bmod 4$, and $N$ is even, let $\delta=\frac{(\Delta_K-1)}{4}f^2$, let $\phi=f$.
	\end{itemize}
	If we define $\cC_{\delta,\phi}(N)$ by
	$$\cC_{\delta,\phi}(N)=\left\{\left(\begin{array}{cc}
	a+b\phi & b\\
	\delta b & a\\
	\end{array}\right): a,b\in\Z/N\Z,\  a^2+ab\phi-\delta b^2 \in (\Z/N\Z)^\times \right\}.$$
	with  $\mathcal{N}_{\delta,\phi}(N) = \left\langle \cC_{\delta,\phi}(N),c_\phi=\left(\begin{array}{cc} -1 & 0\\ \phi & 1\\\end{array}\right)\right\rangle$, then there is a $\Z/N\Z$-basis of $E[N]$ such that the image of $\rho_{E,N}$
	is contained in $\mathcal{N}_{\delta,\phi}(N)$.	Moreover, 
	\begin{enumerate}
		\item $\cC_{\delta,\phi}(N)\cong (\Of/N\Of)^\times$  is a subgroup of index $2$ in  $\mathcal{N}_{\delta,f}(N)$, and
		\item The index of the image of $\rho_{E,N}$ in $\mathcal{N}_{\delta,f}(N)$ coincides with the order of the Galois group $\Gal(K(j_{K,f},E[N])/K(j_{K,f},h(E[N])))$, for a Weber function $h$, and it is a divisor of the order of $\Of^\times/\mathcal{O}_{K,f,N}^\times$, where $\mathcal{O}_{K,f,N}^\times=\{u\in\Of^\times: u\equiv 1 \bmod N\Of\}$.
		\item The index of $\Gal(H_f(E[N])/H_f)$ in $\cC_{\delta,\phi}(N)$ is the same as the index of the Galois group $\Gal(H_f(E[N])/\Q(j_{K,f}))$ in $\mathcal{N}_{\delta,\phi}(N)$.
	\end{enumerate} 
\end{theorem}
\begin{proof} Let $(\delta,\phi)=(\Delta_Kf^2/4,0)$ or $((\Delta_K-1)f^2/4,f)$ be chosen as in the statement of the theorem. By \cite{silverman2}, Theorem 2.2.3, the Galois representation $\rho_{E,N}|_{K(j_{K,f})}\colon\Gal(\overline{K(j_{K,f})}/K(j_{K,f})) \to \Aut(E[N])$ factors through $\Aut_{\Of/N\Of}(E[N])\cong (\Of/N\Of)^\times$. When we represent $\rho_{E,N}|_{K(j_{K,f})}$ with respect to the basis $\{P_N,Q_N\}\cong \{f\tau,1\}$ chosen in Remark \ref{rem-complexconj}, an automorphism $\alpha=a+bf\tau\in (\Of/N\Of)^\times$ corresponds to the matrix 
$$\left(\begin{array}{cc}
a+b\phi & b\\
\delta b & a\\
\end{array}\right)\in  \GL(2,\Z/N\Z)\cong \GL(E[N]).$$
It follows that the image of $\rho_{E,N}|_{K(j_{K,f})}:\Gal(\overline{K(j_{K,f})}/K(j_{K,f})) \to \Aut(E[N])\cong \GL(2,\Z/N\Z)$ is a subgroup of $\cC_{\delta,\phi}(N)\subseteq \GL(2,\Z/N\Z)$.
Moreover, by Lemma \ref{lem-normalizer}, we have that   the image of the Galois representation 
$\rho_{E,N}$ is a subgroup of $\mathcal{N}_{\delta,\phi}(N)$. This shows that the image of $\rho_{E,N}$ is contained in $\mathcal{N}_{\delta,\phi}(N)$. Also, part (1) of the theorem follows directly from the explicit description of $\mathcal{N}_{\delta,\phi}(N)$.

For parts (2) and (3), first note that by Lemma \ref{lem-clark}, we have that $K\subseteq \Q(j_{K,f},E[N])$, since we are assuming that $N\geq 3$. Also, note that the image of $\rho_{E,N}$ is isomorphic to $\Gal(\Q(j_{K,f},E[N])/\Q(j_{K,f}))$ which we just proved embeds into $\mathcal{N}_{\delta,\phi}(N)$. Let us denote  $K(j_{K,f})$ by $H_f$ as before. By part (1) here, and the first injection given by Theorem \ref{thm-incl}, we conclude that $\Gal(H_f(E[N])/H_f) \subseteq \cC_{\delta,\phi}(N)$. Moreover the index of $\cC_{\delta,\phi}(N)$ in $\mathcal{N}_{\delta,\phi}(N)$ is $2$ and the index of $\Gal(H_f(E[N])/H_f)$ in $\Gal(H_f(E[N])/\Q(j_{K,f}))$ is also $2$, therefore the index of $\Gal(H_f(E[N])/H_f)$ in $\cC_{\delta,\phi}(N)$ is the same as the index of the Galois group $\Gal(H_f(E[N])/\Q(j_{K,f}))$ in $\mathcal{N}_{\delta,\phi}(N)$. Finally, Theorem \ref{thm-incl} says that the index of the Galois group $\Gal(H_f(E[N])/H_f)$ in $\cC_{\delta,\phi}(N)$ is as claimed in the statement, and this concludes the proof.
\end{proof} 

\subsection{Proofs of parts (1) and (2) of Theorem \ref{thm-largeimage-intro}}\label{sec-proofsparts1and3}

Our results so far allow us to show parts (1) and (2) of Theorem \ref{thm-largeimage-intro}. We will show parts (3) and (4) later on, when we discuss the properties of the $p$-adic image for primes of good reduction in Section \ref{sec-goodredn} (see Prop. \ref{prop-largeimage-intro-part-3} and Theorem \ref{thm-largeimage-intro-part4}).

\begin{thm}\label{thm-largeimage-intro-parts-1-and-2}
	Let $E/\Q(j_{K,f})$ be an elliptic curve with CM by $\OO_{K,f}$. 	\begin{itemize}
		\item If $\Delta_Kf^2\equiv 0\bmod 4$, let $\delta=\Delta_K f^2/4$, and $\phi=0$.
		\item If $\Delta_K\equiv 1 \bmod 4$, let $\delta=\frac{(\Delta_K-1)}{4}f^2$, let $\phi=f$.
	\end{itemize}
	Then, the following hold:
	\begin{enumerate}
		\item Let $\rho_{E}$ be the Galois representation $\Gal(\overline{\Q(j_{K,f})}/\Q(j_{K,f})) \to \varprojlim \Aut(E[N])\cong \GL(2,\widehat{\Z})$, and let $\mathcal{N}_{\delta,\phi} = \varprojlim \mathcal{N}_{\delta,\phi}(N)$. Then, the index of the image of $\rho_{E}$ in $\mathcal{N}_{\delta,\phi}$ is a divisor of the order $\Of^\times$. In particular, the index is a divisor of $4$ or $6$. Moreover, for every $K$ and $f\geq 1$, and a  fixed $N\geq 3$, there is an elliptic curve $E/\Q(j_{K,f})$ such that the index of the image of $\rho_{E,N}$ in $\mathcal{N}_{\delta,\phi}(N)$ is $1$.
		
		\item Let $p>2$ and $j_{K,f}\neq 0$, or $p>3$. Let $G_{E,p^\infty}\subseteq \mathcal{N}_{\delta,\phi}(p^\infty) $ be the image of $\rho_{E,p^\infty}$ with respect to a suitable basis of $T_p(E)$, and let $G_{E,p}\subseteq \mathcal{N}_{\delta,\phi}(p)$ be the image of $\rho_{E,p} \equiv \rho_{E,p^\infty}\bmod p$. Then, $G_{E,p^\infty}$ is the full inverse image of $G_{E,p}$ via the reduction mod $p$ map $\mathcal{N}_{\delta,\phi}(p^\infty)\to \mathcal{N}_{\delta,\phi}(p)$.  
	\end{enumerate}
\end{thm}
\begin{proof}
	The first part of (1) follows from Theorem \ref{thm-cmrep}, because we have shown that for any $N\geq 3$, the image of $\rho_{E,N}$ in $\mathcal{N}_{\delta,\phi}(N)$ is a divisor of the order of $\Of^\times/\OO_{K,f,N}^\times$, which is in turn a divisor of the order of $\Of^\times$, and this order can be $2$, $4$, or $6$. 
	
	If we fix $K$, $f$, and $N\geq 3$,  then Theorem \ref{thm-twistforfullimage} shows that there exists an elliptic curve $E/\Q(j_{K,f})$ with CM by $\Of$, such that $\Gal(H_f(E[N])/H_f) \cong (\Of/N\Of)^\times$, where $H_f=K(j_{K,f})$. Hence, by Theorem \ref{thm-cmrep}, there is a basis of $E[N]$ such that the image of $\rho_{E,N} \colon \Gal(\overline{H_f}/H_f) \to \GL(2,\Z/N\Z)$ is all of $\cC_{\delta,\phi}(N)$, and therefore the image of $\rho_{E,N}$ is all of  $\mathcal{N}_{\delta,\phi}(N)$. This shows the rest of (1).
	
	For (2), let us fix compatible $\Z/p^n\Z$-basis of $E[p^n]$, for every $n\geq 1$, and a $\Z_p$-basis of $T_p(E)$, such that the image of $\rho_{E,p^n}$ and $\rho_{E,p^\infty}$ are contained in $\mathcal{N}_{\delta,\phi}(p^n)$ and  $\mathcal{N}_{\delta,\phi}(p^\infty)$, respectively. Let $G_{E,p^n}$ be the image of $\rho_{E,p^n}$, and let $G_{E,K,p^n}$ be the image of $\Gal(\overline{H_f}/H_f)$ through $\rho_{E,p^n}$. Then, Theorem \ref{thm-inclp} shows that, if $p>2$ and $j_{K,f}\neq 0$, or $p>3$, for all $n\geq 1$  the group $G_{E,K,p^{n+1}}$ is the full inverse image of $G_{E,K,p^n}$ under the natural reduction map $\cC_{\delta,\phi}(p^{n+1})\to \cC_{\delta,\phi}(p^n)$ induced by the map $(\Of/p^{n+1}\Of)^\times \to (\Of/p^{n}\Of)^\times$.
	
	Now, let $c$ be a fixed complex conjugation. Since $c$ does not fix $K$, it follows it does not fix $H_f$, and therefore it is an element in $\Gal(\overline{\Q(j_{K,f})}/\Q(j_{K,f})$ that is not in  $\Gal(\overline{H_f}/H_f)$, and the image $\gamma_n=\rho_{E,p^n}(c)$ together with $G_{E,K,p^n}$ generate $G_{E,p^n}$. Let $J_n$ be the full inverse image of $\langle \gamma_k, G_{E,K,p} \rangle$ via $\pi\colon \mathcal{N}_{\delta,\phi}(p^n)\to \mathcal{N}_{\delta,\phi}(p)$, where $\gamma_k=\pi(\gamma_n)$. We claim that $J_n = G_{E,p^n}$. Since $\pi(\gamma_n)=\gamma_1$, and since we have shown that $G_{E,K,p^n}$ is the full inverse image of $G_{E,K,p}$, it follows that $G_{E,p^n}\subseteq J$. To prove equality, it suffices to show that if $\gamma_n'$ is any other element in $J_n$ such that $\pi(\gamma_n')=\gamma_k$, then $\gamma_n'\in G_{E,p^n}$. Indeed, since $\# \mathcal{N}_{\delta,\phi}(p^n)/\cC_{\delta,\phi}(p^n) = \# \mathcal{N}_{\delta,\phi}(p)/\cC_{\delta,\phi}(p) =2$, and $\gamma_n$ and $\gamma_n'$ lie above $\gamma_k$ which is not in the Cartan at level $p^k$, it follows that $\gamma_n,\gamma_n'$ are not in the Cartan at level $p^n$, and therefore $\gamma_n^{-1}\cdot \gamma_n'=\alpha \in \cC_{\delta,\phi}(p^n)$. Hence, $\gamma_n' =\gamma_n\cdot \alpha\in G_{E,p^n}$, as desired. This shows that $J_n=G_{E,p^n}$ is the full inverse image of $G_{E,p}$, and concludes the proof of (2).
\end{proof}

In the following section we show that, under some hypotheses, one can choose a change of variables such that the image of $\rho_{E,N}$ is in $\mathcal{N}_{\delta,\phi}(p^n)$ and the image of complex conjugation has a particular shape that is easy to work with.

\subsection{The shape of complex conjugation}\label{sec-cc}

Let $E/\Q(j_{K,f})$ be an elliptic curve with CM by $\OO_{K,f}$, and let $N\geq 3$ (thus, by Lemma \ref{lem-clark}, $K\subseteq \Q(j_{K,f},E[N])$). Let us consider the group $\Aut(E[N])$ of automorphisms of $E[N]$, and the representation $$\rho_{E,N}\colon \Gal\left(\overline{\Q(j_{K,f})}/\Q(j_{K,f})\right)\to \Aut(E[N]).$$
The $N$-torsion of $E$ is a free $\Of/N\Of$-module of rank one (Prop. 1.4 of Ch II, \cite{silverman2}, and \cite[Lemma 1]{parish} in this generality) and, as in Theorem \ref{thm-incl}, if we restrict $\rho_{E,N}$ to the absolute Galois group of $H_f = K(j_{K,f})$, then
$$\rho_{E,N}|_{H_f}\colon\Gal\left(\overline{H_f}/H_f\right)\to \Aut_{\Of/N\Of}(E[N])\cong \left(\Of/N\Of\right)^\times.$$
On the other hand, if we fix an embedding of $\Q(j_{K,f})$ into $\CC$, then complex conjugation in $\CC$ induces a complex conjugation $c\in \Gal\left(\overline{\Q(j_{K,f})}/\Q(j_{K,f})\right)$, and the Galois group of the quadratic extension $H_f/\Q(j_{K,f})$ is generated by the restriction of complex conjugation, which we will denote by $c|_{H_f}$. If we fix an isomorphism $[\cdot] \colon \OO_{K,f} \to \End(E)$, normalized as in \cite[Ch 2., Prop. 1.1]{silverman2}, then $\sigma([\alpha])=[\sigma(\alpha)]$ for any $\sigma \in \Aut(\CC)$ by \cite[Ch. 2, Thm. 2.2]{silverman2}. In particular, this means that $\langle c|_{H_f} \rangle = \Gal(H_f/\Q(j_{K,f}))$ acts on $\Gal(H_f(E[N])/H_f)\subseteq \left(\Of/N\Of\right)^\times$ as complex conjugation. It follows that $\Gal(\Q(j_{K,f},E[N])/\Q(j_{K,f}))$ is a subgroup of the (external)  semi-direct product $\left(\Of/N\Of\right)^\times \rtimes_\varphi \Z/2\Z$ where $\varphi\colon \Z/2\Z\cong \langle c|_{H_f}\rangle \to \Aut(\left(\Of/N\Of\right)^\times)$ is defined by $\varphi(c|_{H_f})$, which is the automorphism that sends $\alpha \bmod N\Of\mapsto c(\alpha) \bmod N\Of$, the congruence class of the complex conjugate of $\alpha\in\Of$.

We note here that, as it is usual in semi-direct products, as an element of the Galois group $\Gal(\Q(j_{K,f},E[N])/\Q(j_{K,f}))$, the complex conjugation $c$ acts internally on $\Gal(\Q(j_{K,f},E[N])/H_f)$ via (group-theoretic) conjugation. Indeed, if we write an element $\sigma\in \Gal(\Q(j_{K,f},E[N])/\Q(j_{K,f}))$ as a pair $(\alpha, \beta)\in \left(\Of/N\Of\right)^\times \rtimes_\varphi \Z/2\Z$, and we put $c=(1,c|_{H_f})$, then if $\sigma \in \Gal(\Q(j_{K,f},E[N])/H_f)$, we can write $\sigma=(\alpha,1)$ and
\begin{align*}
c\cdot \sigma \cdot c^{-1} & = (1,c|_{H_f})(\alpha,1)(1,c|_{H_f})^{-1}\\
&= (1,c|_{H_f})(\alpha,1)(1,c|_{H_f}) = (1,c|_{H_f})(\alpha,c|_{H_f}) = (c(\alpha),1). 
\end{align*}
Thus, internally, $\Gal(\Q(j_{K,f},E[N])/\Q(j_{K,f}))\cong \Gal(\Q(j_{k,f},E[N])/H_f)\rtimes_\varphi \langle c\rangle$ where the semi-direct product is given by $\varphi\colon \langle c\rangle \to \Aut(\Gal(\Q(j_{k,f},E[N])/H_f))$ defined by $\varphi(c)(\sigma)=c\cdot \sigma\cdot  c^{-1}$.

Let us translate this into coordinates. Fix an appropriate $\Z/N\Z$-basis of $E[N]$, so that the image  $\rho_{E,N}(\Gal(\overline{H_f}/H_f))$ is a subgroup of $\cC_{\delta,\phi}(N)\subseteq \GL(2,\Z/N\Z)$, where $\delta$ and $\phi$ are chosen as in Theorem \ref{thm-cmrep-intro}. By Lemma \ref{lem-normalizer}, the image $\rho_{E,N}\left(\Gal\left(\overline{\Q(j_{K,f})}/\Q(j_{K,f})\right)\right)$ is contained in $\mathcal{N}_{\delta,\phi}(N)$. Let us investigate where a complex conjugation lands via $\rho_{E,N}$.

\begin{lemma}\label{lem-zerotrace}
	Let $c \in \Gal\left(\overline{\Q(j_{K,f})}/\Q(j_{K,f})\right)$ be a complex conjugation. Then, $\rho_{E,N}(c)$ is a matrix in $\GL(2,\Z/N\Z)$ with determinant $-1\bmod N$ and zero trace.
\end{lemma}
\begin{proof}
	It follows from the existence of the Weil pairing that  $\chi_{E,N}=\det(\rho_{E,N})$ is the $N$-th cyclotomic character. In particular, $\det(\rho_{E,N}(c))=\chi_{E,N}(c)\equiv -1 \bmod N$. Moreover, $c^2$ is the identity, and therefore $(\rho_{E,N}(c))^2 \equiv \operatorname{Id} \bmod N$. In particular, $\rho_{E,N}(c)$ is congruent to its own inverse modulo $N$, and therefore
	$$\left(\begin{array}{cc} a & b\\ c & d\\\end{array}\right) \equiv \frac{1}{ad-bc}\left(\begin{array}{cc} d & -b\\ -c & a\\\end{array}\right) \equiv \left(\begin{array}{cc} -d & b\\ c & -a\\\end{array}\right).$$
	Hence, $a\equiv -d\bmod N$, and the trace of $\rho_{E,N}(c)$ is zero.
\end{proof}

There exist elements of a Cartan subgroup $\cC_{\delta,\phi}(N)$ of determinant $-1$ and zero trace. For instance, if $p$ is odd, and $\Delta_K$ is a square modulo $p$, let $\delta=\Delta_K f^2/4 = d^2$ for some $d\in\Z_p$, and consider the matrix $\left(\begin{array}{cc} 0 & d^{-1}\\ \delta d^{-1} & 0\\\end{array}\right).$ We shall use a result of Zarhin to show that the image of a complex conjugation via $\rho_{E,N}$ cannot be contained in a Cartan subgroup, for $N\geq 3$.

\begin{thm}[Zarhin, \cite{zarhin}]\label{thm-zarhin} Let $L$ be a number field, let $L^{\text{ab}}$ be the maximal abelian extension of $L$, and let $A/L$ be an abelian variety defined over $L$. Then, $A(L)_{\text{tors}}$ is finite if and only if $A$ has no abelian subvariety with CM-type defined over $L$. 
\end{thm}

\begin{thm}\label{lem-insidecartan}
	Let $E/\Q(j_{K,f})$ be an elliptic curve with CM, let $p$ be a prime, let $n=2$ if $p=2$ and $n=1$ if $p>2$, and choose a basis of $T_p(E)$ as in Theorem \ref{thm-cmrep-intro}. Then, the image of $\rho_{E,p^n}\colon \Gal\left(\overline{\Q(j_{K,f})}/\Q(j_{K,f})\right)\to \GL(2,\Z/p^n\Z)$ cannot be contained in $\calc_{\delta,\phi}(p^n)$. 
\end{thm}
\begin{proof}
	Let us fix a $\Z_p$-basis of $T_p(E)$ so that the image of $\rho_{E,p^\infty}$ is as described in Theorem \ref{thm-cmrep-intro}. In particular, the image $H$ of $\Gal(\overline{K(j_{K,f})}/K(j_{K,f}))$ is contained in the $p$-adic Cartan subgroup $\cC_{\delta,\phi}(p^\infty)$ and the image $G$ of $\Gal(\overline{\Q(j_{K,f})}/\Q(j_{K,f}))$ is contained in $\mathcal{N}_{\delta,\phi}(p^\infty)$.
	
	Suppose that the image of complex conjugation via $\rho_{E,p^\infty}$ is $\gamma\in \GL(2,\Z_p)$ of order $2$, zero trace, and determinant $-1$, and suppose for a contradiction that the image $G_{p^n}$ of $\rho_{E,p^n}$ is contained in the Cartan subgroup (where, $p^n=4$ if $p=2$, and $p^n=p$ if $p>2$).
	\begin{itemize}
		\item Suppose first that $\Delta_K\equiv 0 \bmod 4$ or $f$ is even or $p>2$. Then, $\bar{\gamma}\in G_{p^n}\subseteq \calc_{\delta,\phi}(p^n)$, where $\bar{\gamma}\equiv \gamma \bmod p^n$. If $\gamma$ is {\it not} in the Cartan, then there are $p$-adic integers $e,g\in\Z_p$, such that $\gamma= \left(\begin{array}{cc} -e & g\\ -\delta g & e\\\end{array}\right)$ where $\delta = \Delta_K f^2/4$, and $a,b\in \Z/p^n\Z$ such that 
		$$\bar{\gamma}\equiv \left(\begin{array}{cc} a & b\\ \delta b & a\\\end{array}\right)\equiv \left(\begin{array}{cc} -e & g\\ -\delta g & e\\\end{array}\right) \bmod p^n.$$
		This implies that $a\equiv -a\bmod p^n$ and $\delta b \equiv -\delta b\bmod p^n$, which in turn imply that $2a\equiv 2\delta b\equiv 0 \bmod p^n$. Thus, $a\equiv \delta b\equiv 0 \bmod p$. This is a contradiction because $\det(\bar{\gamma})\equiv a^2-\delta b^2 \equiv 0 \bmod p$ needs to be unit. It follows that $\gamma$ is in the Cartan subgroup of $\GL(2,\Z_p)$, and therefore $G=\langle H,\gamma\rangle$ is also contained in the Cartan.
		
		\item Now suppose that $\Delta_K\equiv 1 \bmod 4$ and $f$ is odd and $p=2$.  Then, $\bar{\gamma}\in G_4\subseteq \calc_{\delta,\phi}(4)$, where $\bar{\gamma}\equiv \gamma \bmod 4$. If $\gamma$ is {\it not} in the Cartan, then there are $2$-adic integers $e,g\in\Z_2$, such that $\gamma= \left(\begin{array}{cc} -e & g\\ ef-\delta g & e\\\end{array}\right)$ where $\delta = (\Delta_K-1) f^2/4$, and $a,b\in \Z/4\Z$ such that 
		$$\bar{\gamma}\equiv \left(\begin{array}{cc} a +bf & b\\ \delta b & a\\\end{array}\right)\equiv \left(\begin{array}{cc} -e & g\\ ef-\delta g & e\\\end{array}\right) \bmod 4.$$
		In particular, $a+bf\equiv -a\bmod 4$ and so $2a+bf\equiv 0 \bmod 4$. Since $f\equiv 1 \bmod 2$, it follows that $b\equiv 0\bmod 2$. The matrix congruence above also yields $\delta b \equiv ef-\delta g$ and $b\equiv g \bmod 4$, and so $2\delta b \equiv af\bmod 4$, which implies that $a\equiv 0 \bmod 2$. But then, $\det(\bar{\gamma})\equiv a^2+abf-\delta b^2\equiv 0 \bmod 2$, which is impossible. We conclude that $\gamma$ is in the Cartan subgroup of $\GL(2,\Z_2)$, and therefore $G=\langle H,\gamma\rangle$ is also contained in the Cartan.
	\end{itemize} 
	In all cases, we conclude that $G$ would be contained in the Cartan, and therefore the extension $\Q(j_{K,f},E[p^\infty])/\Q(j_{K,f})$ is abelian. It follows that if we put $L=\Q(j_{K,f})$, then $E/L$ is an abelian variety with CM, such that $E(L^\text{ab})_\text{tors}$ is infinite, since all of $E[p^\infty]$ would be defined over $L^\text{ab}$. By Zarhin's Theorem \ref{thm-zarhin}, the CM-type of $E$ would be defined over $L$, i.e., all the endomorphisms of $E$ would be defined over $L$. However, the endomorphisms are defined over $K(j_{K,f})$ by Theorem \ref{thm-endo}, and $[K(j_{K,f}):\Q(j_{K,f})]=2$ by Lemma \ref{lem-kjqj}. Thus, not all endomorphisms can be defined over $\Q(j_{K,f})$, and we have reached a contradiction. Hence, the image $G_{p^n}$ cannot be contained in the Cartan subgroup, as claimed. (Note: alternatively, after arriving to the conclusion that the extension $\Q(j_{K,f},E[p^\infty])/\Q(j_{K,f})$ is abelian, the proof could have been concluded using Serre's work \cite{serre66}, Theorem 5.)
\end{proof}

\begin{remark} In some cases $\Q(E[n])$ can be abelian for $n\geq 3$, even for elliptic curves with CM (see \cite{gonzalez-jimenez-lozano-robledo}), without contradicting Theorem \ref{lem-insidecartan}. For instance, let $E: y^2=x^3+16$ defined over $\Q$. Then, $\Q(E[3])=K=\Q(\sqrt{-3})$ is abelian. However, in this case, the image of $\Gal(K(E[3])/K)$ is trivial in  $\cC_{\delta,0}(3)=\left\{ \left(\begin{array}{cc} a & b\\ 0 & a\\\end{array}\right) : a \in (\Z/3\Z)^\times, b\in\Z/3\Z \right\}$, and the image of a complex conjugation is given by the matrix  $\left(\begin{array}{cc} 1 & 0\\ 0 & -1\\\end{array}\right)$, which indeed does not belong to   $\cC_{\delta,0}(3)$, in agreement with Theorem \ref{lem-insidecartan}.

Similarly, if $E:y^2=x^3+x$, then $\Q(E[4])$ is abelian. As we shall see in the proof of Theorem \ref{thm-j1728-intro} below, the image of $\Gal(\Q(i,E[4])/\Q(i))$ is given by the subgroup $\left\langle \left(\begin{array}{cc} 1 & 2\\2 & 1\\\end{array}\right)  \right\rangle\subseteq \cC_{\delta,0}(4)$ and complex conjugation maps to the matrix $\left(\begin{array}{cc} 0 & 3\\ 3 & 0\\\end{array}\right)$ which is not contained in the Cartan subgroup $\cC_{\delta,0}(4)=\left\{ \left(\begin{array}{cc} a & b\\ -b & a\\\end{array}\right) : a,b \in \Z/4\Z, a^2+b^2\not\equiv 0 \bmod 2 \right\}$.

In upcoming work, we will use the results of this paper to describe all the CM elliptic curves over $\Q(j_{K,f})$ and $N\geq 2$ such that the image of $\rho_{E,N}$ is abelian, extending the work of Section 4 of \cite{gonzalez-jimenez-lozano-robledo} which covers the case of $j_{K,f}\in \Q$.
\end{remark}

Theorem \ref{lem-insidecartan} shows that the image of a complex conjugation via $\rho_{E,N}$ lands in $\mathcal{N}_{\delta,\phi}(N) \setminus \cC_{\delta,\phi}(N)$ for any $N\geq 3$. The group $\mathcal{N}_{\delta,\phi}(N)$ can also be described as a semi-direct product, as the following lemma shows.

\begin{lemma}
	Let $\delta$ and $\phi$ be chosen as in Theorem \ref{thm-cmrep-intro}, and let $\gamma$ be any element in $\mathcal{N}_{\delta,\phi}(N) \setminus \cC_{\delta,\phi}(N).$ Then, 
	$$\mathcal{N}_{\delta,\phi}(N) \cong \cC_{\delta,\phi}(N) \rtimes_\varphi \langle \gamma \rangle$$
	where $\varphi \colon \langle \gamma \rangle \to \Aut(\cC_{\delta,\phi}(N))$ is defined by $\gamma \mapsto (\alpha \to \gamma \cdot \alpha \cdot \gamma^{-1})$. More precisely, if $\alpha = \left(\begin{array}{cc} a +b\phi & b\\ \delta b & a\\\end{array}\right)$ and $\gamma \in \mathcal{N}_{\delta,\phi}(N) \setminus \cC_{\delta,\phi}(N)$, then
	$$\gamma \cdot \left(\begin{array}{cc} a +b\phi & b\\ \delta b & a\\\end{array}\right) \cdot \gamma^{-1} = \left(\begin{array}{cc} a  & -b\\ -\delta b & a+b\phi\\\end{array}\right).$$
\end{lemma}
\begin{proof}
	The isomorphism $\mathcal{N}_{\delta,\phi}(N) \cong \cC_{\delta,\phi}(N) \rtimes_\varphi \langle \gamma \rangle$ follows from standard facts about semi-direct products (see \cite{dummitfoote}, Ch. 5, Theorem 12) by noting that $\cC_{\delta,\phi}(N) \vartriangleleft\mathcal{N}_{\delta,\phi}(N)$ because $\mathcal{N}_{\delta,\phi}(N)$ is contained in the normalizer of $\cC_{\delta,\phi}(N)$ in $\GL(2,\Z/N\Z)$, and  $\# \mathcal{N}_{\delta,\phi}(N)/\cC_{\delta,\phi}(N)=2$, so  $\mathcal{N}_{\delta,\phi}(N) = \cC_{\delta,\phi}(N) \cdot \langle \gamma \rangle$ for any $\gamma \in \mathcal{N}_{\delta,\phi}(N) \setminus \cC_{\delta,\phi}(N).$ The second part of the statement amounts to an elementary calculation.
\end{proof}

Note that the matrix $\left(\begin{array}{cc} a +b\phi & b\\ \delta b & a\\\end{array}\right)$ is the multiplication-by-$(a+bf\tau)$ matrix, while the matrix $\left(\begin{array}{cc} a  & -b\\ -\delta b & a+b\phi\\\end{array}\right)$ corresponds to  multiplication-by-$(a+bf-bf\tau)$ matrix, and the complex conjugate of $a+bf\tau$ is precisely $a+bf-bf\tau$. The following lemma will show that if $N$ is odd, and the image is the full group $\mathcal{N}_{\delta,\phi}(N)$, then we can assume that $c$ is given by a matrix of the form $\left(\begin{array}{cc} \varepsilon & 0\\ 0 & -\varepsilon\\\end{array}\right)$, for some $\varepsilon \in \{\pm 1 \}$, after an appropriate change of basis.

\begin{lemma}\label{lem-ccfinal}
	Let $E/\Q(j_{K,f})$ be an elliptic curve with CM by $\Of$, let $p$ be an odd prime, let $G_{E,p^\infty}$ be the image of $\rho_{E,p^\infty}$ and let $G_{E,K,p^\infty}=\rho_{E,p^\infty}(\Gal(\overline{H_f}/H_f))$. Let $c\in\Gal(\overline{\Q(j_{K,f})}/\Q(j_{K,f}))$ be a complex conjugation, and let $\gamma=\rho_{E,p^\infty}(c)$. Then:  
	\begin{enumerate}
		\item There is a $\Z_p$-basis of $T_p(E)$ such that $G_{E,K,p^\infty}\subseteq \cC_{\delta,\phi}(p^\infty)$, and  $\gamma=c_\varepsilon=\left(\begin{array}{cc} \varepsilon & 0\\ 0 & -\varepsilon\\\end{array}\right)$ for some $\varepsilon \in \{\pm 1\}$.
		\item If $\delta=\Delta_K f^2/4 \not \equiv 0 \bmod p$, then for each $\varepsilon\in \{\pm 1 \}$ there is a basis such that $\rho_{E,p^\infty}(c)=c_\varepsilon$.
		\item  If $\delta\equiv 0 \bmod p$, then the group $\langle \cC_{\delta,\phi}(p^\infty),c_1\rangle$ cannot be conjugated to $\langle \cC_{\delta,\phi}(p^\infty),c_{-1}\rangle$ in such a way that $c_1$ is sent to $c_{-1}$.
	\end{enumerate}
\end{lemma}
\begin{proof}
	Since $p$ is odd, we have $\delta = \Delta_K f^2/4$ and $\phi=0$. By Theorem \ref{thm-cmrep} we can choose a $\Z_p$-basis of $T_p(E)$ such that $G_{E,K,p^\infty}\subseteq \cC_{\delta,0}(p^\infty)\subseteq \GL(2,\Z_p)$, and $\gamma=\rho_{E,p^\infty}(c)$ is a matrix with determinant $-1$ and zero trace, so that $G_{E,p^\infty}=\langle G_{E,K,p^\infty}, \gamma\rangle \subseteq \mathcal{N}_{\delta,0}(p^\infty)$. By Theorem \ref{lem-insidecartan}, the element $\gamma$ cannot be contained in the Cartan. Thus $\gamma \in \mathcal{N}_{\delta,\phi}(N) \setminus \cC_{\delta,\phi}(N)$ so we can write $\gamma = \left(\begin{array}{cc} -a & b\\ -\delta b & a\\\end{array}\right)$, for some $a,b\in\Z_p$, and by Lemma \ref{lem-zerotrace}, we have $a^2-\delta  b^2 =1$. Then, Lemma \ref{lem-cc1} shows that if $a\not \equiv 1 \bmod p$, then there is $\alpha \in \cC_{\delta,\phi}(p^\infty)$ such that $\alpha\cdot \gamma \cdot \alpha^{-1}=c_{1}$, and if $a\not\equiv -1 \bmod p$, then there is $\alpha \in \cC_{\delta,\phi}(p^\infty)$ such that $\alpha\cdot \gamma \cdot \alpha^{-1}=c_{-1}$. Notice that since $\alpha$ is in the Cartan, in both cases we have $\alpha\cdot G_{E,K,p^\infty}\cdot \alpha^{-1} = G_{E,K,p^\infty}$, because the Cartan is abelian, and therefore $\alpha \cdot \langle G_{E,K,p^\infty},\gamma \rangle \cdot \alpha^{-1} = \langle G_{E,K,p^\infty}, c_\varepsilon\rangle$, as desired. This shows (1).
	
	For (2), if $\delta\not\equiv 0 \bmod p$, then $\cC_{\delta,\phi}(p^\infty)$ contains a matrix of the form $\alpha = \left(\begin{array}{cc} 0 & b\\ \delta b & 0\\\end{array}\right)$ for some $b\in \Z_p^\times$. Thus, by part (1), there is a basis such that $\gamma=c_1$, and by Lemma \ref{lem-cc2}, we have $\alpha \cdot \langle G_{E,K,p^\infty},c_1 \rangle \cdot \alpha^{-1} = \langle G_{E,K,p^\infty}, c_{-1}\rangle$.
	
	Finally, for (3), Lemma \ref{lem-cc2} says that the only matrices that conjugate $c_1$ to $c_{-1}$ are anti-diagonal, but if $\delta\equiv 0\bmod p$, then there are no anti-diagonal matrices in the normalizer of $\cC_{\delta,\phi}(p^\infty)$, which is $\mathcal{N}_{\delta,\phi}(p^\infty)$ by Lemma \ref{lem-normalizerchar0}. Thus, we cannot change basis preserving the image of $\rho_{E,p^\infty}$ and sending $c_1$ to $c_{-1}$. 
\end{proof}

The following lemma gives a different proof of Lemma \ref{lem-ccfinal}, and it also describes the image of $\rho_{E,p^\infty}$, including the case of $p=2$.

\begin{lemma}\label{lem-ccfinal2}
	Let $E/\Q(j_{K,f})$ be an elliptic curve with CM by $\Of$, let $p$ be a prime, let $G_{E,p^\infty}$ be the image of $\rho_{E,p^\infty}$ and let $G_{E,K,p^\infty}=\rho_{E,p^\infty}(\Gal(\overline{H_f}/H_f))$. Suppose the index of $G_{E,K,p^\infty}$ in $\cC_{\delta,\phi}(p^\infty)$ is $d$ (a divisor of $\# \Of^\times$). Let $c\in\Gal(\overline{\Q(j_{K,f})}/\Q(j_{K,f}))$ be a complex conjugation, let $\gamma=\rho_{E,p^\infty}(c)$, and let $c_\phi =\left(\begin{array}{cc} -1  & 0\\ \phi & 1\\\end{array}\right)$. Then, there is a root of unity $\zeta \in \cC_{\delta,\phi}(p^\infty)$ of order dividing $d$ such that $G_{E,p^\infty} = \langle \gamma, G_{E,K,p^\infty}\rangle= \langle \zeta \cdot c_\phi, G_{E,K,p^\infty}\rangle$.
\end{lemma}
\begin{proof}
	Let $G_{E,p^\infty}$ be the image of $\rho_{E,p^\infty}$ and let $G_{E,K,p^\infty}$ be the image of $G_{H_f}$ via $\rho_{E,p^\infty}$. Then, by Theorem \ref{thm-cmrep-intro}, there is a $\Z_p$-basis such that $G_{E,K,p^\infty}\subseteq \cC_{\delta,\phi}(p^\infty)$ and $G_{E,p^\infty}\subseteq \mathcal{N}_{\delta,\phi}(p^\infty)$, such that, if $c \in G_{\Q(j_{K,f})}$ is a fixed complex conjugation, and $\gamma=\rho_{E,p^\infty}(c)$, then $G_{E,p^\infty}=\langle \gamma, G_{E,K,p^\infty}\rangle$. Moreover, $G_{E,K,p^\infty}$ is a subgroup of index $2$ in $G_{E,p^\infty}$. Further, if the index of $G_{E,K,p^\infty}$ in $\cC_{\delta,\phi}(p^\infty)$ is $d$, then $\langle \zeta_d\rangle \cdot G_{E,K,p^\infty} = \cC_{\delta,\phi}(p^\infty)$, by Cor. \ref{cor-missesrootofunity}.
	
	Note that $c_\phi\gamma \in \cC_{\delta,\phi}(p^\infty)$ because the Cartan is of index $2$ in  $\mathcal{N}_{\delta,\phi}(2^\infty)$ by Lemma \ref{lem-normalizerchar0}, and $c_\phi,\gamma$ are not in the Cartan (by Theorem \ref{lem-insidecartan}). Thus, there is a $d$-th root of unity $\zeta$ such that $\zeta \cdot c_\phi \cdot \gamma \in G_{E,K,p^\infty}$. It follows that $\gamma \in \langle \zeta \cdot c_\phi, G_{E,K,p^\infty}\rangle $ and $\zeta\cdot c_\phi \in \langle \gamma, G_{E,K,p^\infty}\rangle $. Hence, $\langle \zeta \cdot c_\phi, G_{E,K,p^\infty}\rangle = \langle \gamma, G_{E,K,p^\infty}\rangle = G_{E,p^\infty}$, as desired.
\end{proof}

\section{Primes of good reduction}\label{sec-goodredn}

The goal of this section is to determine the image of $\rho_{E,p^\infty}$ for primes $p$ that do not divide $2f\Delta_K$. We begin with a result about the ramification of primes in $\Q(j_{K,f})/\Q$ and $K(j_{K,f})/K$.

\begin{lemma}\label{lem-odd2}
Let $E/\Q(j_{K,f})$ be an elliptic curve with CM by an order $\Of$ of conductor $f$ in an imaginary quadratic field $K$ of discriminant $\Delta_K$, and let $H_f=K(j(E))$ be the ring class field attached to $\Of$. Let $\calc=\operatorname{Cond}(K(j(E))/K)$ be the conductor of the abelian extension $K(j(E))/K$. Then: 
\begin{enumerate}
	\item The relationship between the conductors $f$ of $\Of$ and $\calc$  is given by
$$\operatorname{Cond}(K(j_{K,f})/K)=\begin{cases}
\OO_K & \text{ if } f=2 \text{ or } 3, \text{ and } K=\Q(\sqrt{-3}),\\
\OO_K & \text{ if } f=2, \text{ and } K=\Q(i),\\
(f/2)\OO_K & \text{ if } f \text{ is even and } f/2 \text{ is odd, and } 2 \text{ splits completely in } K,\\
f\OO_K & \text{ otherwise.}
\end{cases}
$$
In particular, $\operatorname{Cond}(K(j_{K,f})/K)=(f/k)\OO_K$ with $k=1,2$, or $3$. 
\item The primes that ramify in $\Q(j_{K,f})/\Q$ divide the quantity $f\Delta_K$, and if $p>2$ is a prime with $p|f$ and $\gcd(p,\Delta_K)=1$, then $p$ ramifies in $\Q(j(E))/\Q$.

\item If $p$ does not divide $2f\Delta_K$, then the representation
$$\chi_{E,p^\infty}|_{H_f} \colon \Gal(\overline{H_f}/H_f) \to \Aut(T_p(E))\to \Z_p^\times$$
  given by $\chi_{E,p^\infty}|_{H_f}=\det(\rho_{E,p^\infty}|_{H_f})$, coincides with the $p$-adic cyclotomic character of $\overline{H_f}/H_f$, and it is surjective onto $\Z_p^\times$. 
\end{enumerate}
\end{lemma}
\begin{proof} Part (1) is Exercise 9.20 in \cite{cox}. For part (2), we note that the primes that ramify in $H_f/\Q$ either ramify in $K/\Q$ or $H_f/K$, and therefore they divide $f\Delta_K$ by part (1). If $p>2$ with $p|f$ and $\gcd(p,\Delta_K)=1$, then $\operatorname{Cond}(K(j_{K,f})/K)=(f/2)\OK$ or $f\OK$, so $p$ ramifies in $K(j_{K,f})/K$, and therefore in $K(j_{K,f})/\Q$. Since $K/\Q$ is unramified at $p$, it follows that $K(j_{K,f})/\Q(j_{K,f})$ is unramified at $p$, and therefore the ramification occurs in $\Q(j_{K,f})/\Q$, as claimed.
	
	For part (3), the fact that $\chi_{E,p^\infty}=\det(\rho_{E,p^\infty})$ coincides with the $p$-adic cyclotomic character is a well known consequence of the existence of the Weil pairing. Moreover, if $F$ is a number field, then the $p$-adic cyclotomic character $\Gal(\overline{F}/F)\to \Z_p^\times$ is surjective if and only if $F\cap \Q(\mu_{p^\infty})=\Q$. If $p$ does not divide $2f\Delta_K$, then $p$ is unramified in $H_f/\Q$, but the cyclotomic extension $\Q(\mu_{p^\infty})/\Q$ is totally ramified at $p$, and therefore $H_f\cap \Q(\mu_{p^\infty})=\Q$, as desired. 
\end{proof} 

\begin{example}
For instance, let $j_{K,f}$ be the $j$-invariant of an elliptic curve with CM by the order $\Of$ of $K=\Q(\sqrt{-3})$ of conductor $f$, with $f=2$, $5$ or $7$. Then:
$$\Q(j_{K,2})=\Q,\quad \Q(j_{K,5}) = \Q(\sqrt{5}) \quad \text{ and } \quad \Q(j_{K,7})=\Q(\sqrt{21}).$$
Notice that $3=|\Delta_K|$ does not divide the discriminant of $\Q(j_{K,5})$, but it does divide the discriminant of $\Q(j_{K,7})$. Also note that $2$ is not ramified in $\Q(j_{K,2})=\Q$.
\end{example}

We shall need a result about the existence  (or lack thereof) of abelian subfields of $\Q(j_{K,f})$. First, we recall the notion of generalized dihedral extensions: let $K$ be an imaginary quadratic field, and let $F/K$ be an abelian extension, such that $F/\Q$ is Galois. Then, we say that $F$ is generalized dihedral over $\Q$ if $\Gal(F/\Q)\cong \Gal(F/K)\rtimes (\Z/2\Z)\cong \Gal(F/K) \rtimes \langle \tau \rangle $
such that $\tau \sigma \tau^{-1} = \sigma^{-1}$ for any $\sigma \in \Gal(F/K)$.   

\begin{thm}[\cite{cox}, Theorem 9.18]\label{Cox1}
	Let $K$ be an imaginary quadratic field. Then, an abelian extension $F$ of $K$ is generalized dihedral over $\Q$ if and only if $F$ is contained in a ring class field of $K$.
\end{thm}

\begin{cor}\label{cor-notgalois}
	Let $L$ be a field of odd degree $d>1$ such that $L\subseteq \Q(j_{K,f})$. Then, $L/\Q$ is not Galois.
\end{cor}
\begin{proof}
	Suppose $L/\Q$ is an extension of odd degree $d>1$, and $L\subseteq \Q(j_{K,f})$. Then, the compositum $LK$ is contained in the ring class field $K(j_{K,f})$ and, since by Theorem \ref{thm-schertz} the extension $K(j_{K,f})/K$ is abelian, it follows that $LK/K$ is abelian as well, because $K\subseteq LK \subseteq K(j_{K,f})$. Hence, Theorem \ref{Cox1} implies that $LK$ is generalized dihedral over $\Q$. In other words,
	$\Gal(LK/\Q)\cong  \Gal(LK/K) \rtimes \langle \tau \rangle $
	such that $\tau \sigma \tau^{-1} = \sigma^{-1}$ for any $\sigma \in \Gal(LK/K)$. Now, $L/\Q$ is Galois if and only if $\Gal(LK/L)$ is normal in $\Gal(LK/\Q)$, and Proposition 11 in Section 5.5 of \cite{dummitfoote} says that $\Gal(LK/L)$ is normal in $\Gal(LK/\Q)$ if and only if the semi-direct action is trivial, i.e., $\sigma=\sigma^{-1}$ for all $\sigma\in \Gal(LK/K)$, which means that every non-trivial element of $\Gal(LK/K)$ has order $2$. However, if $L/\Q$ was Galois, then $\Gal(LK/K)\cong \Gal(L/\Q)$ would be of odd order since  $[L:\Q]$ is odd, and no element of $\Gal(LK/K)$ would have order $2$. Thus, we reach a contradiction and $L/\Q$ cannot be Galois. 
\end{proof}

We are ready to prove part (3) of Theorem \ref{thm-largeimage-intro}.

\begin{prop}\label{prop-imageindex2}\label{prop-largeimage-intro-part-3}
	Let $p>2$ and let $\chi_{E,p^\infty}=\det(\rho_{E,p^\infty})\colon \Gal(\overline{\Q(j_{K,f})}/\Q(j_{K,f}))\to \Z_p^\times$. Then, $\chi_{E,p^\infty}$ is the $p$-adic cyclotomic character, and the index of the image in $\Z_p^\times$ is $1$ or $2$. Further, the index is $2$ if and only if $p\equiv 1 \bmod 4$ and $\Q(\sqrt{p})\subseteq \Q(j_{K,f})$. In particular, $\chi_{E,p^\infty}$ is surjective onto $\Z_p^\times$ for all but finitely many primes $p$.
\end{prop}
\begin{proof}
	Let $J$ be the image  of $\chi_{E,p^\infty}=\det(\rho_{E,p^\infty})$.  	The character $\chi_{E,p^\infty}$ coincides with the $p$-adic cyclotomic character of $\overline{\Q(j_{K,f})}/\Q(j_{K,f})$. Thus, if $[\Z_p^\times: J] = d$, then $L=\Q(j_{K,f})\cap \Q(\zeta_p^\infty)$ is an abelian extension of $\Q$ of degree $d$. By Cor. \ref{cor-notgalois}, it follows that $[L:\Q]=d$ must be a power of $2$. We may assume that $d\geq 2$, so in particular it follows that $d$ is even, and therefore $[\Z_p^\times:J]$ is even. Since $p>2$, the group $\Z_p^\times$ is (topologically) cyclic. In particular, since $[\Z_p^\times : J]$ is even, it follows that $J\subseteq {\Z_p^\times}^2$. 
	Also, we conclude that $J$ is of even index in $\Z_p^\times$ if and only if $\Q(j_{K,f})\cap \Q(\zeta_p^\infty)$ contains a quadratic field. There is a unique quadratic field contained in the $p$-th cyclotomic extension, and therefore $\Q(j_{K,f})\cap \Q(\zeta_p)=\Q\left(\sqrt{(-1)^{\frac{p-1}{2}}p}\right)$. Notice, however, that $j_{K,f}$ has a real conjugate and $\Q(j_{K,f})$ has a real embedding (see, for instance, \cite{cox}, p. 242). Thus, if $p\equiv 3\bmod 4$, then we have reached a contradiction. Thus, $J$ is of even index if and only if $p\equiv 1 \bmod 4$ and $\Q(\sqrt{p})\subseteq \Q(j_{K,f})$, as claimed. 
	
	Suppose that $d\geq 2$ and so $J\subseteq {\Z_p^\times}^2$. By our previous remarks, $d$ is a power of $2$, and  $p\equiv 1 \bmod 4$ with $\Q(\sqrt{p})\subseteq \Q(j_{K,f})$, which implies that $j_{K,f}\neq 0,1728$. Since $\gcd(p,d)=1$, it follows that if we let $J_{E,1}$ be the image of $\chi_{E,p}=\det(\rho_{E,p})$, then $[(\Z/p\Z)^\times:J_1]=d$. By Theorem \ref{thm-twistforfullimage}, there is an elliptic curve $E'/\Q(j_{K,f})$ such that $\Gal(H_f(E'[p])/H_f) \cong (\Of/p\Of)^\times$, where $H_f=K(j_{K,f})$. Since $(\Z/p\Z)^\times \hookrightarrow (\Of/p\Of)^\times$ corresponds to the subgroup of all scalar matrices of $\GL(2,\F_p)$, it follows that ${(\Z/p\Z)^\times}^2 \subseteq J_{E',1}$, where $J_{E',1}$ is the image of $\chi_{E',p}=\det(\rho_{E',p})$. But $j_{K,f}\neq 0,1728$, and by the construction in the proof of Theorem \ref{thm-twistforfullimage}, the curve $E'/\Q(j_{K,f})$ is a quadratic twist of $E/\Q(j_{K,f})$. Therefore, the images of $\chi_{E',p}=\det(\rho_{E',p})$ and $\chi_{E,p}=\det(\rho_{E,p})$ coincide. Hence, we have inclusions ${(\Z/p\Z)^\times}^2\subseteq J_{E,1}\subseteq {(\Z/p\Z)^\times}^2$ and so $d=2$. Hence,  we have $d=1$ or $2$ in all cases.

	 However, $\Q(j_{K,f})/\Q$ is a finite extension and so it contains only a finite number of quadratic fields. Hence, the index of the image of $\chi_{E,p^\infty}$ is $1$ for all but finitely many primes $p$, as desired.
\end{proof}

\begin{example}
	For instance, let $K=\Q(\sqrt{-11})$ and $f=5$. Then, the order $\Of$ is of class number $h(\Of)=4$, and $K(j_{K,f})/\Q$ is a dihedral extension with $\Gal(K(j_{K,f})/\Q)\cong D_8$, where $D_8$ is the dihedral group of order $8$. The field $\Q(j_{K,f})$ is generated by a root of the polynomial
	$$x^4 - x^3 + x^2 + 4x - 4$$
	and it contains a unique quadratic subfield, namely $\Q(\sqrt{5})$ (the extension $\Q(j_{K,f})/\Q$ is totally ramified at $5$). Since $K(j_{K,f})/\Q$ is dihedral, we must have $\Q(j_{K,f})\cap \Q(\zeta_5)=\Q(\sqrt{5})$. Thus, the determinant of the representation $\rho_{E,5}:\Gal(\overline{\Q(j_{K,f})}/\Q(j_{K,f}))\to \Aut(E[5])$ must have $((\Z/5\Z)^\times)^2$ as image.
\end{example}

\subsection{Serre-Tate} In this section we recall some theorems of Serre and Tate (who, in turn, credit Shimura and Taniyama, and Weil) as they appear in \cite{serretate}. Let $E$ be an elliptic curve with complex multiplication and $j$-invariant $j_{K,f}$, that is defined over $\Q(j_{K,f})$, and suppose that $E$ has CM by the order $\Of$ in an imaginary quadratic field $K$. By Theorem \ref{thm-endo}, all endomorphisms $\phi\in \operatorname{End}(E)\cong \Of$  are defined over $H_f=K(j_{K,f})$. For $p$ a prime, the $p$-adic Tate module of $E$, denoted by $T_p(E)$ is a free module over $\OO_p=\Of\otimes_\Z \Z_p$ of rank $1$ by \cite{serretate}, Theorem 5. In particular, $\Aut_{\OO_p}(T_p(E))\cong \mathcal{O}_p^\times$. Since all the endomorphisms are defined over $H_f$, the actions of $\Gal(\overline{H_f}/H_f)$ and $\mathcal{O}_p$ commute. Thus, the natural Galois action produces a representation
$$\rho_{E,p^\infty}|_{H_f}\colon \Gal(\overline{H_f}/H_f) \to \Aut_{\OO_p}(T_p(E))\cong \mathcal{O}_p^\times.$$
Since $\OO_p^\times$ is abelian, we can factor $\rho_{E,p^\infty}|_{H_f}$ through the maximal abelian quotient $\Gal(H_f^\text{ab}/H_f)$ of the Galois group $\Gal(\overline{H_f}/H_f)$. Let $\mathbb{A}_{H_f}^\times$ be the group of ideles of $H_f$. Composing with the global reciprocity map $\mathbb{A}_{H_f}^\times\to \Gal(H_f^\text{ab}/H_f)$ of class field theory, we obtain a continuous representation $P_{E,p^\infty}:\mathbb{A}_{H_f}^\times \to \OO_p^\times$. By the theory of complex multiplication (see \S 6 and \S 7 of \cite{serretate}, in particular, Theorems 10 and 11 and Corollary 1) there is a continuous homomorphism $\varepsilon: \mathbb{A}_{H_f}^\times \to K^\times$ with properties as described in the next result. 

\begin{thm}[\cite{serretate}, \S 7]\label{thm-serretate}
Let $E/\Q(j_{K,f})$ be an elliptic curve with complex multiplication by an order $\Of$ of an imaginary quadratic field $K$, let $H_f =K(j_{K,f})$, and let $\mathbb{A}_{H_f}^\times$ be the group of id\`eles of $H_f$. Then, there exists  a continuous homomorphism $\varepsilon: \mathbb{A}_{H_f}^\times \to K^\times$ such that
\begin{enumerate}
\item The restriction of $\varepsilon$ to $H_f^\times$ is the norm map from $H_f^\times$ down to $K^\times$, i.e., $\varepsilon|_{H_f^\times} = N_{H_f/K}$, where $H_f^\times$ is diagonally embedded into $\mathbb{A}_{H_f}^\times$.
\item We have $P_{E,p^\infty}(a)=\varepsilon(a)N_{H_{f,p}/K_p}(a_p^{-1})$ for every $a\in \mathbb{A}_{H_f}^\times$, where $a_p$ denotes the component of the id\`ele $a$ in the group $H_{f,p}^\times = (\Q_p\otimes H_f)^\times = \prod_{\nu| p} H_{f,\nu}^\times$.
\item The elliptic curve $E$ has good reduction at a place $\nu$ of $H_f$ if and only if $\varepsilon$ is unramified at $\nu$, i.e., $\varepsilon(\OO_{H_{f,\nu}}^\times)=\{1\}$ where $  \OO_{H_{f,\nu}}^\times$ is identified with the subgroup of $\mathbb{A}_{H_f}^\times$ whose components are $1$ except at the place $\nu$, where the values are in $  \OO_{H_{f,\nu}}^\times$, the units in the ring of integers of $H_{f,\nu}$.
\end{enumerate}
\end{thm}

We identify $H_{f,p}^\times$ with the subgroup of $\mathbb{A}_{H_f}^\times$ whose components are $1$ except at the places $\nu$ dividing $p$. Then, if $E$ has good reduction at primes of $H_f$ above $p$, then Theorem \ref{thm-serretate} and the fact that the norm map $N_{H_{f,\nu}/K_\wp}$ is surjective for primes $\nu|\wp|p$ that are unramified (\cite{lang}, II, \S 4, Corollary), imply that $\rho_{E,p^\infty}|_{H_f}$ is surjective when the primes of $H_f$ above $p$ are of good reduction. This is sufficient to prove that $\rho_{E,p^\infty}$ is surjective as well.

\begin{cor}\label{cor-surjective}
	Let $E/\Q(j_{K,f})$ be an elliptic curve with complex multiplication by an order $\Of$ of an imaginary quadratic field $K$, let $H_f =K(j_{K,f})$, and let $\rho_{E,p^\infty}:\Gal(\overline{H_f}/H_f) \to \Aut_{\OO_p}(T_p(E))\cong \mathcal{O}_p^\times$ be defined as above. Let $p$ be a prime such that $f$ is not divisible by $p$, and the primes of $H_f$ above $p$ are of good reduction. Then, $\rho_{E,p^\infty}|_{H_f}$ is surjective. Hence, there is a $\Z_p$-basis of $T_p(E)$ such that the image of $\rho_{E,p^\infty}\colon \Gal(\overline{\Q(j_{K,f})}/\Q(j_{K,f}))\to \GL(2,\Z_p)$ is precisely $\mathcal{N}_{\delta,\phi}(p^\infty)$.
\end{cor}
\begin{proof}
	It suffices to prove that $P_{E,p^\infty}$ is surjective. Since the primes of $H_f$ above $p$ are of good reduction, Theorem \ref{thm-serretate}, part (3), says that $\varepsilon(\OO_{H_{f,p}}^\times)=1$. Moreover, if $f$ is not divisible by $p$, then Lemma \ref{lem-odd2} implies that $H_f/K$ is unramified at primes above $p$. Thus, the norm maps $N_{H_{f,\nu}/K_\wp}$, for primes for primes $\nu|\wp|p$, are surjective, and therefore  $N_{H_{f,p}/K_p}(\OO_{H_{f,p}}^\times) = \OO_p^\times$. Hence, $P_{E,p^\infty}$ is surjective, and thus $\rho_{E,p^\infty}|_{H_f}$ is surjective as well.
	
	For the last part of the statement, let us choose a $\Z_p$-basis of $\rho_{E,p^\infty}$ as in Theorem \ref{thm-cmrep} such that the image of $\Gal(\overline{H_f}/H_f)$ is contained in $\cC_{\delta,\phi}(p^\infty)\cong \OO_p^\times$, and the image of $\rho_{E,p^\infty}$ is contained in $\mathcal{N}_{\delta,\phi}(p^\infty)$. By the first part of the result, if $f$ is not divisible by $p$ and the primes of $H_f$ above $p$ are of good reduction, then the image of $\Gal(\overline{H_f}/H_f)$ is all of $\cC_{\delta,\phi}(p^\infty)\cong \OO_p^\times$. Since the image of complex conjugation cannot be contained in the Cartan subgroup, by Theorem \ref{lem-insidecartan}, it follows that $\rho_{E,p^\infty}$ surjects onto $\mathcal{N}_{\delta,\phi}(p^\infty)$, as claimed.
\end{proof}

Finally, we cite a result that follows from Theorems 8 and 9 of \S 6 of \cite{serretate}. 

\begin{thm}[\cite{serretate}, \S 6, Cor. 1]\label{thm-allgood}
	Let $E/\Q(j_{K,f})$ be an elliptic curve with complex multiplication by an order $\Of$ of an imaginary quadratic field $K$, let $H_f =K(j_{K,f})$, and let $p$ be a prime such that if $K=\Q(i)$ or $\Q(\sqrt{-3})$, then $f\neq p^n$ with $n\geq 1$. Then, there is a another elliptic curve $E'$, defined over $H_f$ and isomorphic to $E/H_f$, such that $E'$ has good reduction at the primes above $p$.
\end{thm}

\subsection{Proof of part (4) of Theorem \ref{thm-largeimage-intro}}

Recall that we have proved parts (1) and (2), and (3) of Theorem \ref{thm-largeimage-intro} in Section \ref{sec-proofsparts1and3} and Proposition \ref{prop-largeimage-intro-part-3}, respectively. In this section we conclude the proof of (4). We first need a lemma.

\begin{lemma}
	Let $\Of$ be an imaginary quadratic field $K$, of discriminant $\Delta_K$, and let $p$ be an odd prime that does not divide $f\Delta_K$. Let $G=(\Of/p\Of)^\times \rtimes_\varphi \langle c \rangle$, with $c$ an element of order $2$, where $\varphi: \langle c \rangle \to \Aut((\Of/p\Of)^\times)$ is given by $c\mapsto (\beta\mapsto c(\beta))$. Suppose $H$ is a subgroup of $\cC=(\Of/p\Of)^\times$ such that $\pm H =\cC$. Then, either $H=\cC$, or the subgroup of $G$ generated by elements of $H$ together with $c$ is all of $G$.
\end{lemma}
\begin{proof} 
Let $p$ be an odd prime not dividing $f\Delta_K$, let $\cC = (\Of/p\Of)^\times$ and suppose that there is a subgroup $H\subsetneq \mathcal{C}$, such that $\pm H = \mathcal{C}$ (therefore $H$ is of index $2$ in $\cC$), and such that the subgroup generated by $H$ and $c$ in $(\Of/p\Of)^\times \rtimes_\varphi \langle c \rangle$ is not the whole group. Since $p$ does not divide $f\Delta_K$, the prime $p$ is unramified in $K$. We distinguish two cases according to whether $p$ is inert or split.
\begin{itemize}
	\item If $p$ is inert in $K/\Q$, then 
	$$\cC = (\Of/p\Of)^\times \cong (\OO_K/p\OO_K)^\times$$
	which is cyclic of degree $p^2-1$. Thus, there is a unique subgroup of index $2$, namely $H=\cC^2$. But $p$ is odd, so $(p^2-1)/2$ is even, and if $\cC = \langle g\rangle$, then $g^{(p^2-1)/2}=-1$ and $(g^{(p^2-1)/4})^2 =- 1$. Hence, $-1$ is a square and therefore $-1\in H$. Hence, $\pm H = H \neq \cC$, in  contradiction with our assumption $\pm H=\cC$.
	
	\item If $p$ is split in $K/\Q$, then 
	$$\cC = (\Of/p\Of)^\times \cong (\OO_K/p\OO_K)^\times \cong (\OO_K/\wp)^\times \times (\OO_K/c(\wp))^\times$$
	where $p\OO_K = \wp\cdot c(\wp)$. Since $p$ is odd, $(\OO_K/\wp)^\times$ and $(\OO_K/c(\wp))^\times$ are of cyclic of (even) order $p-1$, and so $\cC\cong (\Z/p\Z)^\times \times (\Z/p\Z)^\times$. Thus, there are three different subgroups of index $2$ in $\cC$, namely the kernels $H_i$, for $i=1,2,3$, of the three maps $\psi_i:\cC\to \{\pm 1\}$ given by
	$$(a,b)\mapsto \left(\frac{a}{p}\right),\ (a,b)\mapsto \left(\frac{b}{p}\right),\ (a,b)\mapsto \left(\frac{ab}{p}\right).$$
	Note further that $c$ acts on a pair $(a,b)\in (\Z/p\Z)^\times \times (\Z/p\Z)^\times \cong \cC$ by $c(a,b)=(b,a)$. Therefore $H=H_i$, with $i=1$ or $2$, together with $c$, generate all of $\cC\rtimes \Z/2\Z$ in contradiction of our assumption. Finally $\pm H_3 = H_3$ because if $(a,b)\in H_3$, then $(-a,-b)\in H_3$ also. Thus, $\pm H_3=H_3\neq \cC$ again in contradiction with our assumption.
\end{itemize}
Hence, we must have that $H=\cC$, as desired.
\end{proof} 

The following result relates the images of $\rho_{E,p}$ and $\rho_{E',p}$, when $E$ and $E'$ are isomorphic over $\overline{\Q}$.

\begin{thm}\label{thm-twists}
	Let $E$ and $E'$ be two elliptic curves which are isomorphic over $\overline{\mathbb{Q}}$, such that both curves have complex multiplication by an order $\Of$ of an imaginary quadratic field $K$, and that both are defined over $H_f=K(j_{K,f})$. Let $p$ be a prime that does not divide $2f\Delta_K$, and let $G_{E,p}$ and $G_{E',p}$ be respectively the images of the mod $p$ representations $\rho_{E,p}$ and $\rho_{E',p}\colon \Gal(\overline{H_f}/H_f)\to (\Of/p\Of)^\times$, and assume that $\rho_{E',p}$ is surjective. Then:
	\begin{enumerate}
		\item If $j_{K,f}=j(E)=j(E')\neq 0$, then $G_{E,p}=G_{E',p}=(\Of/p\Of)^\times$. In other words, both $\rho_{E,p}$ and $\rho_{E',p}$ are surjective.
		
		\item Suppose $j_{K,f}=j(E)=j(E')=0$, so that $E$ and $E'$ have CM by $\OK$ and $K=\Q(\sqrt{-3})$. 
		\begin{enumerate}
			\item If $p\equiv \pm 1\bmod 9$, then $G_{E,p}=G_{E',p}=(\OK/p\OK)^\times$. Moreover, if $p\equiv 1\bmod 9$, then the image is a split Cartan subgroup, and if $p\equiv -1\bmod 9$, then the image is a non-split Cartan subgroup.
			\item If $p\equiv 2$ or $5\bmod 9$, then $G_{E',p}$ is a non-split Cartan subgroup and either $G_{E,p}=G_{E',p}$, or $G_{E,p}\cong ((\OK/p\OK)^\times)^3$, the cubes of units in $\OK/p\OK$.
			\item If $p\equiv 4$ or $7\bmod 9$, then $G_{E',p}\cong (\OO_K/\wp)^\times \times (\OO_K/c(\wp))^\times$ and either $G_{E,p}=G_{E',p}$, or $G_{E,p}$ is isomorphic to the subgroup $\{(a,b)\in G_{E',p}: a/b \in ((\Z/p\Z)^\times)^3\}$ of $G_{E',p}$.
		\end{enumerate}
	\end{enumerate}
\end{thm}
\begin{proof} 
Let $E$ and $E'$ be two elliptic curves which are isomorphic over $\overline{\mathbb{Q}}$. Thus, $j_{K,f}=j(E)=j(E')$. Assume, further, that both $E$ and $E'$ have complex multiplication by an order $\Of$ of an imaginary quadratic field $K$, and that both are defined over $H_f=K(j_{K,f})$. We shall assume that $E$ and $E'$ are not isomorphic over $H_f$. 

Let us further assume that the Galois group of $H_f(E'[p])/H_f$ is a full Cartan group isomorphic to $G_{E',p}=(\Of/p\Of)^\times$, where $p$ is a prime that does not divide $2f\Delta_K$. In such case, $G_{E',p} = (\Of/p\Of)^\times \cong (\OO_K/p\OO_K)^\times$. We also assume that $G_{E,p}=\Gal(H_f(E[p])/H_f)\subsetneq (\Of/p\Of)^\times$ and $G_{E,p}$ is stable under the action of complex conjugation, i.e., $c(G_{E,p})=G_{E,p}$ (recall that $\Gal(H_f(E[p])/\Q(j_{K,f}))\subseteq (\Of/p\Of)^\times \rtimes_\varphi \langle c \rangle$ with $\varphi\colon \langle c\rangle \to \Aut((\Of/p\Of)^\times)$ is given by $c\mapsto (\beta\mapsto c(\beta))$.  Finally, we remark that $E$ and $E'$ are defined over $H_f$ but isomorphic over $\overline{\Q}$, and therefore their $p$-adic representations differ by a twist $\chi \colon \Gal(\overline{H_f}/H_f)\to \mu_n$ where $n$ divides $4$ or $6$ (see \cite[\S X.5]{silverman}). In other words, if we write $\rho_{E',p}\colon  \Gal(\overline{H_f}/H_f)\to \Aut_{\Of/p\Of}(E'[p])\cong (\Of/p\Of)^\times$, then  $\rho_{E',p} = \chi \cdot \rho_{E,p}$ and, in particular, $G_{E',p}$ is isomorphic to a subgroup of $\mu_n \cdot G_{E,p} \subseteq \Z/n\Z \times G_{E,p}$. Thus, by \cite{silverman}, Ch. I, Prop. 1.4, one of the following options holds:
\begin{itemize}
	\item If $j\neq 0$ or $1728$, then $E'$ is a quadratic twist of $E$. In particular,  $G_{E',p}=\mu_2 \cdot G_{E,p} = \pm G_{E,p}$.
	\item $j=1728$ and $E'$ is a quadratic or quartic twist of $E$. In particular, either $G_{E',p}=\pm G_{E,p}$ or $G_{E',p}=\mu_4 \cdot G_{E,p}$.
		\item $j=0$ and $E'$ is twist of $E$ of degree dividing $6$. In particular, either $G_{E',p}=\pm G_{E,p}$, or $G_{E',p}=\mu_3 \cdot G_{E,p}$, or $G_{E',p}=\mu_6 \cdot G_{E,p}$.
\end{itemize}

We distinguish two cases according to whether $[G_{E',p}:G_{E,p}]>1$ is even or odd.

\begin{enumerate}
	\item Suppose first that $[G_{E',p}:G_{E,p}]$ is even (i.e., $2$, $4$, or $6$). Let us first assume that $p$ is inert in $K$. Since $p$ does not divide $2f\Delta_K$, then $G_{E',p}\cong (\Of/p\Of)^\times \cong (\OO_K/p\OO_K)^\times$ is cyclic of order $p^2-1$. Since $p$ is odd, it follows that the $2$-primary component of $G_{E',p}$ is cyclic of order $2^k\geq 8$, for some $k\geq 3$. But $G_{E,p}$ is a subgroup of index $2$, $4$, or $6$, so its $2$-primary component is of order $2^{k-1}$ or $2^{k-2}$, which is at least of order $2$. Thus, neither one of  $\Z/2\Z \times G_{E,p}$,  $\Z/4\Z \times G_{E,p}$, nor $\Z/6\Z \times G_{E,p}$, can be cyclic, so we have reached a contradiction.
	
	Assume then that $p$ is split in $K$. Then, $$G_{E',p} = (\Of/p\Of)^\times \cong (\OO_K/p\OO_K)^\times \cong (\OO_K/\wp)^\times \times (\OO_K/c(\wp))^\times$$
	where $p\OO_K = \wp\cdot c(\wp)$, and $c$ denotes complex conjugation. In particular, the $2$-torsion of $G_{E',p}$ is given by $\{\pm 1 \}\times \{\pm 1 \}$, and since $E'$ is a twist of $E$, and the index $[G_{E',p}:G_{E,p}]$ is even, we must have $G_{E',p}[2]\cong \Z/2\Z \times G_{E,p}[2] = \pm G_{E,p}[2]$. Thus, $G_{E,p}[2]$ is a subgroup of index $2$ of $\{ \pm 1\}\times \{\pm 1 \}$. Since $c((-1,1))=(1,-1)$, and since we are assuming that $G_{E,p}$ is stable under complex conjugation, we must have $G_{E,p}[2]=\{(1,1),(-1,-1) \}$. If so, $\pm G_{E,p}[2] = G_{E,p}[2] \neq G_{E',p}[2]$ and we have reached a contradiction. Thus, in all cases an even index $[G_{E',p}:G_{E,p}]$ leads to a contradiction.
	
	\item Suppose that $[G_{E',p}:G_{E,p}]>1$ is odd, which means equal to $3$ since the index divides $4$ or $6$. Then, this means that the order of $\Of^\times$ is divisible by $3$, and therefore $K=\Q(\sqrt{-3})$ and $f=1$ (see Lemma \ref{lem-unitsmod}), and $j_{K,f}=j(E)=j(E')=0$. In particular, if $p$ does not divide $2f\Delta_K$, then $p$ is inert if and only if $\left(\frac{-3}{p}\right)=-1$, if and only if $p\equiv 2\bmod 3$. Let us first assume that $p$ is inert in $K$, and so $p\equiv 2 \bmod 3$. Then, $p\equiv 2$, $5$, or $8\bmod 9$. Since $G_{E',p}$ is cyclic of order $p^2-1$, it follows that if $p\equiv 8 \bmod 9$, then $G_{E',p}$ cannot possibly have a subgroup $G_{E,p}$ of index $3$, because $\Z/3\Z \times G_{E,p}$ is not cyclic. On the contrary, if $p\equiv 2$ or $5\bmod 9$, then $p^2-1\equiv 3$ or $6\bmod 9$, and $G_{E',p}$ has a subgroup of index $3$ such that $G_{E',p}\cong \Z/3\Z \times G_{E,p}$, namely $G_{E,p}\cong ((\Of/p\Of)^\times)^3$.
	
	Now assume that $p$ is split in $K$, and so $p\equiv 1$, $4$, or $7\bmod 9$. As before, $G_{E',p} =  (\OO_K/\wp)^\times \times (\OO_K/c(\wp))^\times \cong (\Z/p\Z)^\times\times (\Z/p\Z)^\times$. In particular, if $p\equiv 1\bmod 9$, then there are two copies of cyclic groups of order $3^n\geq 9$, and so if $G_{E,p}$ is a group of index $3$, then $\Z/3\Z\times G_{E,p}$ cannot be isomorphic to $G_{E',p}$. If $p\equiv 4$ or $7\bmod 9$, then there are $4$ possible subgroups $G_{E,p}$ of $G_{E',p}$ of index $3$, namely the subgroups of $(\Z/p\Z)^\times\times (\Z/p\Z)^\times\to (\Z/p\Z)^\times$ of the form $H_1=((\Z/p\Z)^\times)^3\times (\Z/p\Z)^\times$, $H_2=(\Z/p\Z)^\times\times ((\Z/p\Z)^\times)^3$, and
	$$H_3=\{(a,b): a\cdot b \in ((\Z/p\Z)^\times)^3 \},\ H_4=\{(a,b): a/b \in ((\Z/p\Z)^\times)^3 \}.$$
	Since $G_{E,p}$ is supposed to be stable under complex conjugation, $H_1$ and $H_2$ are impossible. Moreover, $\det(G_{E,p}\rtimes \langle c \rangle)=(\Z/p\Z)^\times$, because $p$ does not divide $2f\Delta_K$, and therefore $\det(\rho_{E,p})$ is surjective by Lemma \ref{lem-odd2}. The determinant function on $(a,b)\in \Aut_{\OK/p\OK}(E[p])\cong  (\OO_K/p\OO_K)^\times \cong (\OO_K/\wp)^\times \times (\OO_K/c(\wp))^\times$ is given by $(a,b)\mapsto ab$ because the pair $(a,b)$ gives the action on the submodules $E[\wp\cap \Of]$ and $E[c(\wp)\cap \Of]$ respectively. It follows that if $G_{E,p}\cong H_3$, then the determinant of $G_{E,p}\rtimes \langle c\rangle$ would be the subgroup of cubes (since $p\equiv 1,4,7\bmod 9$, the subgroup of cubes is strictly smaller than $(\Z/p\Z)^\times$), a contradiction. Thus, the image must be $G_{E,p}\cong H_4$.
	
\end{enumerate}
This completes the proof of Theorem \ref{thm-twists}.
\end{proof} 

As a corollary we obtain the following result for $j\neq 0$ (or $j=0$ and $p\neq \pm 1 \bmod 9$), which completes  the proof of (part (4) of) Theorem \ref{thm-largeimage-intro}.

\begin{thm}\label{thm-largeimage-intro-part4}
	Let $E/\Q(j_{K,f})$ be an elliptic curve with CM by an order $\Of$ of conductor $f$ in an imaginary quadratic field $K$ of discriminant $\Delta_K$, and let $p$ be an odd prime that does not divide $f\Delta_K$. Further, assume that $j_{K,f}\neq 0$, or $p\equiv \pm 1 \bmod 9$. Then, if $c\in \Gal(\overline{\Q(j_{K,f})}/\Q(j_{K,f}))$ is a fixed complex conjugation, and we fix $\varepsilon \in \{ \pm 1 \}$, then there is a $\Z_p$-basis of $T_p(E)$ such that the image of $\rho_{E,p^\infty}$ is equal to  $\mathcal{N}_{\delta,\phi}(p^\infty)$ and  $\rho_{E,p^\infty}(c)=c_\varepsilon=\left(\begin{array}{cc} \varepsilon & 0\\ 0 & -\varepsilon\\\end{array}\right)$.
\end{thm}

\begin{proof}
Let $E$ and $p$ be an elliptic curve and a prime as in the statement of the theorem, and consider $E/H_f$, with $H_f=K(j_{K,f})$. Since $p$ does not divide $f$, the conductor $f$ is not a power of $p$. Under these hypotheses, Theorem \ref{thm-allgood} says that there is another elliptic curve $E'/H_f$, such that $E'$ has good reduction at the primes above $p$. Now Corollary \ref{cor-surjective} says that $\rho_{E',p^\infty}|_{H_f}\colon \Gal(\overline{H_f}/H_f)\to \Aut_{\OO_{f,p}}(T_p(E'))\cong \mathcal{O}_{f,p}^\times$ is surjective, where $\OO_{f,p}\cong \Of\otimes \Z_p$. In particular, the representation $\rho_{E',p}|_{H_f}$ is surjective, with image $G_{E',p}\cong (\Of/p\Of)^\times$. Since $j\neq 0$ or $p\equiv \pm 1 \bmod 9$, Theorem \ref{thm-twists} implies that $\rho_{E,p}|_{H_f}$ is also surjective. Then, by Theorem \ref{thm-inclp}, the image of $\rho_{E,p^n}$ is $(\Of/p^n\Of)^\times$, and so the $p$-adic representation $\rho_{E,p^\infty}|_{H_f}: \Gal(\overline{H_f}/H_f)\to \Aut_{\OO_{f,p}}(T_p(E))\cong \mathcal{O}_{f,p}^\times$ is surjective. Let us choose a $\Z_p$-basis of $\rho_{E,p^\infty}$ as in Theorem \ref{thm-cmrep} such that the image of $\Gal(\overline{H_f}/H_f)$ is contained in $\cC_{\delta,\phi}(p^\infty)\cong \OO_p^\times$, and the image of $\rho_{E,p^\infty}$ is contained in $\mathcal{N}_{\delta,\phi}(p^\infty)$. Since we have shown that the image of $\Gal(\overline{H_f}/H_f)$ is all of $\cC_{\delta,\phi}(p^\infty)\cong \OO_p^\times$, and the image of complex conjugation cannot be contained in the Cartan subgroup, by Theorem \ref{lem-insidecartan}, it follows that  $\rho_{E,p^\infty}$ surjects onto $\mathcal{N}_{\delta,\phi}(p^\infty)$, as claimed. Finally, since $p>2$ and  $\delta=\Delta_K f^2/4\not\equiv 0 \bmod p$ by assumption, Lemma \ref{lem-ccfinal} says that the $\Z_p$-basis can be chosen so that the image of $\rho_{E,p^\infty}|_{H_f}$ is $\cC_{\delta,0}(p^\infty)$ and $\rho_{E,p^\infty}(c)=c_\varepsilon$, as desired.
\end{proof}

We conclude this section with a proof of Theorem \ref{thm-jzero-goodredn}.

\subsection{Proof of Theorem \ref{thm-jzero-goodredn}}\label{sec-proof-of-jzero-goodredn}

 Let $E/\Q$ be an elliptic curve with $j(E)=0$. Let $K=\Q(\sqrt{-3})$, $f=1$, and $\OO_{K,f} = \OO_K = \Z[(1+\sqrt{-3})/2]$, and $\delta=\Delta_K f^2/4=-3/4$. Let $p>3$ be a prime, and let $c \in \GQ$ be a complex conjugation. Since $p$ does not divide $2f\Delta_K$, Lemma \ref{lem-ccfinal} shows that there is a $\Z_p$-basis of $T_p(E)$ such that the image $G_{E,p^\infty}$ of $\rho_{E,p^\infty}\colon \GQ \to \Aut(T_p(E))\cong \GL(2,\Z_p)$ is contained in $\mathcal{N}_{\delta,0}(p^\infty)$ and $\rho_{E,p^\infty}(c)=c_\varepsilon = \left(\begin{array}{cc} \varepsilon & 0\\ 0 & -\varepsilon\\ \end{array}\right)$. Since $p>3$,  part (2) of  Theorem \ref{thm-largeimage-intro-parts-1-and-2} shows that $G_{E,p^\infty}$ is the full inverse image via the natural reduction mod $p$ map  $\mathcal{N}_{\delta,0}(p^\infty)\to \mathcal{N}_{\delta,0}(p)$ of the image $G_{E,p}$ of $\rho_{E,p}\equiv \rho_{E,p^\infty}\bmod p$, and by part (1), the curve $E$ is a twist of another curve $E'$ with CM by $\OK$ such that $G_{E,p^\infty}= \mathcal{N}_{\delta,0}(p^\infty)$. Then, by Theorem \ref{thm-twists}, there are three possibilities for the image of $\rho_{E,p}$. We write $G_{E,K,p}$ for the image of $\rho_{E,p}|_{H_f}$.

	\begin{enumerate}
	\item If $p\equiv \pm 1\bmod 9$, then $G_{E,K,p}\cong G_{E',K,p}\cong (\OK/p\OK)^\times$, and therefore $G_{E,K,p}\cong \cC_{\delta,0}(p)$, and $G_{E,p}\cong \mathcal{N}_{\delta,0}(p)$. Thus, by part (2) Theorem \ref{thm-largeimage-intro}, the image of $\rho_{E,p^\infty}$ is all of $\mathcal{N}_{\delta,0}(p^\infty)$.  Moreover, if $p\equiv 1\bmod 9$, then the image is a split Cartan subgroup, and if $p\equiv -1\bmod 9$, then the image is a non-split Cartan subgroup.
	\item If $p\equiv 2$ or $5\bmod 9$, then $G_{E',p}$ is a non-split Cartan subgroup and either $G_{E,p}=G_{E',p}$, or $G_{E,p}\cong ((\OK/p\OK)^\times)^3$, the cubes of units in $\OK/p\OK$. It follows that $G_{E,p}$ is contained in the normalizer of a non-split Cartan subgroup and either $G_{E,p}=\mathcal{N}_{\delta,0}(p)$, or $G_{E,p}=\langle \cC_{\delta,0}(p)^3, c_\varepsilon \rangle $.
	\item If $p\equiv 4$ or $7\bmod 9$, then $G_{E',p}\cong (\OO_K/\wp)^\times \times (\OO_K/c(\wp))^\times$ and either $G_{E,p}=G_{E',p}$, or $G_{E,p}$ is isomorphic to the subgroup $\{(a,b)\in G_{E',p}: a/b \in ((\Z/p\Z)^\times)^3\}$ of $G_{E',p}$. It follows that $G_{E,p}$ is contained in the normalizer of a split Cartan, and either $G_{E,p^\infty}=\mathcal{N}_{\delta,0}(p^\infty)$, or $G_{E,p^\infty}$ is isomorphic to the subgroup $$\left\langle \left\{\left(\begin{array}{cc} a & 0\\ 0 & b\\ \end{array}\right) : a/b \in (\Z_p^\times)^3\right\}, \gamma = \left(\begin{array}{cc} 0 & \varepsilon\\ \varepsilon & 0\\ \end{array}\right)\right\rangle.$$
\end{enumerate}

This concludes the proof of Theorem \ref{thm-jzero-goodredn}.

\section{Odd primes dividing the conductor or discriminant}\label{sec-badprimes}

In this section we study the image of $\rho_{E,p}:\Gal(\overline{\Q(j_{K,f})}/\Q(j_{K,f}))\to \GL(2,\Z/p\Z)$ when $p$ is an odd prime divisor of $f\Delta_K$, and then we use this knowledge and part (2) Theorem \ref{thm-largeimage-intro} to prove Theorem \ref{thm-badprimes-intro}.

\begin{thm}\label{thm-badprimes}
	Let $E/\Q(j_{K,f})$ be an elliptic curve with CM by an order $\Of$ of conductor $f\geq 1$ of an imaginary quadratic field $K$. Let $p$ be an odd prime dividing $f\Delta_K$ (thus, $j\neq 1728$ where $f\Delta_K=-4$), and let $G_{E,p}$ be the image of $\rho_{E,p}\colon\Gal(\overline{\Q(j_{K,f})}/\Q(j_{K,f}))\to \Aut(E[p])\cong \GL(2,\Z/p\Z)$. Let $c\in\Gal(\overline{\Q(j_{K,f})}/\Q(j_{K,f}))$ be a fixed complex conjugation. Then, there is a $\Z/p\Z$-basis of $E[p]$ such that $G_{E,p}\subseteq \mathcal{N}_{\delta,0}(p)$, with $\delta=\Delta_K f^2/4\equiv 0 \bmod p$, and $\rho_{E,p}(c)=c_\varepsilon=\left(\begin{array}{cc} \varepsilon & 0\\ 0 & -\varepsilon\\ \end{array}\right)$ for some $\varepsilon \in \{\pm 1 \}$. Moreover, $G_{E,p}$ is precisely one of the following groups of $\GL(2,\Z/p\Z)$:
	\begin{enumerate}
	\item[(a)] If $j\neq 0,1728$, then either $G_{E,p}=\mathcal{N}_{\delta,0}(p)$, or $G_{E,p}$ is generated by $c_{\varepsilon}$ and the group
	$$\left\{ \left(\begin{array}{cc} a & b\\ 0 & a\\ \end{array}\right): a\in ((\Z/p\Z)^\times)^2, b\in  \Z/p\Z\right\},$$
	where $\delta\equiv 0 \bmod p$.
	\item[(b)] If $j=0$, then $p=3$, and $G_{E,3}=\mathcal{N}_{\delta,0}(3)$, or $G_{E,3}$ is generated by $c_{\varepsilon}$ and one of the following groups
	$$\left\{ \left(\begin{array}{cc}  a & b\\ -3b/4 & a\\ \end{array}\right): a,b\in  \Z/3\Z,\ a\equiv 1 \bmod 3\right\}\ \text{ or } \left\{ \left(\begin{array}{cc} a & b\\ -3b/4 & c\\ \end{array}\right): a,c\in (\Z/3\Z)^\times,\ b\equiv 0 \bmod 3 \right\},$$
	$$\text{or }\ \left\{ \left(\begin{array}{cc}  a & b\\ -3b/4 & a\\ \end{array}\right) : a\equiv 1,\ b\equiv 0 \bmod 3 \right\}\subseteq \GL(2,\Z/3\Z).$$
\end{enumerate}
\end{thm}
\begin{proof}
Let $E/\Q(j_{K,f})$ be, as before, an elliptic curve with CM by $\Of\subset K$, and $p$ is an odd prime divisor of $f\Delta_K$. If $j_{K,f}=1728$, then $\Delta_K=-4$ and $f=1$, so no odd prime divides $f\Delta_K$. Thus, we assume that $j_{K,f}\neq 1728$. We fix a complex conjugation $c$. By Theorem \ref{thm-cmrep-intro} there is a $\Z/p\Z$-basis of $E[p]$ such that the image $G_{E,p}$ of $\rho_{E,p}$ is contained in  $\mathcal{N}_{\delta,0}(p)$, and by Lemma \ref{lem-ccfinal}, the basis can be chosen so that $\gamma=\rho_{E,p}(c) = c_\varepsilon$ for some $\varepsilon \in \{\pm 1 \}$.

Let $G_{E,K,p}$ be the image of $\rho_{E,p}|_{H_f}\colon \Gal(\overline{H_f}/H_f)\to \Aut_{\Of/p\Of}(E[p])$. Then, by our discussion in Section \ref{sec-cc}, the image $G_{E,p}$ is isomorphic to $G_{E,K,p}\rtimes \langle c_\varepsilon\rangle$, where the homomorphism $\varphi\colon \langle c_\varepsilon \rangle \to \Aut((\Of/p\Of)^\times)$ is given by $c\mapsto (\beta\mapsto c(\beta))$. 

\begin{enumerate}
	\item[(a)] Assume first that $j_{K,f}\neq 0,1728$. Then, $G_{E,K,p}$ is a subgroup of index at most $2$ in $\OO_p=\Of/p\Of$ by Theorem \ref{thm-incl}. Suppose that $[\OO_p:G_{E,K,p}]=2$. Since $p$ is odd and it divides $f\Delta_K$, the group $\OO_{p}\cong (\Z/p\Z)^\times \times \Z/p\Z$ is cyclic, and therefore it has a unique subgroup of index $2$, namely $((\Z/p\Z)^\times)^2 \times \Z/p\Z$. Thus, the image of $\rho_{E,p}$ would be isomorphic to
	\begin{eqnarray*}\label{eq-sqsubgp}
	\left(((\Z/p\Z)^\times)^2 \times \Z/p\Z\right) \rtimes \langle c _\varepsilon \rangle \cong \left\langle \left\{ \left(\begin{array}{cc}  a & b\\ \delta b & a\\ \end{array}\right): a\in ((\Z/p\Z)^\times)^2, b\in  \Z/p\Z\right\}, \left(\begin{array}{cc}  \varepsilon & 0\\ 0 & -\varepsilon\\ \end{array}\right) \right\rangle,\end{eqnarray*}
where $\delta = \Delta_K f^2/4 \equiv 0 \bmod p$. This shows (a).

	\item[(b)] Suppose $j_{K,f}=0$. Then, $K=\Q(\sqrt{-3})$, the conductor is $f=1$, and $\Of=\OO_K$ is the full ring of integers. Since $f\Delta_K=-3$, the unique prime dividing $f\Delta_K$ is $p=3$. In this case $\Of^\times$ has order $6$, so the index $[\OO_{p}:G_{E,K,p}]$ is a divisor of $6$, where $\OO_p = \OK/p\OK$ as before. If the index is $2$, then our considerations in part (a) apply here. Note, though, that $((\Z/3\Z)^\times)^2$ is trivial. Thus, the possible images are 
	$$\left\{ \left(\begin{array}{cc} \pm 1 & b\\ 0 & 1\\ \end{array}\right): b\in  \Z/3\Z\right\}\ \text{ or } \ \left\{ \left(\begin{array}{cc}  1 & b\\ 0 & \pm 1\\ \end{array}\right): b\in  \Z/3\Z\right\}.$$
	Since $\OO_{p}\cong (\Z/3\Z)^\times \times \Z/3\Z$,  if $[\OO_{p}:G_{E,K,p}]=3$, then we must have $G_{E,K,p}\cong (\Z/3\Z)^\times$, and there is a unique such subgroup, which corresponds to the subgroup of scalar matrices $\{a\cdot \operatorname{Id} : a\equiv \pm 1 \bmod 3\}$. Together with complex conjugation we obtain
		\begin{eqnarray*}
		(\Z/3\Z)^\times  \rtimes \langle c_\varepsilon \rangle \cong \left\{ \left(\begin{array}{cc} a & 0\\ 0 & c\\ \end{array}\right): a,c\in (\Z/3\Z)^\times \right\},\end{eqnarray*}
		which is a full split Cartan subgroup of $\GL(2,\Z/3\Z)$.
		
		Finally, if $[\OO_{p}:G_{E,K,p}]=6$, then we must have that $G_{E,K,p}$ is trivial. Hence, the image of $\rho_{E,3}$ must be of the form
			$$\left\{ \left(\begin{array}{cc} \pm 1 & 0\\ 0 & 1\\ \end{array}\right) \right\}\ \text{ or } \ \left\{ \left(\begin{array}{cc}  1 & 0\\ 0 & \pm 1\\ \end{array}\right) \right\}.$$
\end{enumerate}
This concludes the proof of the theorem.
\end{proof} 

Before we complete the proof of Theorem \ref{thm-badprimes-intro}, we need a result about subgroups of Cartan subgroups for the case of $j=0$ and $p=3$. Below, we adopt the following notation: $\langle a_1,\ldots,a_m\rangle /3^n$ denotes the multiplicative subgroup of $(\OK/3^n\OK)^\times$ generated by the elements $a_1,\ldots,a_m\bmod 3^n\OK$.

\begin{lemma}\label{lem-cartan-jzero}
	Let $n\geq 2$, let $K=\Q(\sqrt{-3})$, let $\cC_n = (\OK/3^n\OK)^\times$ of index $2$, and let $\pi$ be the natural reduction map $\OK/3^n\OK\to \OK/3^{n-1}\OK$. Then, 
	\begin{enumerate}
		\item For each $n\geq 1$, we have
		$$\cC_n = (\OK/3^n\OK)^\times = \langle -\zeta_3,4,1+\tau\rangle /3^n \cong \Z/6\Z \times \Z/3^{n-1}\Z \times \Z/3^{n-1}\Z,$$
		where $\tau=\sqrt{-3}/2$, and  $\zeta_3=-2^{-1}+\tau$ is a primitive $3$rd root of unity.
		\item If $n\geq 2$, and $H_n\subseteq \cC_n$ is a subgroup, then   $-1\in H_n$ if and only if $-1\in \pi(H_n)=H_{n-1}$. Moreover, there is a unique subgroup $H_n$ of index $2$ in $\cC_n$ that does not contain $-1$, namely $H_n = \langle \zeta_3,4,1+\tau\rangle / 3^n$. 
		\item Let $n\geq 3$ and let $H_n\subseteq \cC_n$ be a subgroup, and let $H_{n-1}=\pi(H_n)$. Then, $\zeta_3\in H_n$ if and only if $\zeta_3\in H_{n-1}$.
		\item For any $n\geq 2$, there are precisely two subgroups of $\cC_n$ of index $2^a\cdot 3$, for $a=0,1$, given by $H_{n,a}^i=\langle (-1)^{a+1},4,\zeta_3^i(1+\tau)\rangle/3^n $, for $i=0,1$, such that
		\begin{itemize} \item  $\zeta_3\not\in H_{n,a}^i$, 
			\item $H_{n,a}^i$ is invariant under complex conjugation, and
			\item the index of $H_{n,a}^i\bmod 3\OK$ in $\cC_1$ is $2^a$ (in particular, $\zeta_3 \in H_{n,a}^i \bmod 3\OK$).
		\end{itemize}   
	\end{enumerate} 
\end{lemma}
\begin{proof}
	Part (1) follows from
	\begin{align*} \cC_n &= (\OK/3^n\OK)^\times \cong (\Z/3^n\Z)^\times\times  \Z/3\Z\times \Z/3^{n-1}\Z\\
	&=\langle 2,\zeta_3,1+\tau\rangle / 3^n = \langle -\zeta_3,4,1+\tau\rangle /3^n \\
	&\cong \Z/6\Z \times \Z/3^{n-1}\Z \times \Z/3^{n-1}\Z.
	\end{align*} 
	In particular,  $\# \cC_n = 2\cdot 3^{2n-1}$, and so there is a unique element of order $2$ in $\cC_n$, namely $-1 \bmod 3^n\OK$, and a unique subgroup of $\cC_n$ of index $2$ that does not contain $-1$, namely $G_n=\langle \zeta_3,4,1+\tau\rangle /3^n$. Now suppose $H_n\subseteq \cC_n$ is a subgroup. Then, $-1\not\in H_n$ if and only if $H_n\subseteq G_n$. Since $G_{n-1}=\pi(G_n)$, it follows that $-1\in H_n$, if and only if $-1\in H_{n-1}$, as claimed. This shows (2). 
	
	For (3), let $n\geq 3$, and suppose that $H_n$ is a subgroup of $\cC_n$. First suppose that $\zeta_3\in H_n$. Then, $\zeta_3=\pi(\zeta_3)\in H_{n-1}$. Now suppose that $\zeta_3\not\in H_n$. Hence, $H_n$ is contained in an index $3$ subgroup $J_n\subseteq \cC_n$ that does not contain $\zeta_3$. Since, $\cC_n = \langle a=-\zeta_3,b=4,c=1+\tau\rangle /3^n $, it follows that $J_n = \langle a^3,a^ib,a^jc\rangle /3^n$, for some $i,j\in \{0,1,2\}$, and these are all distinct as long as $n\geq 2$. Moreover, $J_{n-1}=\pi(J_n)$ is also of index $3$ in $\cC_n$, for all $n\geq 2$. Since we assumed $n\geq 3$, it follows that $H_{n-1}\subseteq J_{n-1}$, and therefore $a=\zeta_3\not\in H_{n-1}$ either, as we needed to show. This shows (3).
	
	For (4), suppose $n\geq 2$, and $J_n$ is a subgroup of index $3$ in $\cC_n$ such that $\zeta_3\not\in J_n$. Then, $J_n= \langle a^3,a^ib,a^jc\rangle /3^n$, for some $i,j\in \{0,1,2\}$, with notation as in part (3). The only such subgroups that are stable under complex conjugation, and that surject to $\cC_1$ when reduced modulo $3\OK$, are given by
	$$H_{n,0}^0=\langle -1,4,1+\tau\rangle/3^n,\ \text{ and } H_{n,0}^1=\langle -1,4,\zeta_3(1+\tau)\rangle/3^n.$$
	The index of $H_{n,0}^i$ in $\cC_n$ is $3$, for $i=0,1$. Now, each of $H_{n,0}^i$ has a unique subgroup of index $2$, namely $H_{n,1}^i$. It follows that $H_{n,1}^i$ for $i\in {0,1}$ are the unique subgroups of $\cC$ of index $6$ that do not contain $\zeta_3$ and are stable under complex conjugation. This concludes the proof of (4) and of the lemma.
\end{proof}

\begin{remark}
	Let $J_2=\langle -1,4,1+\tau\rangle /9$. Then, $J_2$ is a subgroup of index $3$ of $\cC_2$. However, $J_1=\pi(J_2)=\cC_1$, because $1+\tau \equiv \zeta_3 \bmod 3\OK$ (because $-2^{-1}\equiv -(-1)^{-1}\equiv -(-1)\equiv 1 \bmod 3$). Similarly, $J_2' = \langle 4,1+\tau\rangle/9$ is of index $6$ in $\cC_2$, and $J_1'=\pi(J_2')$ is of index $2$ in $\cC_1$.
\end{remark}

\subsection{Proof of Theorem \ref{thm-badprimes-intro}}\label{sec-proof-badprimes}

		As before, if $F$ is a field, then $G_F$ denotes the absolute Galois group of $F$. Let $E/\Q(j_{K,f})$ be an elliptic curve with CM by $\Of$ of conductor $f\geq 1$. Let $p$ be an odd prime dividing $f\Delta_K$ (thus, $j\neq 1728$ where $f\Delta_K=-4$), let $G_{E,p^\infty}$ be the image of $\rho_{E,p^\infty}\colon G_{\Q(j_{K,f})}\to \Aut(T_p(E))$, let $G_{E,p}\equiv G_{E,p^\infty}\bmod p$ be the image of $\rho_{E,p}$, and let $c\in\Gal(\overline{\Q(j_{K,f})}/\Q(j_{K,f}))$ be a fixed complex conjugation. Then, by Theorem \ref{thm-largeimage-intro}, part (4) (which was shown in Theorem \ref{thm-largeimage-intro-part4}), there is a $\Z_p$-basis of $T_p(E)$ such that $G_{E,p^\infty}\subseteq \mathcal{N}_{\delta,0}(p^\infty)$, with $\delta=\Delta_K f^2/4$, and $\rho_{E,p^\infty}(c)=c_\varepsilon=\left(\begin{array}{cc} \varepsilon & 0\\ 0 & -\varepsilon\\ \end{array}\right)$ for some $\varepsilon \in \{\pm 1 \}$.

		First suppose that $j_{K,f}\neq 0,1728$, and let $p$ be odd. Then, Theorem \ref{thm-badprimes} shows that either $G_{E,p}=\mathcal{N}_{\delta,0}(p)$, or $G_{E,p}$ is generated by $c_{\varepsilon}$ and the group
		$$G_{E,K,p}=\left\{ \left(\begin{array}{cc} a & b\\ \delta b & a\\ \end{array}\right): a\in ((\Z/p\Z)^\times)^2, b\in  \Z/p\Z\right\}\subseteq \cC_{\delta,0}(p),$$
		where $\delta\equiv 0 \bmod p$, and $G_{E,K,p}=\rho_{E,p}(G_{H_f})$. By Theorem \ref{thm-inclp}, and since $j_{K,f}\neq 0$ and $p$ is odd, it follows that $G_{E,K,p^\infty}=\rho_{E,p^\infty}(G_{H_f})$ is the full inverse image of $G_{E,K,p}$ in $\cC_{\delta,\phi}(p^\infty)$. Hence, we conclude that either $G_{E,p^\infty}=\mathcal{N}_{\delta,0}(p^\infty)$, or $G_{E,p^\infty}$ is generated by $c_{\varepsilon}$ and the group
		$$\left\{ \left(\begin{array}{cc} a & b\\ \delta b & a\\ \end{array}\right): a\in ((\Z_p)^\times)^2, b\in  \Z_p\right\},$$
		as claimed.
		
		Now suppose that $K=\Q(\sqrt{-3})$ and $f=1$, $j_{K,f}=0$, $p=3$, and $E/\Q$ is an elliptic curve with $j(E)=0$. We distinguish several cases according to the image $G_{H_f}$ under $\rho_{E,3}$, which was computed in Theorem \ref{thm-badprimes}. Here $\delta=-3/4$ and $\phi=0$.
	\begin{itemize}
		\item If $[\cC_{\delta,0}(3):G_{E,K,3}]=6$, then $G_{E,K,3}\subseteq \GL(2,\Z/3\Z)$ is trivial. In particular, $\Gal(K(E[3])/K)$ is trivial (where we note that $H_f=K=\Q(\sqrt{-3})$), and therefore $[H_f(E[3]):H_f(h(E[3]))]=1$. Now Theorem \ref{thm-inclp} shows that $G_{E,K,3^\infty}$ is the full inverse image of $G_{E,K,3}$ in $\cC_{\delta,0}(3^\infty)$, and therefore  $G_{E,3^\infty}$ is generated by $c_{\varepsilon}$, for some $\varepsilon\in \{\pm 1 \}$, and 
		$$ \left\{ \left(\begin{array}{cc}  a & b\\ -3b/4 & a\\ \end{array}\right) : a\equiv 1,\ b\equiv 0 \bmod 3 \right\}\subseteq \GL(2,\Z_3).$$
		
		\item If $[\cC_{\delta,0}(3):G_{E,K,3}]=3$, then $G_{E,K,3}$ is the subgroup of order $2$ of scalar matrices in $\GL(2,\Z/3\Z)$. Notice that $(\OK/3\OK)^\times/(\OK^\times/\OO_{K,f,3}^\times)$ is trivial. In particular, $[H_f(E[3]):H_f(h(E[3]))]=2$, which is relatively prime to $p=3$. Thus, by Theorem \ref{thm-inclp}, $G_{E,3^\infty}$ is generated by $c_{\varepsilon}$, for some $\varepsilon\in \{\pm 1 \}$, and 
		$$ \left\{ \left(\begin{array}{cc}  a & b\\ -3b/4 & a\\ \end{array}\right) : a\in (\Z_3)^\times,\ b\equiv 0 \bmod 3 \right\}\subseteq \GL(2,\Z_3).$$
		
		\item If $d_1=[\cC_{\delta,0}(3):G_{E,K,3}]=2^{1-a}$, with $a=0$ or $1$, then $G_{E,K,3}$ is the subgroup of order $2^a\cdot 3$ generated by matrices $\left\langle \left(\begin{array}{cc} -1 & 0\\ 0 & -1\\ \end{array}\right)^a, \left(\begin{array}{cc} 1 & 1\\ 0 & 1\\ \end{array}\right)\right\rangle.$ Thus, if $a=0$, then $-1\bmod 3\OK$ is not in $G_{E,K,3}$ and, by Lemma \ref{lem-cartan-jzero} part (2), the unit $-1\bmod 3^n\OK$ is not in $G_{E,K,3^n}$ for any $n\geq 2$. In particular, the index $d_2=[\cC_{\delta,0}(9):G_{E,K,9}]=2$ or $6$.  If $d_2=2$, then $\zeta_3\in G_{E,K,9}\subseteq (\OK/9\OK)^\times$ and, therefore, by Lemma \ref{lem-cartan-jzero}, $\zeta_3\in G_{E,K,3^n}$ for all $n\geq 1$. Thus, in this case $G_{E,3^\infty}$ is generated by $c_{\varepsilon}$, for some $\varepsilon\in \{\pm 1 \}$, and 
		$$ \left\{ \left(\begin{array}{cc}  a & b\\ -3b/4 & a\\ \end{array}\right) : a,b\in (\Z_3)^\times,\ a\equiv 1\bmod 3 \right\}\subseteq \GL(2,\Z_3).$$
		If $d_1=2$ but $d_2=6$, then Lemma \ref{lem-cartan-jzero} part (4) shows that there are precisely two possibilities for $G_{E,K,3^n}$, and therefore for $G_{E,K,3^\infty}$. Thus, $G_{E,3^\infty}$ is generated by 
		$c_{\varepsilon}$, for some $\varepsilon\in \{\pm 1 \}$, and 
		$$ \left\langle \left(\begin{array}{cc}  4 & 0\\ 0 & 4\\ \end{array}\right) ,\left(\begin{array}{cc}  1 & 1\\ -3/4 & 1\\ \end{array}\right)\right\rangle,\ \text{ or } \ \left\langle \left(\begin{array}{cc}  4 & 0\\ 0 & 4\\ \end{array}\right) ,\left(\begin{array}{cc}  -5/4 & 1/2\\ -3/8 & -5/4\\ \end{array}\right)\right\rangle\subseteq \GL(2,\Z_3).$$ 
		If $a=1$ (i.e., when $d_1=1$ and $d_2=3$), the argument is similar except that $-\operatorname{Id}$ is in the image. Thus, $G_{E,K,3}=\cC_{\delta,0}(3)$, and in this case $G_{E,3^\infty}=\mathcal{N}_{\delta,0}(3^\infty)$ or it  is generated by $c_{\varepsilon}$, for some $\varepsilon\in \{\pm 1 \}$, and 
		$$ \left\langle \left(\begin{array}{cc}  2 & 0\\ 0 & 2\\ \end{array}\right) ,\left(\begin{array}{cc}  1 & 1\\ -3/4 & 1\\ \end{array}\right)\right\rangle,\ \text{ or } \ \left\langle \left(\begin{array}{cc}  2 & 0\\ 0 & 2\\ \end{array}\right) ,\left(\begin{array}{cc}  -5/4 & 1/2\\ -3/8 & -5/4\\ \end{array}\right)\right\rangle\subseteq \GL(2,\Z_3).$$ 
	\end{itemize}
This concludes the proof of Theorem \ref{thm-badprimes-intro}. 

\begin{example}\label{ex-jzero}
	Here we provide examples of elliptic curves $E:y^2=x^3+s$, defined over $\Q$ and with $j(E)=0$, with each type of possible $3$-adic image. We note that \cite[Prop. 1.16]{zywina} gives conditions on $s$ for the possible images mod $3$. Recall that $c_\varepsilon=\left(\begin{array}{cc}  \varepsilon & 0\\ 0 & -\varepsilon\\ \end{array}\right)$.	 Below we give pairs of values $(s,c_\varepsilon)$ such that $E:y^2=x^3+s$ has the desired image and complex conjugation maps to $c_\varepsilon$.
		\begin{enumerate} 
			\item[(o)] If $[\mathcal{N}_{\delta,0}(3^\infty):G_{E,3^\infty}]=1$, then one can take $(s,c)=(3,c_1)$ or $(-1,c_{-1})$.
			\item[(i)] If $[\mathcal{N}_{\delta,0}(3^\infty):G_{E,3^\infty}]=2$, then $G_{E,3^\infty}$ is generated by $c_{\varepsilon}$ and 
		$$\left\{ \left(\begin{array}{cc}  a & b\\ -3b/4 & a\\ \end{array}\right): a,b\in  \Z_3,\ a\equiv 1 \bmod 3\right\}\subseteq \GL(2,\Z_3),$$
		and one can take $(1,c_1)$ or $(-3,c_{-1})$.
		
		\item[(ii)]  If $[\mathcal{N}_{\delta,0}(3^\infty):G_{E,3^\infty}]=3$, then $G_{E,3^\infty}$ is generated by $c_{\varepsilon}$ and 
		$$\left\{ \left(\begin{array}{cc} a & b\\ -3b/4 & a\\ \end{array}\right): a\in (\Z_3)^\times,\ b\equiv 0 \bmod 3 \right\},$$
		$$\text{or }\  \left\langle \left(\begin{array}{cc}  2 & 0\\ 0 & 2\\ \end{array}\right) ,\left(\begin{array}{cc}  1 & 1\\ -3/4 & 1\\ \end{array}\right)\right\rangle,\ \text{ or } \ \left\langle \left(\begin{array}{cc}  2 & 0\\ 0 & 2\\ \end{array}\right) ,\left(\begin{array}{cc}  -5/4 & 1/2\\ -3/8 & -5/4\\ \end{array}\right)\right\rangle\subseteq \GL(2,\Z_3),$$
		and one can take $(2,c_1)$ or $(-2,c_{-1})$, $(6,c_{1})$ or $(-6,c_{-1})$, and $(18,c_1)$ or $(-18,c_{-1})$, respectively. 
		\item[(iii)]  If $[\mathcal{N}_{\delta,0}(3^\infty):G_{E,3^\infty}]=6$, then $G_{E,3^\infty}$ is generated by $c_{\varepsilon}$ and one of
		$$ \text{or }\ \left\{ \left(\begin{array}{cc}  a & b\\ -3b/4 & a\\ \end{array}\right) : a\equiv 1,\ b\equiv 0 \bmod 3\Z_3 \right\},$$
		$$ \text{or }\ \left\langle \left(\begin{array}{cc}  4 & 0\\ 0 & 4\\ \end{array}\right) ,\left(\begin{array}{cc}  1 & 1\\ -3/4 & 1\\ \end{array}\right)\right\rangle,\ \text{ or } \ \left\langle \left(\begin{array}{cc}  4 & 0\\ 0 & 4\\ \end{array}\right) ,\left(\begin{array}{cc}  -5/4 & 1/2\\ -3/8 & -5/4\\ \end{array}\right)\right\rangle\subseteq \GL(2,\Z_3),$$
		and one can take $(16,c_1)$ 
		 or $(-432,c_{-1})$,
		  $(1296,c_1)$ or $(-48,c_{-1})$, and $(144,c_1)$ or $(-3888,c_{-1})$, respectively. 
	\end{enumerate}
	Note: the author would like to help Drew Sutherland for his help in computing the images above (using code written for \cite{sutherland}, and some new and improved code).
\end{example}

\section{$2$-adic Galois Representations}\label{sec-2adic}
In this section we restrict our attention to the last case we have left to consider, that is, the case of $2$-adic Galois representations attached to CM curves, so we fix $p=2$. First, we record the group structure of a $2$-adic Cartan subgroup. Below, we assume the convention that $\Z/2^n\Z$ represents the trivial group if $n\leq 0$. Here we continue using the notation of Lemma \ref{lem-cartan-jzero}, where $\langle a_1,\ldots,a_m\rangle /2^n$ denotes the multiplicative subgroup of $(\Of/2^n\Of)^\times$ generated by the elements $a_1,\ldots,a_m\bmod 2^n\Of$.

\begin{thm}\label{thm-2adiccartan}
	Let $n\geq 1$, let $d=\Delta_K$ be a fundamental discriminant, and let $f\geq 1$.  Let $\delta$, $\phi$, and let $\cC_n = (\Of/2^n\Of)^\times\cong \cC_{\delta,\phi}(2^n)\subseteq \GL(2,\Z/2^n\Z)$, where $\cC_{\delta,\phi}(2^n)$ is defined as in Theorem \ref{thm-cmrep}. Then,
	\begin{enumerate}
		\item If $d\equiv 1 \bmod 4$, and $f\equiv 0 \bmod 2$, then there are three different cases to consider:
		\begin{enumerate}
			\item If $f\equiv 0 \bmod 4$, then $\cC_n\cong \Z/2\Z\times \Z/2^{n-2}\Z \times \Z/2^n\Z$, for all $n\geq 1$, and 
			$$\cC_n =\langle -1, 3, 1+f\tau \rangle/ 2^n,$$
			where $f\tau = (f/2)\sqrt{\Delta_K}.$
			\item If $f\equiv 2 \bmod 4$, and $d\equiv 1 \bmod 8$, then $\delta \equiv 1\bmod 8$, and $\delta= r^2$ for some $r\in\Z_2$. Thus, 
			$$\cC_n= \langle -1, 3, \frac{f\tau}{r}, 1+2f\tau \rangle/ 2^n \cong \begin{cases}
			\Z/2\Z & \text{ if } n=1,\\
			(\Z/2\Z\times \Z/2^{n-2}\Z)\times (\Z/2\Z\times \Z/2^{n-1}\Z) & \text{ if } n\geq 2.
			\end{cases} $$
			\item If $f\equiv 2 \bmod 4$, and $d\equiv 5 \bmod 8$, then $\cC_{\delta,0}(2^n)\cong \Z/2\Z\times \Z/2^{n-1}\Z \times \Z/2^{n-1}\Z$, for all $n\geq 1$, and 
			$\cC_n=\langle -1, f\tau, 1+2f\tau \rangle/ 2^n$.
		\end{enumerate} 	
		\item If $d\equiv 4\bmod 8$, then $d=-4d'$ with $d'\equiv 1 \bmod 4$. Thus, there are two cases to consider:
		\begin{enumerate}
			\item If $d'\equiv 1 \bmod 8$ and $f\not\equiv 0 \bmod 2$, then $-\delta\equiv 1 \bmod 8$ and $-\delta=r^2$ for some $r\in\Z_2$. Thus,
			$$\cC_n=\langle f\tau/r, 3, 1+2f\tau \rangle/ 2^n \cong \begin{cases}
			\Z/2\Z & \text{ if } n=1,\\
			\Z/4\Z\times \Z/2^{n-2}\Z\times \Z/2^{n-1}\Z & \text{ if } n\geq 2.
			\end{cases} $$
			
			\item If $d'\equiv 5 \bmod 8$ and $f\not\equiv 0 \bmod 2$, then
			$$\cC_n=\langle -1, f\tau, 1+2f\tau \rangle/ 2^n \cong \begin{cases}
			\Z/2\Z & \text{ if } n=1,\\
			\Z/2\Z \times \Z/4\Z & \text{ if } n=2,\\
			\Z/2\Z\times \Z/2^{n-1}\Z\times \Z/2^{n-1}\Z & \text{ if } n\geq 3.
			\end{cases} $$
		\end{enumerate} 
			\item If $d\equiv 0 \bmod 8$, or if $d\equiv 4 \bmod 8$ and $f\equiv 0 \bmod 2$, then
			$$\cC_n=\langle -1, 3, 1+f\tau \rangle/ 2^n \cong \begin{cases}
			\Z/2\Z & \text{ if } n=1,\\
			\Z/2\Z\times \Z/2^{n-2}\Z\times \Z/2^{n}\Z & \text{ if } n\geq 2.
			\end{cases} $$
		
		\item If $d\equiv 1\bmod 4$, and $f\not\equiv 0 \bmod 2$, then there are two cases to consider:
		\begin{enumerate}
			\item If $d\equiv 1 \bmod 8$, then $d= r^2$ for some $r\in\Z_2$, and
			$$\cC_n=\langle -1, 3, -(1+2\tau/f)/r, 1+4f\tau \rangle/ 2^n \cong \begin{cases}
			\{0\} & \text{ if } n=1,\\
			(\Z/2\Z\times \Z/2^{n-2}\Z)\times (\Z/2\Z\times \Z/2^{n-2}\Z) & \text{ if } n\geq 2.
			\end{cases} $$
			\item If $d\equiv 5 \bmod 8$, then 
			$$\cC_n=\langle -1, 3+4f\tau,f-(\delta/f^2)\tau \rangle/ 2^n \cong \begin{cases}
			\Z/3\Z & \text{ if } n=1,\\
			\Z/2\Z\times \Z/2^{n-2}\Z\times \Z/(2^{n-1}\cdot 3\Z) & \text{ if } n\geq 2.
			\end{cases} $$
		\end{enumerate}
	\end{enumerate}
\end{thm}

The proof of Theorem \ref{thm-2adiccartan} is a matter of verifying each of the cases (according to the splitting behaviour of $2$ in $\Of$), and so it is omitted. We first present a proof of Theorem \ref{thm-2adictwist-intro}, which deals with $2$-adic images when $j_{K,f}\neq 0$ or $1728$.

\begin{lemma}\label{lem-2adicindex2subgpsofCartan}
	Let $n\geq 2$, and let $H_n$ be a subgroup of $(\Of/2^n\Of)^\times$ of index $2$, and let $H_{n-1}$ be the image of $H_n$ under the natural reduction map $\Of/2^n\Of\to \Of/2^{n-1}\Of$. Then, 
	\begin{enumerate}
				\item If 
				\begin{itemize}
					\item 
				$n\geq 3$, and $\Delta_K=-4d'$ with $d'\equiv 1\bmod 8$, and $f\equiv 1 \bmod 2$, or 
				\item $n=2$ and $\Delta_K\equiv 4 \bmod 8$ and $f\equiv 1 \bmod 2$,
			\end{itemize} then every subgroup of index $2$ of $(\Of/2^n\Of)^\times$ contains $-1$.
		\item Let $n\geq 2$, let $H_n$ be a subgroup of index $2$ of $(\Of/2^n\Of)^\times$, and let $H_2\equiv H_n \bmod 4\Of$ be the reduction of $H_n$ modulo $4$. Suppose that:
		\begin{itemize}
			\item $-1$ is not in $H_2$, and
			\item $H_2$ is fixed under complex conjugation.
		\end{itemize}
		Then, $\Delta_Kf^2\equiv 0 \bmod 16$ and there are precisely two such subgroups $H_n$, namely
		$$\langle 5,1+f\tau\rangle/2^n \text{ and } \langle  5,-1-f\tau\rangle/2^n.$$

		\item If $n\geq 3$, the unit $-1\in H_n$ if and only if $-1\in H_{n+1}$.
		\item Suppose $n\geq 3$ and $H_n$ is a subgroup of index $2$ of $(\Of/2^n\Of)^\times$ such that:
		 \begin{itemize}
		 	\item $-1$ is not in $H_n$,
		 	\item $H_n$ is fixed under complex conjugation,
		 	\item $H_n$ surjects onto $(\Of/4\Of)^\times$ when reduced $\bmod 4\Of$.
		 \end{itemize}
	  	Then, $\Delta_K\equiv 0 \bmod 8$ and there are precisely two such subgroups, namely
	  	$$\langle 3,1+f\tau\rangle/2^n \text{ and } \langle 3, -1-f\tau\rangle/2^n.$$
	\end{enumerate} 
\end{lemma}
\begin{proof}

	By Theorem \ref{thm-2adiccartan}, if $n\geq 3$, then the group $\cC_n = (\Of/2^n\Of)^\times$ is generated as a multiplicative group by $\langle \alpha_1,\ldots,\alpha_t\rangle$ with $t=3$ or $4$ depending on the congruence classes of $\delta$ and $\phi$ modulo $8$, and such that either $\alpha_1=-1$ is of order $2$ or $\alpha_1^2=-1$ and $\alpha_1$ is of order $4$ (but note that the latter possibility only occurs when $\Delta_K=-4d'$ with $d'\equiv 1\bmod 8$, and $f\equiv 1 \bmod 2$).
	
	First suppose that $\Delta_K=-4d'$ with $d'\equiv 1\bmod 8$, and $f\equiv 1 \bmod 2$. Then, by Theorem \ref{thm-2adiccartan}, we have $\cC_n=\langle \alpha_1,\alpha_2,\alpha_3\rangle \cong \Z/4\Z \times \Z/2^{n-2}\Z \times \Z/2^{n-1}\Z$, with $\alpha_1^2=-1$ and $\alpha_1=f\tau/r$ is of order $4$, and $\alpha_2,\alpha_3$ are of order $2^{n-2}$ and $2^{n-1}$, respectively. Thus, there are $7$ subgroups of index $2$ of $\cC_n$, and are given by
	$$\langle \alpha_1^2,\alpha_2,\alpha_3 \rangle, \langle \alpha_1^2,\alpha_1\alpha_2,\alpha_3 \rangle, \langle \alpha_1^2,\alpha_2,\alpha_1\alpha_3 \rangle, \langle \alpha_1^2,\alpha_1\alpha_2,\alpha_1\alpha_3 \rangle, \langle \alpha_1,\alpha_2^2,\alpha_3 \rangle,\langle \alpha_1,\alpha_2^2,\alpha_2\alpha_3 \rangle,\langle \alpha_1,\alpha_2,\alpha_3^2 \rangle,$$
	and since every subgroup contains $\alpha_1$ or $\alpha_1^2=-1$, we conclude that $-1$ is in each group, as claimed. 
	
	Now suppose that $n=2$, and $f\equiv 1 \bmod 2$, and $\Delta\equiv 4\bmod 8$. According to Theorem \ref{thm-2adiccartan} part (2), in both cases, we have $\cC_2 =\langle f\tau,1+2f\tau\rangle=\langle \beta_1,\beta_2\rangle$, so the subgroups of index $2$ are given by $\langle \beta_1^2,\beta_2\rangle$, $\langle\beta_1,\beta_2^2\rangle$, and $\langle \beta_1^2,\beta_1\beta_2\rangle=\langle \beta_1\beta_2,\beta_2^2\rangle$. Since $\beta_1=f\tau$ squares to be congruent to $-1 \bmod 4\Of$, it follows that $-1$ is in all subgroups of index $2$, as claimed. This shows (1).

		Suppose $n\geq 2$ and $H_n$ is a subgroup of index $2$ of $(\Of/2^n\Of)^\times$ such that $-1$ is not in $H_2\equiv H_n \bmod 4\Of$, and
	$H_2$ is fixed under complex conjugation. Suppose first that $\Delta_kf^2\not\equiv 0 \bmod 16$. Then, by part (1), we have $\Delta_K\equiv 1 \bmod 4$ because when $\Delta_K\equiv 4\bmod 8$, the index $2$ subgroups contain $-1$. Now, by Theorem \ref{thm-2adiccartan}, the group $\cC_2 = (\Of/4\Of)^\times$ is generated in all cases as a multiplicative group by at most three elements $\langle -1,\alpha_2,\alpha_3\rangle/2^2$. In each case, one can compute the subgroups of index $2$ that miss $-1$, and verify that such subgroup is not stable under complex conjugation. Now suppose that $\Delta_Kf^2\equiv 0\bmod 16$. Then, in both cases, we have $\cC_2 = \langle -1,1+f\tau\rangle/2^2$, and therefore the only subgroups that miss $-1$ are $J_1=\langle 1 +f\tau\rangle/2^2$ and $J_2=\langle-1-f\tau\rangle/2^2$. Hence, $H_n$ reduces to $J_1$ or $J_2$. Since $H_n$ is of index $2$, it follows that $H_n$ is the full inverse image in $\cC_n$ of one of them. Thus, $H_n$ is $J_1=\langle 5,1 +f\tau\rangle/2^n$ or $J_2=\langle5,-1-f\tau\rangle/2^n$. This shows (2).

	For (3), first suppose that $-1\in H_{n+1}$. Since $-1\equiv -1 \bmod 2^n$, it follows that $-1\in H_n$ as well. Now suppose that $-1\not\in H_{n+1}$ and let us show that $H_{n}$ does not contain $-1$ either. By part (1), this is impossible if $\Delta_K=-4d'$ with $d'\equiv 1\bmod 8$, and $f\equiv 1 \bmod 2$. Thus, let us assume we are in any of the remaining cases. Thus, $\cC_{n+1} =\langle \alpha_1,\ldots,\alpha_t\rangle$ with $t=3$ or $4$, and $\alpha_1=-1$. If $-1\not\in H_{n+1}$, which is of index $2$ in $\cC_n$, it follows that $H_{n+1}=\langle \alpha_1^{a_1}\alpha_2,\ldots,\alpha_1^{a_t}\alpha_t\rangle$, for some $a_i\in \{0,1 \}$, since all subgroups of index $2$ that do not contain $-1$ are of this shape (there are precisely $2^{t-1}$ such subgroups). Therefore, $H_n = \langle \alpha_1^{a_1}\alpha_2,\ldots,\alpha_1^{a_t}\alpha_t\rangle$ as well. But then $-1\not\in H_n$, because if it were, then  $\cC_n=\langle \alpha_1,\alpha_2,\ldots,\alpha_t\rangle \subseteq H_n \subsetneq \cC_n$ 
	would be contradictory. Hence $-1\not\in H_n$ as we wanted to show.
	
	For (4), suppose that $H_n$ is a subgroup of index $2$ with the three properties listed in the statement, and consider the reduction $H_3 \equiv H_n \bmod 8\Of$. Then (by part (3)) the subgroup $H_3$ also satisfies the three properties of $H_n$. For every case listed in Theorem \ref{thm-2adiccartan}, one can compute all those (finitely many) subgroups $H$ of index $2$ in $(\Of/8\Of)^\times$ that satisfy all three properties (we have done so with Magma), and reach the conclusion that $H$ can only exist when $\Delta_K\equiv 0 \bmod 8$, and the only two such subgroups of $\langle -1,3,1+f\tau\rangle$ are given by $\langle 3,1+f\tau\rangle \text{ and } \langle 3, -1-f\tau\rangle,$ as claimed. 
\end{proof}

\begin{thm}\label{thm-2adictwist}
Let $E/\Q(j_{K,f})$ be an elliptic curve with CM by an order $\Of$ in an imaginary quadratic field $K$, with $j_{K,f}\neq 0$ or $1728$, and suppose that $\Gal(H_f(E[2^n])/H_f)\subsetneq (\Of/2^n\Of)^\times$ for some $n\geq 1$. Then, for all $n\geq 3$, we have $\Gal(H_f(E[2^n])/H_f)\cong (\Of/2^n\Of)^\times/\{\pm 1\}$, and there are two possibilities:
\begin{enumerate} 
\item If $\Gal(H_f(E[4])/H_f)\cong (\Of/4\Of)^\times/\{\pm 1\}$, then:

\begin{enumerate} 
\item $\operatorname{disc}(\Of)=\Delta_Kf^2\equiv 0 \bmod 16$. In particular, we have either
\begin{itemize}
\item $\Delta_K\equiv 1 \bmod 4$ and $f\equiv 0\bmod 4$, or
\item $\Delta_K\equiv 0\bmod 4$ and $f\equiv 0 \bmod 2$.
\end{itemize} 
\item The $4$th roots of unity $\mu_4$ are contained in $H_f$, i.e., $\Q(i)\subseteq H_f$.
\item For each $n\geq 2$, there is a $\Z/2^n\Z$-basis of $E[2^n]$ such that the image of the Galois representation $\rho_{E,2^n}\colon \Gal(\overline{H_f}/H_f)\to \GL(2,\Z/2^n\Z)$ is one of the groups
$$J_1=\left\langle \left(\begin{array}{cc} 5 & 0\\ 0 & 5\\\end{array}\right),\left(\begin{array}{cc} 1 & 1\\ \delta & 1\\\end{array}\right)\right\rangle \text{ or } J_2=\left\langle \left(\begin{array}{cc} 5 & 0\\ 0 & 5\\\end{array}\right),\left(\begin{array}{cc} -1 & -1\\ -\delta & -1\\\end{array}\right)\right\rangle\subseteq \cC_{\delta,0}(2^n).$$
\end{enumerate}
\item If $\Gal(H_f(E[4])/H_f)\cong (\Of/4\Of)^\times$, then:
\begin{enumerate}
	\item $\Delta_K\equiv 0 \bmod 8$.
	\item For each $n\geq 3$, there is a $\Z/2^n\Z$-basis of $E[2^n]$ such that the image of the Galois representation $\rho_{E,2^n}\colon \Gal(\overline{H_f}/H_f)\to \GL(2,\Z/2^n\Z)$ is  the group
	$$J_1'=\left\langle \left(\begin{array}{cc} 3 & 0\\ 0 & 3\\\end{array}\right),\left(\begin{array}{cc} 1 & 1\\ \delta & 1\\\end{array}\right)\right\rangle \text{ or } J_2'=\left\langle \left(\begin{array}{cc} 3 & 0\\ 0 & 3\\\end{array}\right),\left(\begin{array}{cc} -1 & -1\\ -\delta & -1\\\end{array}\right)\right\rangle\subseteq \cC_{\delta,0}(2^n).$$
\end{enumerate}
\end{enumerate}
Finally, there is some $\varepsilon \in \{\pm 1 \}$ and $\alpha\in {3,5}$ such that the image of $\rho_{E,2^\infty}$ is a conjugate of 
$$\left\langle \left(\begin{array}{cc} \varepsilon & 0\\ 0 & -\varepsilon\\\end{array}\right),\left(\begin{array}{cc} \alpha & 0\\ 0  & \alpha\\\end{array}\right),\left(\begin{array}{cc} 1 & 1\\ \delta  & 1\\\end{array}\right)\right\rangle \text{ or } \left\langle \left(\begin{array}{cc} \varepsilon & 0\\ 0 & -\varepsilon\\\end{array}\right),\left(\begin{array}{cc} \alpha & 0\\ 0  & \alpha\\\end{array}\right),\left(\begin{array}{cc} -1 & -1\\ -\delta  & -1\\\end{array}\right)\right\rangle \subseteq \GL(2,\Z_2).$$
\end{thm}
\begin{proof}
Let $E/\Q(j_{K,f})$ be an elliptic curve with CM by an order $\Of$ in an imaginary quadratic field $K$, with $j_{K,f}\neq 0$ or $1728$. Fix a basis of $E[2^n]$ such that the image of $\rho_{E,2^n}$ is contained in $\mathcal{N}_{\delta,\phi}(2^n)$, as in Theorem \ref{thm-cmrep-intro}, and the image of $\Gal\left(\overline{H_f}/H_f\right)$ via $\rho_{E,2^n}$ is contained in $\cC_{\delta,\phi}(2^n)$. Using $\rho_{E,2^n}$, we identify $G_{E,K,2^n}=\Gal(H_f(E[2^n])/H_f)$ with a subgroup of $\cC_n=(\Of/2^n\Of)^\times$.

Suppose that $G_{E,K,2^n}\subsetneq \cC_n$ for some $n\geq 1$. Since $j_{K,f}\neq 0,1728$, we have $\Of^\times \cong \{ \pm 1\}$. Then, Theorem \ref{thm-incl} and Lemma \ref{lem-unitsmod} show that $n\geq 2$, and we must have $G_{E,K,2^n}\cong \cC_n/\{\pm 1\}$. Corollary \ref{cor-missesrootofunity} shows that $-1$ must be the root of unity that is missing in the image $G_{E,K,2^n}$ as a subgroup of $\cC_n$, and therefore $G_{E,K,2^n}$ is a subgroup of $\cC_n$ of index $2$ that does not contain $-1$. Now, Lemma \ref{lem-2adicindex2subgpsofCartan} says that if $-1$ is missing for $n=3$ if and only if it is missing for all $n\geq 3$, so we must have that if $\Gal(H_f(E[2^n])/H_f)\cong \cC_n/\{\pm 1\}$ occurs for one $n\geq 3$, then it happens for all $n\geq 3$. Thus, it suffices to study when $-1$ is missing from $G_{E,K,2^n}$ for $n=2$ and $n=3$. As we explained in Section \ref{sec-cc}, if we fix a complex conjugation $c\in G_{\Q(j_{K,f})}$, and $\gamma=\rho_{E,2^\infty}(c)$, then conjugating $G_{E,K,2^n}$ by $\gamma$ acts as complex conjugation on the elements of $\cC_n=(\Of/2^n\Of)^\times$. Therefore, $G_{E,K,2^n}\subsetneq \cC_n$ must be a subgroup of index $2$ that is stable under complex conjugation that does not contain $-1\bmod 2^n\Of$.

Let us first suppose that $G_{E,K,2^n}=\Gal(H_f(E[2^n])/H_f)$ is a subgroup of index $2$ of $\cC_n$ (isomorphic to a subgroup of $\cC_{\delta,\phi}(2^n)\subseteq \GL(2,\Z/2^n\Z)$) that is missing the class of $-1\bmod 2^n\Of$ and is stable under complex conjugation, for all $n\geq 2$. Then, Lemma \ref{lem-2adicindex2subgpsofCartan} part (2) shows that $\Delta_Kf^2\equiv 0 \bmod 16$, and there are precisely two possibilities for $G_{E,K,2^n}$, namely 
$$J_1=\left\langle \left(\begin{array}{cc} 5 & 0\\ 0 & 5\\\end{array}\right),\left(\begin{array}{cc} 1 & 1\\ \delta & 1\\\end{array}\right)\right\rangle \text{ or } J_2=\left\langle \left(\begin{array}{cc} 5 & 0\\ 0 & 5\\\end{array}\right),\left(\begin{array}{cc} -1 & -1\\ -\delta & -1\\\end{array}\right)\right\rangle\subseteq \cC_{\delta,0}(2^n).$$
This shows (1a) and (1c). For (1b), we note that the determinant is trivial modulo $4$ on both groups. It follows that the $4$-th cyclotomic character is trivial over $H_f$, which means that $\mu_4\subset H_f$. Thus, we would have $\Q(i)\subseteq H_f=K(j_{K,f})$, as claimed. This shows (1).

For (2), let us now suppose that $G_{E,K,4}=\Gal(H_f(E[4])/H_f)$ is all of of $\cC_2$ (that is, $G_{E,K,4}$ is not missing $-1$), but $G_8=G_{E,K,8}=\Gal(H_f(E[8])/H_f)$ is a subgroup of $\cC_3$ of index $2$ that does not contain $-1$. Therefore, $G_{E,K,8}\subseteq \cC_8$ must be a subgroup that is stable under complex conjugation, and, by assumption, the reduction $G_{E,K,8}\bmod 4\Of$ must be all of $\cC_2$. Thus, Lemma \ref{lem-2adicindex2subgpsofCartan} part (4) shows that $\Delta_K\equiv 0\bmod 8$, and there are exactly two possible subgroups of $\GL(2,\Z/2^n\Z)$ given by
$$ J_1'=\left\langle \left(\begin{array}{cc} 3 & 0\\ 0 & 3\\\end{array}\right),\left(\begin{array}{cc} 1 & 1\\ \delta & 1\\\end{array}\right)\right\rangle \text{ or } J_2'=\left\langle \left(\begin{array}{cc} 3 & 0\\ 0 & 3\\\end{array}\right),\left(\begin{array}{cc} -1 & -1\\ -\delta & -1\\\end{array}\right)\right\rangle \subseteq \cC_{\delta,0}(2^n).$$
This shows (2).

Finally, by Theorem \ref{thm-cmrep}, the image of $\rho_{E,2^\infty}$ will be generated by the image of $\Gal(\overline{H_f}/H_f)$ and the image (call it $\gamma$) of complex conjugation, as a subgroup of $\mathcal{N}_{\delta,0}(2^\infty)$. Since $\gamma\in \mathcal{N}_{\delta,0}(2^\infty)$ must be a matrix of determinant equal to $-1$, zero trace, and order $2$, it follows that
$$\gamma=\left(\begin{array}{cc} -g & h\\ -\delta h & g\\\end{array}\right),$$
for some $g,h\in\Z_2$ with $g^2-\delta h^2=1$, or $g^2=1+\delta h^2$. If we consider instead $\gamma' = \left(\begin{array}{cc} g & h\\ \delta h & g\\\end{array}\right)$, then $\gamma'\in \cC_{\delta,0}(2^\infty)$. In case (1) and (2) above, $G_{E,K,2^n}$ is a subgroup of index $2$ isomorphic to $\cC_{\delta,0}(2^n)/\{\pm 1\}$ for all $n\geq 2$ and $n\geq 3$, respectively. In particular, $G_{E,K,2^\infty}$ is a subgroup of index $2$ of $\cC_{\delta,0}(2^\infty)$ isomorphic to $\cC_{\delta,0}(2^\infty)/\{\pm 1\}$. Thus, there is a $\varepsilon \in \{\pm 1 \}$ such that $\varepsilon \gamma' \in G_{E,K,2^\infty}$. Since
$$\left(\begin{array}{cc} -g & h\\ -\delta h & g\\\end{array}\right)\cdot \left(\begin{array}{cc} -\varepsilon & 0\\ 0 & \varepsilon\\\end{array}\right)^{-1}=\left(\begin{array}{cc} -g & h\\ -\delta h & g\\\end{array}\right)\cdot \left(\begin{array}{cc} -\varepsilon & 0\\ 0 & \varepsilon\\\end{array}\right)=\left(\begin{array}{cc} g\varepsilon & h\varepsilon\\ \delta h\varepsilon & g\varepsilon\\\end{array}\right)=\varepsilon \gamma',$$
it follows that if we put $c_{-\varepsilon}= \left(\begin{array}{cc} -\varepsilon & 0\\ 0 & \varepsilon\\\end{array}\right)$, then $\gamma' = c_{-\varepsilon}\cdot (\varepsilon \gamma') \subseteq \langle c_{-\varepsilon} \rangle \cdot G_{E,K,2^\infty}$. Thus, the image of $\rho_{E,2^\infty}$ is generated by $\gamma$ and $G_{E,K,2^\infty}$, and this group coincides with the subgroup generated by $c_{-\varepsilon}$ and $G_{E,K,2^\infty}$. By parts (1) and (2), the group $G_{E,K,2^\infty}$ is one of $J_1,J_2,J_1'$, or $J_2'$, and this proves the last statement of the theorem.
\end{proof}

\begin{example}\label{ex-2adicimages}
	Here we provide examples of elliptic curves $E$, defined over $\Q$, with each type of possible $2$-adic images as described in Theorem \ref{thm-2adictwist}. Recall that $c_\varepsilon=\left(\begin{array}{cc}  \varepsilon & 0\\ 0 & -\varepsilon\\ \end{array}\right)$.	 
	\begin{enumerate} 
		\item[(0)] If $\Delta_Kf^2\not\equiv 0 \bmod 8$, and $j_{K,f}\neq 0,1728$, then Theorem \ref{thm-2adictwist} shows that the image of $\rho_{E,2^\infty}$ must be all of $\mathcal{N}_{\delta,\phi}(2^\infty)$. For instance, let $E: y^2=x^3-1715x+33614$, which has CM by the maximal order of $\Q(\sqrt{-7})$. Here $\delta=(\Delta_K-1)f^2/4=-2$ and $\phi=f=1$, so according to Lemma \ref{lem-2adicindex2subgpsofCartan} the Cartan subgroup $\cC_{\delta,\phi}(8)$ has size $16$ and $\mathcal{N}_{\delta,\phi}(8)$ is of size $32$. We have verified using Magma that $\Gal(\Q(E[8])/\Q)\cong \mathcal{N}_{-2,1}(8)$ and the image of $\rho_{E,2^\infty}$ is a conjugate of
		$$\mathcal{N}_{-2,1}(8) = \left\langle \left(\begin{array}{cc}  -1 & 0\\ 1 & 1\\ \end{array}\right),\left(\begin{array}{cc} a+b & b\\ -2b & a\\ \end{array}\right) : a,b \in \Z/8\Z,\ a^2+ab+2b^2\equiv 1 \bmod 2\right\rangle.$$
		
		\item[(1)] If $\Gal(H_f(E[4])/H_f)\cong (\Of/4\Of)^\times/\{\pm 1\}$, then Theorem \ref{thm-2adictwist} says that $\Delta_k f^2\equiv 0 \bmod 16$, and since $E$ is defined over $\Q$, we must have $j_{K,f}\in \Q$, and therefore $\Delta_K=-4$ and $f=2$. In particular, $\delta=-4$ and $\phi=0$.
		\begin{itemize}
			\item The elliptic curve $E_1: y^2=x^3-11x+14$ has $2$-adic image
			$$\left\langle \left(\begin{array}{cc} 1 & 0\\ 0 & -1\\\end{array}\right), \left(\begin{array}{cc} 5 & 0\\ 0 & 5\\\end{array}\right),\left(\begin{array}{cc} -1 & -1\\ -\delta & -1\\\end{array}\right)\right\rangle \subseteq \mathcal{N}_{\delta,0}(2^\infty)\subseteq \GL(2,\Z_2).$$ 
			\item The elliptic curve $E_2: y^2=x^3-11x-14$ has $2$-adic image
			$$\left\langle \left(\begin{array}{cc} -1 & 0\\ 0 & 1\\\end{array}\right), \left(\begin{array}{cc} 5 & 0\\ 0 & 5\\\end{array}\right),\left(\begin{array}{cc} -1 & -1\\ -\delta & -1\\\end{array}\right)\right\rangle \subseteq \mathcal{N}_{\delta,0}(2^\infty)\subseteq \GL(2,\Z_2).$$
			\item The elliptic curve $E_3: y^2=x^3-44x+112$ has $2$-adic image
			$$\left\langle \left(\begin{array}{cc} 1 & 0\\ 0 & -1\\\end{array}\right), \left(\begin{array}{cc} 5 & 0\\ 0 & 5\\\end{array}\right),\left(\begin{array}{cc} 1 & 1\\ \delta & 1\\\end{array}\right)\right\rangle \subseteq \mathcal{N}_{\delta,0}(2^\infty)\subseteq \GL(2,\Z_2).$$
			\item The elliptic curve $E_4: y^2=x^3-44x-112$ has $2$-adic image
			$$\left\langle \left(\begin{array}{cc} -1 & 0\\ 0 & 1\\\end{array}\right), \left(\begin{array}{cc} 5 & 0\\ 0 & 5\\\end{array}\right),\left(\begin{array}{cc} 1 & 1\\ \delta & 1\\\end{array}\right)\right\rangle \subseteq \mathcal{N}_{\delta,0}(2^\infty)\subseteq \GL(2,\Z_2).$$
		\end{itemize}
		
		\item[(2)] If $\Gal(H_f(E[4])/H_f)\cong (\Of/4\Of)^\times$, then $\Delta_K\equiv 0 \bmod 8$, and since $E$ is defined over $\Q$, we must have $K=\Q(\sqrt{-2})$, $f=1$, and $j_{K,f}=2^6\cdot 5^3$.
		\begin{itemize}
			\item The elliptic curve $E_1: y^2=x^3-4320x+96768$ has $2$-adic image
			$$\left\langle \left(\begin{array}{cc} 1 & 0\\ 0 & -1\\\end{array}\right), \left(\begin{array}{cc} 3 & 0\\ 0 & 3\\\end{array}\right),\left(\begin{array}{cc} -1 & -1\\ -\delta & -1\\\end{array}\right)\right\rangle \subseteq \mathcal{N}_{\delta,0}(2^\infty)\subseteq \GL(2,\Z_2).$$ 
			\item The elliptic curve $E_2: y^2=x^3-4320x-96768$ has $2$-adic image
			$$\left\langle \left(\begin{array}{cc} -1 & 0\\ 0 & 1\\\end{array}\right), \left(\begin{array}{cc} 3 & 0\\ 0 & 3\\\end{array}\right),\left(\begin{array}{cc} 1 & 1\\ \delta & 1\\\end{array}\right)\right\rangle \subseteq \mathcal{N}_{\delta,0}(2^\infty)\subseteq \GL(2,\Z_2).$$ 
			\item The elliptic curve $E_3: y^2=x^3- 17280x + 774144$ has $2$-adic image
			$$\left\langle \left(\begin{array}{cc} 1 & 0\\ 0 & -1\\\end{array}\right), \left(\begin{array}{cc} 3 & 0\\ 0 & 3\\\end{array}\right),\left(\begin{array}{cc} 1 & 1\\ \delta & 1\\\end{array}\right)\right\rangle \subseteq \mathcal{N}_{\delta,0}(2^\infty)\subseteq \GL(2,\Z_2).$$
			\item The elliptic curve $E_4: y^2=x^3- 17280x - 774144$ has $2$-adic image
			$$\left\langle \left(\begin{array}{cc} -1 & 0\\ 0 & 1\\\end{array}\right), \left(\begin{array}{cc} 3 & 0\\ 0 & 3\\\end{array}\right),\left(\begin{array}{cc} -1 & -1\\ -\delta & -1\\\end{array}\right)\right\rangle \subseteq \mathcal{N}_{\delta,0}(2^\infty)\subseteq \GL(2,\Z_2).$$
		\end{itemize}
	\end{enumerate}
	Note: the author would like to help Drew Sutherland for his help in computing the images above (using code written for \cite{sutherland}, and some new and improved code).
\end{example}

\begin{example}
Let $K=\Q(\sqrt{-3})$ and let $\Of=\OO_4$ be the order of $K$ of conductor $f=4$. Let $E$ be an elliptic curve with CM by $\OO_4$. Then, $j_{K,f}$ is a root of $x^2-2835810000x+6549518250000$ and $\Q(j_{K,f})=\Q(\sqrt{3})$. In fact, the roots of the minimal polynomial of $j_{K,f}$ are
$$1417905000\pm 818626500\sqrt{3}.$$
Let us fix $j_{K,f}=1417905000 - 818626500\sqrt{3}$. In particular, $H_f=K(j_{K,f})=\Q(\sqrt{3},\sqrt{-3})=\Q(i,\sqrt{3})$. For instance, an elliptic curve with such $j=j_{K,f}$ is given by
\begin{align*} y^2 + \frac{(j - 599278500)}{818626500}xy + \frac{(j
    - 599278500)}{818626500}y = x^3 - x^2 - \frac{(j + 355785750)}{136437750}x - \frac{(j + 14691375)}{68218875},\end{align*}
or, equivalently,
$$E/\Q(\sqrt{3})\colon y^2 + (1-\sqrt{3})xy + (1-\sqrt{3})y = x^3 - x^2 + (6\sqrt{3} -
    13)x + (12\sqrt{3} - 21).$$
This is the curve \href{http://www.lmfdb.org/EllipticCurve/2.2.12.1/16.1/a/3}{\texttt{16.1-a3}} in the LMFDB database for elliptic curves over $\Q(\sqrt{3})$. Its conductor ideal is $(4)$. Since $\Delta_K=-3\equiv 1 \bmod 4$ and $f\equiv 0 \bmod 4$, Theorem \ref{thm-cmrep} says that the mod $2$ image of $\rho_{E,2}$ is contained in $\mathcal{N}_{\delta,0}(2) = \cC_{\delta,0}(2)$ (the complex multiplication matrix reduces to the identity modulo $2$). Moreover,  $\Of^\times/\mathcal{O}_{K,f,2}^\times$ is trivial, so the mod $2$ image is precisely $\cC_{\delta,0}(2)$, which by Theorem \ref{thm-2adiccartan} is $\cC_{\delta,0}(2)\cong \Z/2\Z$, given by
$$\rho_{E,2}(\Gal(\overline{\Q(j_{K,f})}/\Q(j_{K,f}))) = \left\{\left(\begin{array}{cc} 1 & 0\\ 0 & 1\\\end{array}\right),\left(\begin{array}{cc} 1 & 1\\ 0 & 1\\\end{array}\right) \right\}\subseteq \GL(2,\Z/2\Z).$$
Similarly, the image of the mod $4$ representation is contained in $\mathcal{N}_{\delta,0}(4) $, given by
$$\mathcal{N}_{\delta,0}(4)  =\left\langle \left(\begin{array}{cc} -1 & 0\\ 0 & 1\\\end{array}\right),\left(\begin{array}{cc} -1 & 0\\ 0 & -1\\\end{array}\right),\left(\begin{array}{cc} 1 & 1\\ 0 & 1\\\end{array}\right)\right\rangle.$$
However, $\Of^\times/\mathcal{O}_{K,f,4}^\times \cong \{\pm 1\}$ so the possibility exists that the image is in fact isomorphic to a subgroup of $\mathcal{N}_{\delta,0}(4) $ of index $2$, that does not contain $-\operatorname{Id}$. If the image was all of $\mathcal{N}_{\delta,0}(4) $ (a group of order $16$), then we would have that the extension $\Q(j_{K,f},E[4])/\Q(j_{K,f})$ is of degree $16$, and $\Q(j_{K,f},E[4])/\Q$ of degree $32$. However, for our choice of $E$, we have $\Q(j_{K,f},E[4])/\Q(j_{K,f})$ of degree $8$ and $\Q(j_{K,f},E[4])/\Q$ of degree $16$. Hence, the image of $\rho_{E,4}$ is of index $2$ in $\mathcal{N}_{\delta,0}(4) $, and therefore the $2$-adic image is also of index $2$ in $\mathcal{N}_{\delta,0}(2^\infty)$. 
Note that $\Delta_Kf^2 = -3\cdot 4^2 = -48\equiv 0 \bmod 16$ and $\Q(i)\subseteq H_f=\Q(\sqrt{3},i)$, as predicted by Theorem \ref{thm-2adictwist}.
\end{example}

In the next subsections, we study the $2$-adic images when $j=0$ or $1728$.

\subsection{$2$-adic images for $j=1728$}\label{sec-2adic1728}
Here we prove Theorem \ref{thm-j1728-intro}. Note that the proof contains much more information than the statement itself, since we produce explicit $\Z/4\Z$-bases of $E[4]$ where the image is contained in $\mathcal{N}_{\delta,0}(4)$.

\begin{lemma}\label{lem-2adiccartanj1728}
	Let $n\geq 1$, let $K=\Q(i)$, let $\Delta_K=-4$, let $f=1$, let $\tau=\sqrt{\Delta_K}/2 = i$, and put $\delta=-1$, $\phi=0$.  Let $\cC_n = (\OK/2^n\OK)^\times\cong \cC_{\delta,\phi}(2^n)\subseteq \GL(2,\Z/2^n\Z)$, where $\cC_{\delta,\phi}(2^n)$ is defined as in Theorem \ref{thm-cmrep}. Let $\pi\colon \cC_n\to \cC_{n-1}$ be given by reduction mod $2^{n-1}\OK$. Then,
	\begin{enumerate}
		\item $\cC_n=\langle \tau, 3, 1+2\tau \rangle/ 2^n \cong \begin{cases}
			\Z/2\Z & \text{ if } n=1,\\
			\Z/4\Z\times \Z/2^{n-2}\Z\times \Z/2^{n-1}\Z & \text{ if } n\geq 2.
		\end{cases} $
		\item For all $n\geq 2$, every subgroup of index $2$ of $\cC_n$ contains the class of $-1$.
		\item Let $n\geq 4$, and let $H_n$ be a subgroup of $\cC_n$, and let $H_{n-1}=\pi(H_n)\subseteq \cC_{n-1}$. Then, $\tau \in H_n$ if and only if $\tau \in H_{n-1}$.
		\item Suppose that   $H_n$, for some $n\geq 3$, is a subgroup with the following properties:
		\begin{itemize} 
			\item The index of $H_n$ in $\cC_n$ is $4$, and $H_n$ is  missing the classes of $\tau$ and $-1 \bmod 2^n\OK$, 
			\item The index of $H_2=\pi(H_n)$ in $\cC_2$ is strictly less than $4$.
			\item $H_n$ is stable under complex conjugation.  
		\end{itemize} 	
			Then, $[\cC_2:H_2]=2$, and $H_n=\langle -3, 2-\tau\rangle/2^n$ or $\langle -3,-2+\tau\rangle/2^n$.
		\item If $H_3$ is a subgroup of $\cC_3$ of index $2$, such that $-1\in H_3$ and $\tau\not\in H_3$, and such that $H_3$ is stable under complex conjugation, then $H_2=\pi(H_3)$ is of index $2$ in $\cC_2$.
		
		\item Suppose $H_n\subseteq \cC_n$ is a subgroup that is the full inverse image its reduction mod $4\OK$ in $\cC_2$. Then:
		\begin{itemize}
			\item If $[\cC_n:H_n]=2$, then $H_n$ is one of:
			$$\langle -1,3,1+2\tau\rangle,\ \langle -1,3,2+\tau \rangle,\ \text{ or } \ \langle \tau, 3, -3+4\tau\rangle.$$
			\item If $[\cC_n:H_n]=4$, then $H_n$ is one of:
			$$\langle 5,1+2\tau\rangle,\ \langle 5,-1-2\tau \rangle,\ \text{ or } \ \langle -1, 5, -3+4\tau\rangle.$$
		\end{itemize} 
	\end{enumerate}
\end{lemma}

\begin{proof}
	Part (1) has been shown in Theorem \ref{thm-2adiccartan}, and part (2) in Lemma \ref{lem-2adicindex2subgpsofCartan}, part (1). 
	
	For (3), clearly if $\tau\in H_n$, then $\tau\equiv \tau \bmod 2^{n-1}\OK$ belongs to $H_{n-1}$. Now suppose that $\tau\not\in H_n$. Part (1) shows that $\cC_n=\langle \alpha_1,\alpha_2,\alpha_3\rangle \cong \Z/4\Z \times \Z/2^{n-2}\Z \times \Z/2^{n-1}\Z$, with $\alpha_1=\tau-i$ of order $4$ and $\alpha_1^2=-1$, and $\alpha_2,\alpha_3$ are of order $2^{n-2}$ and $2^{n-1}$, respectively. Thus, there are $7$ subgroups of index $2$ of $\cC_n$, and are given by 
	$$\langle \alpha_1^2,\alpha_2,\alpha_3 \rangle, \langle \alpha_1^2,\alpha_1\alpha_2,\alpha_3 \rangle, \langle \alpha_1^2,\alpha_2,\alpha_1\alpha_3 \rangle, \langle \alpha_1^2,\alpha_1\alpha_2,\alpha_1\alpha_3 \rangle, \langle \alpha_1,\alpha_2^2,\alpha_3 \rangle,\langle \alpha_1,\alpha_2^2,\alpha_2\alpha_3 \rangle,\langle \alpha_1,\alpha_2,\alpha_3^2 \rangle.$$
	Thus, if $H_n$ is a subgroup such that $\tau \not\in H_n$, then $H_n$ is contained in one of the four subgroups $J_{i,j} = \langle \alpha_1^2,\alpha_1^i\alpha_2,\alpha_1^j\alpha_3\rangle/2^n$, for $i,j \in \{0,1 \}$, and therefore $H_{n-1}$ is contained in one of $J_{i,j}'=\pi(J_{i,j}) \bmod 2^{n-1}\OK$. If $n-1\geq 3$, then $J_{i,j}'=\langle \alpha_1^2,\alpha_1^i\alpha_2,\alpha_1^j\alpha_3\rangle/2^{n-1}$ is generated by the classes of the same elements, and therefore $\tau=\alpha_1$ is not in $J_{i,j}'$ either, and it follows that $H_{n-1}$ cannot contain $\tau$. This shows (3).
	
	For (4), suppose that $n\geq 3$ and $H_n$, for some $n\geq 3$, is a subgroup of index $4$ in $\cC_n$ missing $\tau$. Then, $H_n = \langle 3\tau^i,(1+2\tau)\tau^j\rangle$ for some $i\in \{ 0, 2\}$ and $0\leq j\leq 3$ (note that if $i=1$ or $3$, then $(3\tau)^2=-1$, but $-1\not\in H_n$). Out of these possibilities, there are only two that are stable under complex conjugation, namely $\langle 3\tau^2,(1+2\tau)\tau^3\rangle = \langle -3, 2-\tau\rangle/2^n$ or $\langle 3\tau^2,(1+2\tau)\tau\rangle = \langle -3,-2+\tau\rangle/2^n$, as claimed.
	
	For (5), since the index $[\cC_2:H_2]$ is a divisor of the index $[\cC_3:H_3]$, it suffices to show that $[\cC_2:H_2]=1$ cannot occur. For a contradiction, suppose that $H_3$ is as in the statement, and $[\cC_2:H_2]=1$. The subgroups of $\cC_3$ that do not contain $\tau$ but $-1\in H_3$, are those of the form $\langle \alpha_1^2,\alpha_1^i\alpha_2,\alpha_1^j\alpha_3\rangle/2^3$. Of these, the only subgroups that contain $\tau$ when reduced mod $4\OK$ are given by $\langle -1, 3\tau, -2+\tau\rangle$ and $\langle -1, 3\tau, 2+\tau\rangle/2^3$. However, these subgroups are not stable under complex conjugation. Thus, $[\cC_2:H_2]=1$ is impossible, and we must have $[\cC_2:H_2]=2$.
	
	Finally, for (6), we note that it suffices to determine the subgroups $H_2$ of $\cC_2$ of index $2$ and $4$, and then compute their full inverse images in $\cC_n$. Since $\cC_2 = \langle \tau,1+2\tau \rangle/2^2 \cong \Z/4\Z \times \Z/2\Z$, it follows that there are three subgroups of index $2$:
	$$\langle -1,1+2\tau\rangle,\ \langle -1,2+\tau \rangle,\ \text{ or } \ \langle \tau\rangle/2^2,$$
	and three subgroups of index $4$:
	$$\langle 1+2\tau\rangle,\ \langle -1-2\tau \rangle,\ \text{ or } \ \langle -1\rangle/4.$$
	The full inverse images of these groups in $\cC_n$ are as claimed in the statement of the lemma.
\end{proof}

We are now ready to give a full proof of Theorem \ref{thm-j1728-intro}.

\begin{thm}\label{thm-j1728}
	Let $E/\Q$ be an elliptic curve with $j(E)=1728$, and let $c\in \GQ$ be a complex conjugation, and $\gamma=\rho_{E,2^\infty}(c)$. Let $G_{E,2^\infty}$ be the image of $\rho_{E,2^\infty}$ and let $G_{E,K,2^\infty}=\rho_{E,2^\infty}(G_{\Q(i)})$. Then, there is a $\Z_2$-basis of $T_2(E)$ such that $G_{E,K,2^\infty}$ is one of the following groups:
	\begin{itemize} 
		\item If $[\cC_{-1,0}(2^\infty):G_{E,K,2^\infty}]=1$, then $G_{E,K,2^\infty}$ is all of $\cC_{-1,0}(2^\infty)$, i.e., 
		$$G_1= \left\{\left(\begin{array}{cc} a & b\\ -b & a\\\end{array}\right)\in \GL(2,\Z_2) : a^2+b^2\not\equiv 0 \bmod 2 \right\}.$$
		
		\item If $[\cC_{-1,0}(2^\infty):G_{E,K,2^\infty}]=2$, then $G_{E,K,2^\infty}$ is one of the following groups:
		$$G_{2,a}=\left\langle -\operatorname{Id}, 3\cdot \operatorname{Id},\left(\begin{array}{cc} 1 & 2\\ -2 & 1\\\end{array}\right) \right\rangle, \text{ or } G_{2,b}=\left\langle -\operatorname{Id}, 3\cdot \operatorname{Id},\left(\begin{array}{cc} 2 & 1\\ -1 & 2\\\end{array}\right) \right\rangle.$$ 
		
		\item If $[\cC_{-1,0}(2^\infty):G_{E,K,2^\infty}]=4$, then $G_{E,K,2^\infty}$ is one of the following groups:
		$$G_{4,a}=\left\langle 5\cdot \operatorname{Id},\left(\begin{array}{cc} 1 & 2\\ -2 & 1\\\end{array}\right) \right\rangle, \text{ or } G_{4,b}=\left\langle 5\cdot \operatorname{Id},\left(\begin{array}{cc} -1 & -2\\ 2 & -1\\\end{array}\right) \right\rangle, \text{ or }$$ 
		$$ G_{4,c}= \left\langle -3\cdot \operatorname{Id},\left(\begin{array}{cc} 2 & -1\\ 1 & 2\\\end{array}\right) \right\rangle, \text{ or } G_{4,d}=\left\langle -3\cdot \operatorname{Id},\left(\begin{array}{cc} -2 & 1\\ -1 & -2\\\end{array}\right) \right\rangle.$$
\end{itemize}
Moreover, $G_{E,2^\infty}=\langle \gamma,G_{E,K,2^\infty}\rangle = \langle \gamma', G_{E,K,2^\infty}\rangle$ is generated by one of the groups above, and an element 
$$\gamma' \in \left\{ c_1=\left(\begin{array}{cc} 1 & 0\\ 0 & -1\\\end{array}\right),c_{-1}=\left(\begin{array}{cc} -1 & 0\\ 0 & 1\\\end{array}\right),c_1'=\left(\begin{array}{cc} 0 & 1\\ 1 & 0\\\end{array}\right),c_{-1}'=\left(\begin{array}{cc} 0 & -1\\ -1 & 0\\\end{array}\right) \right\},$$ and  
 $\gamma \equiv \gamma' \bmod 4.$
\end{thm}

\begin{proof} Since $j_{K,f}=1728$, it follows that $K=\Q(i)$, $\Delta_K=-4$ and $f=1$, and therefore we set $\delta=-1$ and $\phi=0$. Put $\tau=i$. Fix a basis of $E[2^n]$ such that the image of $\rho_{E,2^n}$ is contained in $\mathcal{N}_{\delta,0}(2^n)$, as in Theorem \ref{thm-cmrep-intro}, and the image of $\Gal\left(\overline{H_f}/H_f\right)$ via $\rho_{E,2^n}$ is contained in $\cC_{\delta,0}(2^n)$. Using $\rho_{E,2^n}$, we identify $G_{E,K,2^n}=\Gal(H_f(E[2^n])/H_f)$ with a subgroup of $\cC_n=(\OK/2^n\OK)^\times$. The index of $G_{E,K,2^n}$ in $\cC_n$ is a divisor of $4=\#\OK^\times$. We distinguish two cases according to the values of the index $[\cC_n:G_{E,K,2^n}]$ for $n\geq 2$:
	
	\begin{enumerate}
		\item If the index $[\cC_n:G_{E,K,2^n}]=d=1,2$, or $4$, is constant for all $n\geq 2$, then it follows that $G_{E,K,2^n}$ is the full inverse image of $G_{E,K,4}$ under the natural map $\cC_n\to \cC_4$. Therefore, $G_{E,K,2^n}$ is one of the groups described in Lemma \ref{lem-2adiccartanj1728}, part (6). However, note that if $[\cC_n:G_{E,K,2^n}]=4$ (resp. $2$), then $G_{E,K,2^n}$ must be missing both $-1$ and $i$ (resp. $\tau=i$), by Cor. \ref{cor-missesrootofunity}. Thus, the images $\langle -1,5,-3+4\tau\rangle$ (resp. $\langle \tau,3,-3+4\tau$) cannot occur. 
		\item Suppose $[\cC_n:G_{E,K,2^n}]$ is not constant for $n\geq 2$. Note that $G_{E,K,2^{n-1}}=\pi(G_{E,K,2^n})$ via $\pi\colon \cC_n \to \cC_{n-1}$. Thus, $[\cC_n:G_{E,K,2^{n-1}}]$ is a divisor of $[\cC_n:G_{E,K,2^{n}}]$. Moreover, Cor. \ref{cor-missesrootofunity} shows that if $[\cC_n:G_{E,K,2^{n}}]=2$ or $4$, then $G_{E,K,2^n}$ is missing $\tau=i$ or $-1$, respectively. We distinguish a few cases:
		\begin{itemize}
			\item If $[\cC_n:G_{E,K,2^n}]=1$ for all $n\geq m$, for some $m\geq 3$, then $[H_f(E[2^m]):H_f(h(E[2^m]))]=4=\# \OK^\times$. Now Theorem \ref{thm-inclp} shows that  $[H_f(E[2^n]):H_f(h(E[2^n]))]=4$ for all $n\geq 2$, and therefore $[\cC_n:G_{E,K,2^n}]=1$ for all $n\geq 2$ (which is case (1)).
			
			\item If $[\cC_n:G_{E,K,2^n}]=2$ for all $n\geq m$, for some $m\geq 3$, then Lemma \ref{lem-2adiccartanj1728}, part (2), shows that $-1 \in G_{E,K,2^m}$ (and therefore $-1\in G_{E,K,2^n}$ for all $n\geq 2$), so $G_{E,K,2^m}$ must be missing $\pm \tau$. Lemma \ref{lem-2adiccartanj1728} part (3) shows that $\tau$ cannot be in $G_{E,K,2^3}$ either. So $G_{E,K,2^3}$ contains $-1$ but not $\tau$. Hence, $[\cC_n:G_{E,K,2^3}]=2$ also. Now, Lemma \ref{lem-2adiccartanj1728} part (5) shows that $[\cC_n:G_{E,K,2^2}]=2$ also, and therefore $[\cC_n:G_{E,K,2^n}]=2$ for all $n\geq 2$ (which is case (1)).
			
			\item If $[\cC_n:G_{E,K,2^n}]=4$ for all $n\geq m$, for some $m\geq 3$, then we must have $[\cC_n:G_{E,K,2^n}]=4$ for all $n\geq 3$. If we assume that $[\cC_n:G_{E,K,2^2}]\neq 4$, then Lemma \ref{lem-2adiccartanj1728}, part (4) shows that $[\cC_2:G_{E,K,2}]=2$, and $H_n=\langle -3, 2-\tau\rangle/2^n$ or $\langle -3,-2+\tau\rangle/2^n$. Hence, $G_{E,K,2^\infty}$ is generated by
			$$\left\langle \left(\begin{array}{cc} -3 & 0\\ 0 & -3\\\end{array}\right),\left(\begin{array}{cc} 2 & -1\\ 1 & 2\\\end{array}\right) \right\rangle, \text{ or } \left\langle \left(\begin{array}{cc} -3 & 0\\ 0 & -3\\\end{array}\right),\left(\begin{array}{cc} -2 & 1\\ -1 & -2\\\end{array}\right) \right\rangle\subseteq \cC_{-1,0}(2^\infty)\subseteq \GL(2,\Z_2).$$
		\end{itemize}
	\end{enumerate}
	
	Finally, by Lemma \ref{lem-ccfinal2}, if we let $c_\phi =\left(\begin{array}{cc} -1  & 0\\ 0 & 1\\\end{array}\right)$. Then, there is a root of unity $\zeta \in \cC_{\delta,\phi}(2^\infty)$ of order dividing $4$ such that $G_{E,p^\infty} = \langle \gamma, G_{E,K,p^\infty}\rangle= \langle \zeta \cdot c_\phi, G_{E,K,p^\infty}\rangle$. Since $\zeta = \pm 1$ or $\pm i$, then $\gamma'=\zeta \cdot c_\phi$ is one of 
	$\left\{ \left(\begin{array}{cc} 1 & 0\\ 0 & -1\\\end{array}\right),\left(\begin{array}{cc} -1 & 0\\ 0 & 1\\\end{array}\right),\left(\begin{array}{cc} 0 & 1\\ 1 & 0\\\end{array}\right),\left(\begin{array}{cc} 0 & -1\\ -1 & 0\\\end{array}\right) \right\}.$
	
	In addition, note that $\gamma=\left(\begin{array}{cc} -a & b\\ b & a\\\end{array}\right) \in \mathcal{N}_{-1,0}(2^\infty)\subseteq  \GL(2,\Z_2)$, for some $a,b\in\Z_2$, is of order $2$, trace $0$, determinant $-1$. Thus, $a^2+b^2=1$, and so $a$ or $b\equiv 0\bmod 4$. Thus, $\gamma\equiv \gamma'\bmod 4$. 
	\end{proof}

\begin{example}\label{ex-j1728-2adic}
	Here we provide examples of elliptic curves $E:y^2=x^3+sx$, defined over $\Q$ and with $j(E)=1728$, with each type of possible $2$-adic image. We use the notation $c_\varepsilon$ and $c_\varepsilon'$, and the labels $G_1,G_{2,a},\ldots$ for subgroups of $\cC_{-1,0}(2^\infty)$, introduced in Lemma \ref{thm-j1728}.	 Below we give pairs of values $(s,G,c)$ such that $E:y^2=x^3+sx$ has the desired image $G_{E,2^\infty}=\langle c, G\rangle$ and the image of a complex conjugation is congruent to $c\bmod 4$.
\begin{itemize} 
	\item If $[\cC_{-1,0}(2^\infty):G_{E,K,2^\infty}]=1$, then one can take $(s,G,c)=(3,G_1,c_1)$ or $(-3,G_1,c_{-1})$. We note, however, that $\langle c_\varepsilon,G_1\rangle =\langle c_{-\varepsilon},G_1\rangle= \langle c_{\varepsilon}',G_1\rangle$ for any $\varepsilon\in \{\pm 1 \}$.
	
	\item If $[\cC_{-1,0}(2^\infty):G_{E,K,2^\infty}]=2$, then one can take $(-9,G_{2,a},c_{-1})$,
	 $(9,G_{2,a},c_{-1}')$, 
	   $(18,G_{2,b},c_{-1}')$, or 
	  $(-18,G_{2,b},c_{-1})$.
	
	\item If $[\cC_{-1,0}(2^\infty):G_{E,K,2^\infty}]=4$, then one can take $(-4,G_{4,a},c_{-1})$, 
	 $(1,G_{4,a},c_{-1}')$, 
	   $(-1,G_{4,b},c_{-1})$, 
	    $(4,G_{4,b},c_{-1}')$, 
	     $(2,G_{4,c},c_{1})$ 
	, $(-2,G_{4,d},c_{-1})$ 
	, $(8,G_{4,d},c_{1})$ 
	, or $(-8,G_{4,c},c_{-1})$.
\end{itemize}
\end{example}

\subsection{$2$-adic images for $j=0$}\label{sec-2adic0}
Here we provide a proof of Theorem \ref{thm-jzero-intro}. We first need a lemma about the possible subgroups of the Cartan of index $2$, $3$, and $6$.

\begin{lemma}\label{lem-2adiccartanjzero}
	Let $n\geq 1$, let $K=\Q(\sqrt{-3})$, let $\Delta_K=-3$, let $f=1$, let $\tau=(1+\sqrt{\Delta_K})/2 = (1+\sqrt{-3})/2$, and put $\delta=-1$, $\phi=1$.  Let $\cC_n = (\OK/2^n\OK)^\times\cong \cC_{\delta,\phi}(2^n)\subseteq \GL(2,\Z/2^n\Z)$, where $\cC_{\delta,\phi}(2^n)$ is defined as in Theorem \ref{thm-cmrep}. Then,
	\begin{enumerate}
		\item $\cC_n=\langle -1, 3+4\tau,1+\tau \rangle/ 2^n \cong \begin{cases}
		\Z/3\Z & \text{ if } n=1,\\
		\Z/2\Z\times \Z/2^{n-2}\Z\times \Z/(2^{n-1}\cdot 3\Z) & \text{ if } n\geq 2.
		\end{cases} $
		\item For all $n\geq 2$, there are no subgroups of index $2$ or $6$ of $\cC_n$ missing $-1$ and stable under complex conjugation.
		\item For all $n\geq 2$, there is a unique subgroup of index $3$ of $\cC_n$, given by $\langle -1, 3+4\tau, -3+6\tau\rangle/2^n$, where $-3+6\tau=(1+\tau)^3$.
	\end{enumerate}
\end{lemma}
\begin{proof}
	Part (1) was shown in Theorem \ref{thm-2adiccartan}.  Thus, for all $n\geq 2$, we have $\cC_n = \langle -1, \alpha_1,\alpha_2\rangle/2^n$ with $\alpha_1$ of order $2^{n-2}$ and $\alpha_2=1+\tau$ of order $2^{n-1}\cdot 3$. Moreover:
	\begin{itemize}
		\item Since $\alpha_1\equiv -1\bmod 4$, there are $2$ possible subgroups mod $4\OK$ that miss the class of $-1$, namely $\langle 1+\tau\rangle$ and $\langle -1-\tau \rangle/2^2$. These subgroups are not stable under complex conjugation (note that $\bar{\tau}=1-\tau$). 
		\item If $n\geq 3$, there are $4$ subgroups of index $2$ that miss $-1$, namely $\langle (-1)^a \alpha_1,(-1)^b\alpha_2\rangle/2^n$, for some $a,b\in\{0,1\}$, and none are stable under complex conjugation.
	\end{itemize}   
Since the $3$-adic valuation of $\#\cC_n=2^{(2n-2)}\cdot 3$ is exactly $1$, if $H_n$ is a subgroup of index $6$ that does not contain $-1$, then $H_n$ must be of the form $\langle (-1)^a \alpha_1,(-1)^b\alpha_2^3\rangle/2^n$, for some $a,b\in\{0,1 \}$, where $\alpha_2^3=(1+\tau)^3=-3+6\tau$. None of these subgroups are stable under complex conjugation. This shows (2).

For (3), since $\#\cC_n=2^{(2n-2)}\cdot 3$, there is a unique subgroup of index $3$, given by $\langle -1, \alpha_1,\alpha_2^3\rangle/2^n$.
\end{proof}

\begin{thm}\label{thm-jzero}
	Let $E/\Q$ be an elliptic curve with $j(E)=0$, and let $c\in \GQ$ be a complex conjugation. Let $G_{E,2^\infty}$ be the image of $\rho_{E,2^\infty}$ and let $G_{E,K,2^\infty}=\rho_{E,2^\infty}(G_{\Q(\sqrt{-3})})$. Then, there is a $\Z_2$-basis of $T_2(E)$ such that the image $G_{E,2\infty}$ of $\rho_{E,2^\infty}$ is one of the following groups of $\GL(2,\Z_2)$, with $\gamma=\rho_{E,2^\infty}(c)$.
\begin{itemize}
	\item Either, $[\cC_{-1,1}(2^\infty):G_{E,K,2^\infty}]=3$, and 
	\begin{align*} G_{E,2^\infty} &=\left\langle \gamma', -\operatorname{Id}, \left(\begin{array}{cc} 7 & 4\\ -4 & 3\\\end{array}\right), \left(\begin{array}{cc} 3 & 6\\ -6 & -3\\\end{array}\right)\right\rangle\\
		&=\left\langle \gamma', \left\{\left(\begin{array}{cc} a+b & b\\ -b & a\\\end{array}\right)\in \GL(2,\Z_2) : a\not\equiv 0 \bmod 2,\ b\equiv 0 \bmod 2 \right\}\right\rangle,\end{align*}  
	Moreover, $\left\{\left(\begin{array}{cc} a+b & b\\ -b & a\\\end{array}\right): b\equiv 0 \bmod 2 \right\}$ is precisely the set of matrices that correspond to the subgroup of cubes of Cartan elements $\cC_{-1,1}(2^\infty)^3$, which is the unique group of index $3$ in $\cC_{-1,1}(2^\infty)$, 
	\item Or, $[\cC_{-1,1}(2^\infty):G_{E,K,2^\infty}]=1$, and 
	\begin{align*} G_{E,2^\infty} &=\mathcal{N}_{-1,1}(2^\infty)=\left\langle \gamma', -\operatorname{Id}, \left(\begin{array}{cc} 7 & 4\\ -4 & 3\\\end{array}\right), \left(\begin{array}{cc} 2 & 1\\ -1 & 1\\\end{array}\right)\right\rangle\\
	&=\left\langle \gamma', \left\{\left(\begin{array}{cc} a+b & b\\ -b & a\\\end{array}\right)\in \GL(2,\Z_2) : a\not\equiv 0 \text{ or } b\not\equiv 0 \bmod 2 \right\}\right\rangle\end{align*}
\end{itemize} 
where $\gamma'\in \left\{ \left(\begin{array}{cc} 0 & 1\\ 1 & 0\\\end{array}\right),\left(\begin{array}{cc} 0 & -1\\ -1 & 0\\\end{array}\right)\right\},$ and $\gamma\equiv \gamma'\bmod 4$.
\end{thm}
\begin{proof}
	
	Since $j_{K,f}=0$, it follows that $K=\Q(\sqrt{-3})$, $\Delta_K=-3$, and $f=1$, and therefore we set $\delta=-1$ and $\phi=1$. Put $\tau=i$. Fix a basis of $E[2^n]$ such that the image of $\rho_{E,2^n}$ is contained in $\mathcal{N}_{-1,1}(2^n)$, as in Theorem \ref{thm-cmrep-intro}, and the image of $\Gal\left(\overline{H_f}/H_f\right)$ via $\rho_{E,2^n}$ is contained in $\cC_{-1,1}(2^n)$. Using $\rho_{E,2^n}$, we identify $G_{E,K,2^n}=\Gal(H_f(E[2^n])/H_f)$ with a subgroup of $\cC_n=(\OK/2^n\OK)^\times$. The index of $G_{E,K,2^n}$ in $\cC_n$ is a divisor of $6=\#\OK^\times$. We distinguish three cases according to the values of the index $[\cC_n:G_{E,K,2^n}]$ for $n\geq 2$:
	
	\begin{itemize}
		\item If $[\cC_{-1,1}:G_{E,K,2^\infty}]=1$, then $G_{E,K,2^\infty}=\cC_{-1,1}(2^\infty)$ and $G_{E,2^\infty}=\mathcal{N}_{-1,1}(2^\infty)$.
		\item If $[\cC_{-1,1}:G_{E,K,2^\infty}]=2$ (resp. $6$), then, by Cor. \ref{cor-missesrootofunity}, the group $G_{E,K,2^\infty}$ is a subgroup of index $2$ (resp. $6$) of $\cC_{-1,1}(2^\infty)$ missing $-1$ and stable under complex conjugation. However, Lemma \ref{lem-2adiccartanjzero} shows that this is impossible.
		\item If $[\cC_{-1,1}:G_{E,K,2^\infty}]=3$, then $G_{E,K,2^\infty}$ is a subgroup of index $3$, and Lemma \ref{lem-2adiccartanjzero} shows that there is exactly one such subgroup, namely the one described in the statement of the theorem.
	\end{itemize}
Let $c$ be a complex conjugation, and let $\gamma=\rho_{E,2^\infty}(c)$. By Lemma \ref{lem-ccfinal2}, there is a root of unity $\zeta$ (of order dividing $6$ in this case), such that $G_{E,2^\infty}=\langle \gamma, G_{E,K,2^\infty}\rangle = \langle \zeta\cdot c_\phi, G_{E,K,2^\infty}\rangle$, where $c_\phi=\left(\begin{array}{cc} -1 & 0\\ 1 & 1\\\end{array}\right)$. Thus, the statement holds with $\gamma' = \zeta c_\phi$ and, moreover, the options for $\zeta c_\phi$ as $\zeta$ varies are $ c_{1,1},c_{1,-1},c_{2,1},c_{2,-1},c_{3,1},c_{3,-1}$ defined, respectively, by
$$\left(\begin{array}{cc} -1 & 0\\ 1 & 1\\\end{array}\right),\left(\begin{array}{cc} 1 & 0\\ -1 & -1\\\end{array}\right),\left(\begin{array}{cc} 0 & 1\\ 1 & 0\\\end{array}\right),\left(\begin{array}{cc} 0 & -1\\ -1 & 0\\\end{array}\right),\left(\begin{array}{cc} 1 & 1\\ 0 & -1\\\end{array}\right),\left(\begin{array}{cc} -1 & -1\\ 0 & 1\\\end{array}\right).$$ 

Now, if $G_{E,K,2^\infty}=\cC_{-1,1}(2^\infty)$, then $\mathcal{N}_{-1,1}(2^\infty)=\langle c_{i,\varepsilon},\cC_{-1,1}(2^\infty)\rangle$ for any $1\leq i\leq 3$, and any $\varepsilon \in \{\pm 1 \}$. If $G_{E,K,2^\infty}=\cC_{-1,1}(2^\infty)^3$, however, then $G_{E,2^\infty}=\langle c_{i,1},\cC_{-1,1}(2^\infty)^3 \rangle$ for some fixed $1\leq i \leq 3$, then $c_{i,-1}\in G_{E,2^\infty}$ but $c_{j,\varepsilon}$ is not in the image for any $j\neq i$ and any $\varepsilon = \pm 1$. Thus, in both cases to determine $\gamma'$ such that $G_{E,2^\infty}=\langle \gamma',G_{E,K,2^\infty}\rangle$, it suffices to find $i$ such that $c_{i,1}\in G_{E,2^2}$ (since all the $c_{i,\varepsilon}$ reduce to distinct elements mod $4$). For this, we work in coordinates. 

Let us first consider a function field in one variable $\Q(s)$ and let $E/\Q(s)$ be given by $y^2=x^3+s$. Then, the $2$-nd and $4$-th division polynomials of $E$ are given by $$\psi_2(x)= x^3+s, \ \text{ and } \ \psi_4(x)=(x^3+s)(x^6+20sx^3-8s^2),$$
so that $\psi_4(x)/\psi_2(x)=x^6+20sx^3-8s^2$, with roots satisfying $x^3=(-10\pm 6\sqrt{3})s$. In particular, if $\zeta=\zeta_{12}$ denotes a primitive $12$-th root of unity, then the points 
$$P=((\zeta ^3 + \zeta^2  + \zeta)t^2 , -\alpha\cdot t^3),\ \text{ and }\ Q=((-2\zeta ^3-\zeta^2 + \zeta  + 1)t^2 , \alpha\cdot t^3)$$
generate $E[4]$, where $t=\sqrt[6]{s}$, and $\alpha=\sqrt{-9+6\sqrt{3}}$. In particular, if $c$ is a complex conjugation, then one can show that, with respect to the basis $\{P,Q \}$ of $E[4]$, we have
\begin{itemize} 
	\item  $\rho_{E,2^2}(c)\equiv \gamma \equiv  \left(\begin{array}{cc} 0 & 1\\ 1 & 0\\\end{array}\right)=c_{2,1}$ if $s<0$, and $\gamma\equiv c_{2,-1} \bmod 4$ if $s>0$, and
	\item  $G_{E,2^2} = \langle \gamma, \cC_{-1,1}^3\rangle$ when $s\in (\Q^\times)^3$, and $G_{E,2^2} = \langle \gamma, \cC_{-1,1}\rangle$ otherwise. 
	\end{itemize} 
Hence, in all cases, $\gamma \equiv c_{2,\varepsilon} \bmod 4$, for some $\varepsilon \in \{\pm 1 \}$, and, in particular, $c_{2,\varepsilon}\in G_{E,2^\infty}$. 
\end{proof} 	

\begin{example}\label{ex-j0-2adic}
	Here we provide examples of elliptic curves $E:y^2=x^3+s$, defined over $\Q$ and with $j(E)=0$, with each type of possible $2$-adic image. We use the notation $c_\varepsilon'=\left(\begin{array}{cc} 0 & \varepsilon\\ \varepsilon & 0\\\end{array}\right)$.	 Below we give pairs of values $(s,G,c)$ such that $E:y^2=x^3+s$ has the desired image $G_{E,2^\infty}=\langle c, G\rangle$, with $G=\cC_{-1,1}(2^\infty)$ or $\cC_{-1,1}(2^\infty)^3$ and the image of a complex conjugation is congruent to $c\bmod 4$.
	\begin{itemize} 
		\item If $[\cC_{-1,1}(2^\infty):G_{E,K,2^\infty}]=1$, then one can take, for example, $(s,G,c)=(2,\cC_{-1,1}(2^\infty),c_{-1}')$ or $(-2,\cC_{-1,1}(2^\infty),c_{1}')$. 
		
		\item If $[\cC_{-1,1}(2^\infty):G_{E,K,2^\infty}]=3$, then one can take $(1,\cC_{-1,1}(2^\infty)^3,c_{-1}')$, or $(-1,\cC_{-1,1}(2^\infty)^3,c_{1}')$.
	\end{itemize}
\end{example}

\bibliography{bibliography}
\bibliographystyle{plain}

\end{document}